\documentclass[a4paper,12pt,reqno]{amsart}
\usepackage{hyperref}
\hypersetup{backref,colorlinks=true,allcolors=black}
\usepackage{mathrsfs}
\usepackage{enumitem}
\usepackage{dsfont}
\usepackage{mathtools}
\usepackage{xcolor}
\usepackage[final]{graphicx}
\usepackage{upgreek}
\usepackage{amsmath,amssymb}
\usepackage{nameref}
\usepackage{comment} 

\usepackage{soul}

\newtheorem{theorem}{Theorem}[section]
\newtheorem{proposition}[theorem]{Proposition}
\newtheorem{lemma}[theorem]{Lemma}
\newtheorem{corollary}[theorem]{Corollary}

\theoremstyle{definition}
\newtheorem{definition}[theorem]{Definition}
\newtheorem{example}[theorem]{Example}
\newtheorem{remark}[theorem]{Remark}

\newtheorem{innercustomthm}{Assumption}
\newenvironment{customassumption}[1]
{\renewcommand\theinnercustomthm{#1}\innercustomthm}
{\endinnercustomthm}

\newtheorem{innercustomgeneric}{\customgenericname}
\providecommand{\customgenericname}{}
\newcommand{\newcustomtheorem}[2]{%
	\newenvironment{#1}[1]
	{%
		\renewcommand\customgenericname{#2}%
		\renewcommand\theinnercustomgeneric{##1}%
		\innercustomgeneric
	}
	{\endinnercustomgeneric}
}

\newcustomtheorem{customthm}{Theorem}
\newcustomtheorem{customlemma}{Lemma}
\newcustomtheorem{customassump}{Assumption}
\newcustomtheorem{customproof}{Proof}

\definecolor{crimson}{rgb}{0.7294,0.0666,0.0470}

\newcommand{\ignore}[1]{}
\newcommand{\R}{\mathbb{R}}
\newcommand{\N}{\mathbb{N}}

\newcommand{\norm}[1]{\left\lVert #1 \right\rVert}
\newcommand{\abs}[1]{\left\vert #1 \right\rvert}
\newcommand{\E}[1]{\mathbb{E}{\left[ #1\right]}}
\newcommand{\EE}{\mathbb{E}}

\DeclarePairedDelimiter\autobracket{(}{)}
\newcommand{\brac}[1]{\autobracket*{#1}}
\newcommand{\inner}[1]{\left\langle #1 \right\rangle}

\numberwithin{equation}{section}

\theoremstyle{definition}

\makeatletter
\@addtoreset{proofcase}{theorem}
\makeatother

\allowdisplaybreaks

\theoremstyle{definition}

\makeatletter
\@addtoreset{proofstep}{theorem}
\makeatother

\usepackage{microtype}
\usepackage{parskip}

\definecolor{crimson}{rgb}{0.7294,0.0666,0.0470}
\begin{document}
	\bibliographystyle{amsalpha}

	\author{Thanh Dang$^1$}
	\address{$^1$Department of Computer Science, University of Rochester, Rochester, NY 14620 }
	\email{ycloud777@gmail.com}
	\author{Yaozhong Hu$^2$}\thanks{Y. Hu is supported by   NSERC discovery grant     RGPIN
2024-05941    and  centennial  fund from University of Alberta.} 
	\address{$^2$Department of Mathematical and Statistical Sciences, University of Alberta, Edmonton, AB T6G 2G1, Canada}
	\email{yaozhong@ualberta.ca}
	\title[]{Density convergence on Markov diffusion chaos via Stein's method}
	\begin{abstract}  We study the difference between the probability density of a random variable $F$ on Markov diffusion chaos and the probability density of a general target distribution $\mathcal{Z}$. In the special case where $F$ is a \textit{chaotic} random variables and $\mathcal{Z}$ is a Pearson target, we extend our study to the $k$-th derivatives of the densities for all $k\in \N$. In particular, we obtain four moment theorems for the convergence of the $k$-th derivatives of the densities of $F$ to the corresponding $k$-th derivatives of the density of a Pearson target. Our work therefore significantly extends earlier works \cite{HLN14,BDH24} which studies density convergence of random variables on Wiener chaos to respectively the normal and Gamma targets.

    We provide two applications of our results. The first application is about weighted sum of i.i.d. Gamma distribution where we show convergence in laws of this weighted sum to another Gamma distribution automatically implies convergence in densities. In the second application, we show that for a large class of Pearson diffusions,  the density of its   solution with any initial condition exponentially converge to its limiting density. Moreover, this exponential convergence holds for the $k$-th derivatives of the densities for all $k\in \N$. 
		\end{abstract}
		
	\maketitle
	


		\section{Introduction}

In \cite{nualartpeccati2005central}, Nualart and Peccati provides an exhaustive characterization of normal convergence inside a fixed Wiener chaos. In particular, they observe that for a sequence of random variables $F_n$ in a fixed Wiener chaos, convergence in laws to a normal distribution $N$ is equivalent to convergence of the first four moments of $F_n$ to the first four moments of $N$. This can be viewed as a significant simplification of the classical method of moments and is later known as the four moment phenomenon. Peccati and Tudor in \cite{peccatiandtudor2005gaussian} then extend  this result  to  a multidimensional setting. The main tools  they used are  the Malliavin calculus and chaos expansion. 
Motivated by  these works     Nourdin and Peccati (\cite{NP09main})  
combined   Malliavin calculus and Stein's method   to derive quantitative convergence rate. This turns out to be widely successful and kick-start the early development of the Malliavin-Stein method, a combination of Malliavin calculus and Stein's method to study quantitative approximation of central and non-central targets.


Later on, Ledoux in \cite{Led12} makes a beautiful discovery that the four moment theorem on Wiener chaos can be replicated in the more general setting of Markov diffusion chaos. The authors of \cite{ACP14,bourguintaqqu2019} extend Ledoux's idea and come up with a novel property called the \textit{chaos grades}, which allow them to derive moment theorems that target the family of Pearson distributions. 

The majority of four moment theorems in literature are about convergence in some probability distance and/or convergence in laws. As a break from this trend, \cite{HLN14} is the first reference to offer a four moment theorem for density convergence of random variables on Wiener chaos to a normal target. It is followed by \cite{BDH24} which offers a four moment theorem for density convergence on Wiener chaos to a Gamma target. As continuation of these works, we will consider convergence in density to more general distributions for  random variables in the more general setting of Markov diffusion chaos. In Section \ref{section_generaldiffusion}, we provide sufficient conditions for the densities of these random variables to converge to a general density that satisfies the first item in Assumption \ref{assum_generaldiffusion}
of Section 4. In Section \ref{section_pearson}, we consider a narrower setting that is density convergence of eigenfunctions of some Markov diffusion generator to a Pearson target. This narrower setting allows us to derive sufficient conditions for not just convergence in densities, but also convergence of derivatives of the densities. The highlight of the theoretical  part of our paper are  four moment density  discrepancy   theorems for the family of Pearson distributions, to be presented at the end of Section \ref{section_pearson}.

Let us first describe our methodology by providing briefly the argument to obtain a four moment   theorem   for density convergence to a normal target.  Assume a full Markov triple $\brac{E,\mathcal{F},\mu}$ equipped with a diffusion generator $L$ and carr\'{e} du champ operator $\Gamma$, to be introduced in Section \ref{section_prelim}. $F$ is an eigenfunction corresponding to the eigenvalue  $ -q$ of $L$ and $\mathcal{N}$ denotes the standard normal distribution. 
Our starting point is the density formula in Theorem \ref{theorem_herrydensity} which applies to smooth random variables on Markov diffusion chaos and says
\begin{align*}
    p_F(x)=\E{\mathds{1}_{\{F>x\}}\brac{\frac{qF}{\Gamma(F)}+\frac{\Gamma(F,\Gamma(F))}{\Gamma(F)^2} }}, \qquad p_{\mathcal{N}}(x)=\E{\mathds{1}_{\{\mathcal{N}>x\}}\mathcal{N} }.
\end{align*}
 This allows us to write
\begin{align}
\label{intro_comparepFtopN}
    p_F(x)-p_{\mathcal{N}}(x)= \brac{\E{\mathds{1}_{\{F>x\}}F }-\E{\mathds{1}_{\{\mathcal{N}>x\}}\mathcal{N} }}+\mathcal{R}
\end{align}
where the remainder term $ \mathcal{R}=\E{\mathds{1}_{\{F>x\}} A(\Gamma(F)-q) }+\E{\mathds{1}_{\{F>x\}} B\Gamma(F,\Gamma(F)-q)}$. Here $A$ and $B$ denote some random variables satisfying $\E{A^4},\E{B^4}<\infty$. The first term on the right hand side of \eqref{intro_comparepFtopN} can be uniformly bounded via Stein's method as
  \begin{align*}
        \sup_{x\in \R}\abs{\E{\mathds{1}_{\{F>x\}}F }-\E{\mathds{1}_{\{\mathcal{N}>x\}}\mathcal{N} } }&\leq  C\brac{\E{F^{2}}^{1/2}+1 } \E{\brac{\Gamma(F)-q }^2}^{1/2},
        \end{align*}
where $C$ is some positive constant. Regarding the second term in $\mathcal{R}$ which contains the iterated gradient $\Gamma(F,\Gamma(F)-q)$, the fact that $\Gamma$ is a positive operator on its domain means $\Gamma(F,G)^2\leq \Gamma(F)\Gamma(G) $, so that
\begin{align*}
   \sup_{x\in \R} \abs{\E{\mathds{1}_{\{F>x\}} B\Gamma(F,\Gamma(F)-q)}}\leq \E{\abs{B}^{3/2}}^{1/2}\E{\brac{L\Gamma(F) }^2}^{1/4} \E{\brac{\Gamma(F)-q}^2}^{1/4}. 
\end{align*}
This particular argument is a novelty compared to \cite{HLN14}, whose Lemma 4.2 bounds an analogous term with the product formula on Wiener chaos. Combining the previous calculations give us
\begin{align*}
\sup_{x\in \R}\abs{p_F(x)-p_{\mathcal{N}}(x) }&\leq  C\brac{\E{F^{2}}^{1/2}+\E{A^2}^{1/2}+1 } \E{\brac{\Gamma(F)-q }^2}^{1/2} \\
&\qquad\qquad\qquad\qquad+\E{\abs{B}^{3/2}}^{1/2}\E{\brac{L\Gamma(F) }^2}^{1/4} \E{\brac{\Gamma(F)-q}^2}^{1/4}. 
\end{align*}
In particular, the price to pay in order to extend the results of \cite{HLN14} from the setting of Wiener chaos to the setting of Markov diffusion chaos is that the above estimate produces  a slower rate of convergence,  namely a rate of $1/4$ versus a rate  of $1/2$ on Wiener chaos (see \cite[Proof of Theorem 4.1]{HLN14}). Finally,  by making an additional assumption that $F$ is an eigenfunction of $L$ in the \emph{chaotic} sense as introduced in \cite{ACP14, bourguintaqqu2019}, we can further bound $\E{\brac{\Gamma(F)-q}^2}$ by a combination of the first four moment of $F$, which yields a density four moment theorem for the normal target.

If the above normal target distribution is replaced by some other distribution, things becomes more complicated. We refer to \cite{BDH24} for the treatment the density discrepancy when the probability space is the Wiener space and the target distribution is  a Gamma distribution.
This paper is along the line of \cite{BDH24} to study   the density discrepancy.  The target 
distributions are extended from  the family of  Gamma distributions to   general distributions that admit densities. As the first step of the paper,  we need to find an appropriate 
 representation formulas of the density  and their derivatives. An example of such formulas is $p_{\mathcal{N}}(x)=\E{\mathds{1}_{\{\mathcal{N}>x\}}\mathcal{N} }$ for density of the standard normal distribution $N$. We are able to obtain analogous formulas for very general probability density functions  and their derivatives in Section \ref{section_somedensity}.  As the second step, we need to set up an appropriate Stein's equation. This has been a complicated task both in \cite{HLN14}  for a normal distribution, and in \cite{BDH24} for a Gamma distribution.  It is even more complicated here in our situation and is done in Sections \ref{section_generaldiffusion} and \ref{section_pearson}. While Section \ref{section_generaldiffusion} is about general density approximation, Section \ref{section_pearson} focuses on approximation of densities in the Pearson family of distribution. The additional structure of the latter allows us to study convergence of the densities as well as the derivatives of the densities. Specifically, Theorem \ref{theorem_fourmoment} is a four moment theorem for the convergence of the $k$-th derivatives of the densities of \textit{chaotic} eigenfunctions to the corresponding $k$-th derivatives of the density of a Pearson target.  Furthermore, we show in Corollary \ref {cor_fourmomentwienerchaos} that two main results in \cite{HLN14} and \cite{BDH24} are special cases of Theorem \ref{theorem_fourmoment}.

There have been some nice applications of Stein's method to obtain the density convergence  to normal targets in 
\cite{hu15,HLN14,sun2025density,kuzgun2024convergencedensity} among others.  Compared to
proving convergence in some probability distance, proving convergence 
in density is much harder since typically, the existence of density is guaranteed by the 
the existence of certain negative moments,  which is usually  an extremely  difficult task.   
 To provide an idea how to proceed with the verification of the conditions about negative moments presented in this work, we give two examples. The first example is about weighted sums 
of i.i.d. Gamma-distributed random variables. We show in Theorem \ref{theorem_superconvergencetogamma} that if the weighted sums converge in distribution to a Gamma distribution with density $\frac{1}{\Gamma(\gamma)}x^{\gamma-1}e^{-x}\mathds{1}_{ \{x> 0\} }$ and $\alpha>6$, then the convergence also holds for the densities. This application is partly motivated by the \textit{superconvergence} phenomenon recently discovered in \cite{polyherry2023centralsuperconvergence,herry2023regularity}. 

On the other hand, there have been an enormous number of works on the convergence  of  solutions of stochastic differential equations to its stationary distribution in  the  framework of   
stochastic stability theory (see e.g. \cite{khasminskii2012stochastic}, \cite{meyn1992stabilitypart1} and the references therein). 
In fact, this is a fundamental question of random dynamical systems.  Up to present,  most results in literature are about convergence in  probability laws.  It is a very natural question to ask if such convergences also hold in   terms of densities. 
Our second example in this work is about the  convergence of the probability density $p_{X_t}$ of the solution $X_t$ of a one-dimensional  stochastic 
differential equation  to   its limiting
probability density. In the case of a large class of Pearson diffusion, we show that exponential convergence also holds for the $k$-th derivatives of $p_{X_t}$ for any $k\in \N$. The idea of our proof is to combine Malliavin calculus, Stein method and classical stability results via Lyapunov functions.   We hope that in the future, our results can be extended to 
more  general stochastic differential equations (especially for multidimensional SDEs) and improve earlier results from  convergence in distribution to convergence in density. 

The organization  of this paper is as follows.   Sections \ref{section_prelim} and \ref{section_somedensity} cover the preliminaries on Markov diffusion generators and the crucial density representation formulas that are the starting point of our paper. Section \ref{section_generaldiffusion} considers density convergence of random variables on Markov diffusion chaos to a general density, while Section \ref{section_pearson} study convergence of the $k$-th derivatives of the densities for $k\in \N$ of eigenfunctions to a Pearson target. Section \ref{section_superconvergencetogamma} contains our first example
on the  density convergence of  weighted sum of i.i.d. Gamma random variables.  Section \ref{section_expconvergence} is about the convergence
rate of the density of the solution to a one-dimensional 
stochastic differential equation to a Pearson density. Finally, Appendix \ref{appendix_malliavin} covers some key components of Malliavin calculus which we will make use of in a result in Section \ref{section_pearson}.


\section{Preliminaries on Markov diffusion generators}
      \label{section_prelim}
 
Informally, the framework we shall employ consists of some underlying reversible diffusive Markov process $\left\{
  X_t\colon t\geq 0 \right\}$ with
an invariant probability measure $\nu$, associated  with a semigroup  $\left\{ P_t\colon t\geq
0\right\}$, infinitesimal
generator $L$ and carr\'e  du champ operator  $\Gamma$, where all of these objects
are inherently connected. From an abstract
point of view, a standard and elegant way to introduce this setting is through the   so
called Markov triples, where one starts from the invariant measure $\nu$,  the
carr\'e  du champ $\Gamma$ and a suitable algebra of functions (random
variables), from which the generator $L$, the semigroup $\left\{ P_t\colon t\geq
0\right\}$ (including
their $L^2$-domains) and thus also
the Markov process $\left\{ X_t\colon t\geq
0\right\}$ are constructed.

The framework consists of some underlying reversible diffusive Markov process $\left\{
  X_t\colon t\geq 0 \right\}$ with
an invariant probability measure $\nu$, associated semigroups $\left\{ P_t\colon t\geq
0\right\}$, infinitesimal
generator $L$ and carré du champ $\Gamma$, where all of these objects
are inherently connected. From an abstract
point of view, a standard and elegant way to introduce this setting is through so
called Markov triples, where one starts from the invariant measure $\nu$, the
carré du champ $\Gamma$ and a suitable algebra of functions (random
variables), from which the generator $L$, the semigroup $\left\{ P_t\colon t\geq
0\right\}$ (including
their $L^2$-domains) and thus also
the Markov process $\left\{ X_t\colon t\geq
0\right\}$ are constructed.

The assumptions we   make on $(E,\nu,\Gamma)$  here are those of the  so-called 
\emph{diffusion properties}
 in the sense of \cite[Part I,
Chapter 3]{bakry2014analysis}, which we summarize below together with some useful consequences of the assumptions. 
\begin{enumerate}[label=\roman*)]
\item $(E,\mathcal{F},\nu)$ is a probability space and $L^2(E,\mathcal{F},\nu)$ is separable.
\item $\mathcal{A}$ is a vector space of real-valued random variables on $(E,\mathcal{F},\nu)$. It is stable under products and under the action of
 smooth functions $\Psi \colon \R^k \to \R$.
\item $\mathcal{A}_0$ as a sub-algebra of $\mathcal{A}$ contains bounded functions which are dense in $L^p(E,\nu)$ for every $p \in [1,\infty)$. We assume
  that $\mathcal{A}_0$ is also stable under the action of smooth functions $\Psi$
  as above and also that $\mathcal{A}_0$ is an ideal in $A$ (if $F\in
  \mathcal{A}_0$ and $G \in \mathcal{A}$, then $FG \in \mathcal{A}_0$). 
\item The carr\'{e} du champ operator $\Gamma \colon \mathcal{A}_0 \times \mathcal{A}_0 \to
  \mathcal{A}_0$ is a bi-linear symmetric map such that $\Gamma(F) \geq 0$ for all $F \in
  \mathcal{A}_0$. For every $F \in \mathcal{A}_0$ there exists a finite constant $c_F$ such
  that for every $G \in \mathcal{A}_{0}$
  \begin{equation*}
 \abs{\E{\Gamma(F,G)}}    \leq c_F \norm{G}_2,
  \end{equation*}
  where $\norm{G}_2^2 = \int_E G^2 d\nu$.
  The Dirichlet form $\mathcal{E}$ is defined on $\mathcal{A}_0 \times \mathcal{A}_0$ by
  \begin{equation*}
    \mathcal{E}(F,G) = \E{\Gamma(F,G)}.
  \end{equation*}
  \item
  The domain $\operatorname{Dom}(\mathcal{E}) \subseteq L^2(E,\nu)$ is obtained by completing
  $\mathcal{A}_0$ with respect to the norm $\norm{F}_{\mathcal{E}} =  
    \norm{F}_2 +\left( \mathcal{E}(F,F) \right)^{1/2}$. The Dirichlet form 
  $\mathcal{E}$ and the carr\'{e} du champ operator $\Gamma$ are extended to
  $\operatorname{Dom}(\mathcal{E}) \times \operatorname{Dom}(\mathcal{E})$ by continuity and   polarization. 
  We thus have that $\Gamma \colon \operatorname{Dom}(\mathcal{E}) \times
  \operatorname{Dom}(\mathcal{E}) \to L^1(E,\nu)$.
\item $L$ is a self-adjoint linear operator,   defined on $\mathcal{A}_0$
  via the \emph{integration by parts formula}
  \begin{equation}
    \label{intbyparts_gammaandL}
  \E{GLF}=\E{FLG}  = -\E{\Gamma(F,G)}
  \end{equation}
  for all $F,G \in \mathcal{A}_0$. 
  We assume that $L(\mathcal{A}_0) \subseteq
  \mathcal{A}_0$ and that for every $F\in \mathcal{A}_0$, $ \E{LF}=0$. Moreover, the first part of the above equation is known as the reversible property of $L$ and is related to reversibility
in time of the associated Markov process whenever the initial law is the invariant measure $\nu$.

\item The domain $\operatorname{Dom}(L) \subseteq \operatorname{Dom}(\mathcal{E})$ 
  consists of all $F \in \operatorname{Dom}(\mathcal{E})$ such that
  \begin{equation*}
    \abs{\mathcal{E}(F,G)} \leq c_F \norm{G}_2
  \end{equation*}
  for all $G \in \operatorname{Dom}(\mathcal{E})$, where $c_F$ is a finite constant. The
  operator $L$ is extended from $\mathcal{A}_0$ to $\operatorname{Dom}(L)$ by the
  integration by parts formula~\eqref{intbyparts_gammaandL}. The norm in this domain is defined as 
  \begin{align*}
      \norm{F}_{\operatorname{Dom}(L)}=\brac{\norm{F}_2^2+\norm{LF}_2^2}^{1/2}.
  \end{align*}

  Moreover, we assume that $L 1 = 0$ and
  that $L$ is ergodic: $LF =0$ implies that $F$ is constant for all $F \in \operatorname{Dom}(L)$.
\item The operator $L \colon \mathcal{A} \to \mathcal{A}$ is an extension of $L
  \colon \mathcal{A}_0 \to \mathcal{A}_0$. On $\mathcal{A} \times \mathcal{A}$, the
  carr\'{e} du champ $\Gamma$ is defined by
  \begin{equation*}
    \label{def_gammabyL}
    \Gamma(F,G) := \frac{1}{2} \left( L(FG) - FLG - GLF \right).
  \end{equation*}

\item For all $F \in \mathcal{A}$, we assume $\Gamma(F) \geq 0$ with equality if, and only
  if, $F$ is constant. This yields the Cauchy-Schwarz inequality (\cite[(1.4.3)]{bakry2014analysis})
  \begin{align}
  \label{gammaispositive}
      \Gamma(F,G)^2\leq \Gamma(F)\Gamma(G). 
  \end{align}

\item The \emph{diffusion property} holds. For all $\mathcal{C}^{\infty}$-functions
  $\Psi \colon \R^p \to \R$ and $F_1,\dots,F_p,G \in \mathcal{A}$ one has
  \begin{equation}
    \label{diffusionproperty_gamma}
    \Gamma \left( \Psi \left( F_1,\dots,F_p \right),G \right)
    =
    \sum_{j=1}^p \partial_j \Psi(F_1,\dots,F_p) \, \Gamma(F_j,G)
  \end{equation}
  and
  \begin{equation*}
    \label{eq:30}
    L \Psi(F_1,\dots,F_p) = \sum_{i=1}^p \partial_i \Psi(F_1,\dots,F_p) L F_{i}
    +
    \sum_{i,j=1}^p \partial_{ij} \Psi(F_1,\dots,F_p) \, \Gamma(F_i,F_j).
  \end{equation*}
\item The integration by parts formula~\eqref{intbyparts_gammaandL} continues to hold if $F \in
  \mathcal{A}$ and $G \in \mathcal{A}_0$ (or vice versa).
  \item The eigen-functions of $L$ in $\operatorname{Dom}(L)$ form the eigen-subspaces of $L$ known as Markov diffusion chaos. 
\end{enumerate}

For an $F\in \mathrm{Range}(L)$ with $\mathbb{E}(F)=0$,  
we define $L^{-1}F=G$ if $F=LG$.  This is
well-defined since if $LG_1=LG_2$,   or $L(G_1-G_2)=0$,
then $\mathbb{E}(G_2-G_1)=0$ by our ergodic assumption on $L$ 
(part (vii)).   In general, for any $F\in \mathrm{Range}(L)$,
we define $L^{-1} F=L^{-1} (F-\mathbb{E}(F))$.  
It is easy to verify  that 
  the  inverse $L^{-1}$   satisfies that  for any $F\in \mathrm{Range} (L)$  
  \begin{align}
      \label{relation_Landinverse}
      LL^{-1}F=L^{-1}LF=F-\E{F}\,. 
  \end{align}


We provide here three examples of full Markov triples associated with orthogonal polynomials. 
\begin{example}
    Our first  example is the full Markov triple associated with the one-dimensional Ornstein-Uhlenbeck process: $E=\R$, $\mu$ is the standard Gaussian distribution and $\Gamma(f)=(f')^2$. The Ornstein-Uhlenbeck generator is given by $Lf=xf'+f''$ and the  suitable algebra $\mathcal{A}$ is the set of all polynomials. 
The eigenfunctions of $L$ are the orthogonal Hermite polynomials $\{H_n:n\geq 1 \}$ which satisfy the recursive relation
    \begin{align*}
        H_1(x):=1, \qquad  H_{n+1}(x):=xH_n(x)-H_n'(x). 
    \end{align*}
Moreover per \cite[Section 2.7.2]{bakry2014analysis}, this can be extended the infinite-dimensional setting where $E$ is a Banach space and $\mu$ is a Gaussian measure on $E$. 
\end{example}

\begin{example}
    \label{example_jacobistructure}
    Our next example is the full Markov triple associated with the one-dimensional Jacobi process: $E=\R$, $\mu$ is a beta distribution with density $  \frac{1}{\operatorname{B}(\alpha,\beta)}x^{\alpha-1}(1-x)^{\beta-1}\mathds{1}_{(0,1)}(x);\alpha,\beta>0 $ and $\Gamma(f)=x(1-x)(f')^2$. The Jacobi operator is given by $Lf=x(1-x)f''+\brac{\alpha-(\alpha+\beta)x}f'$ and $\mathcal{A}$ is again the set of polynomials. Eigenfunctions of $L$ belong to the orthogonal family of Jacobi polynomials $\{P_n:n\geq 1 \}$ defined by
    \begin{align*}
        P_n(x)=\frac{(-1)^n}{2^n n!}(1-x)^{-\alpha}(1+x)^\beta\frac{d^n}{dx^n}\brac{(1-x)^\alpha (1+x)^\beta(1-x^2)^n}. 
    \end{align*}
\end{example}

\begin{example}
\label{example_laguerre}
    The last example is the full Markov triple associated with the one-dimensional Laguerre process (\cite[Section 2.7.3]{bakry2014analysis}): $E=\R$, $\nu$ is a Gamma distribution with density $\pi(x)=\frac{1}{\Gamma(\alpha)}x^{\alpha-1}e^{-x},\alpha>0$ and $\Gamma(f)(x)=x(f'(x))^2$, with $x$ distributed as $\nu$. The Laguerre generator is given by $Lf(x)=xf''(x)+(\alpha-x)f'(x)$ and a suitable algebra $\mathcal{A}$ is the set of polynomials. Moreover, regarding the Laguerre semigroups, one can find a Mehler-type formula for them in \cite{arras2017stroll}. Finally, the eigen-functions of $L$ are the Laguerre polynomials $\{Q_n:n\geq 1 \}$ defined by
    \begin{align*}
       Q_n(x)=\frac{x^{-\alpha}e^x}{n!}\frac{d^n}{dx^n}\brac{e^{-x} x^{n+\alpha}}. 
    \end{align*}

    By a tensorization procedure \cite[Section 1.15.3]{bakry2014analysis}, one can extend the above structure to the multidimensional setting: $\tilde{E}=\R^d$, $\tilde{\pi}(x)=\pi(x_1)\ldots\pi(x_d)$ and 
    \begin{align*}
        \tilde{L}f(x)&= \sum_{i=1}^d x_i\frac{\partial^2 f}{\partial x_i^2}(x)+(\alpha-x_i)\frac{\partial f}{\partial x_i}(x); \quad\tilde{\Gamma}f(x)=\sum_{i=1}^dx_i\brac{\frac{\partial f}{\partial x_i}(x)}^2. 
    \end{align*}
    In addition, the eigenfunctions of $\tilde{L}$ corresponding to the eigenvalue $-p$ are 
    \begin{align*}
        \prod_{i_1+\ldots+i_d=p}c(i_1,\ldots,i_d)\prod_{j=1}^d Q_{i_j}(x_j)
    \end{align*}
    where $Q_n,n\geq 1$ are the one-dimensional Laguerre polynomials. 
\end{example}


Our next goal is to define \textit{chaotic} random variables and \textit{chaos grade}, which are important concepts introduced by the authors of \cite{ACP14, bourguintaqqu2019} to obtain four moment theorems on Markov diffusion chaos (with respect to certain probability distance). Let $\Lambda\in (-\infty,0]$ be the set of eigenvalues of the generator $L$. The assumption that $L^2(E,\mathcal{F},\mu)$ is separable implies the set $\Lambda$ is countable. Then
\begin{definition}
\label{def_chaoticrv}
(\cite[Definition 3.5]{bourguintaqqu2019})
    An eigenfunction $F$ corresponding to an eigenvalue $-q$ of $L$ is called {\textit{chaotic}} if there is $\eta>1$ such that $-q\eta\in \Lambda$ and 
    \begin{align*}
        F^2\in \bigoplus_{\substack{-\kappa\in \Lambda\\\kappa\leq \eta q}} \operatorname{Ker}\brac{L+\kappa \mathds{I}}. 
    \end{align*}
    In this case, the smallest $\eta$ satisfying the above requirement is the {\textit{chaos grade}} of $F$.
\end{definition}

\begin{example}
    Per \cite[Example 3.6]{bourguintaqqu2019} and \cite[Section 4]{ACP14}, the eigenfunctions associated with the Ornstein-Uhlenbeck process, the Laguerre process and the Jacobi process are chaotic eigenfunctions. In particular, the chaos grade of any eigenfunctions associated with the Ornstein-Uhlenbeck process and the Laguerre process is $2$. 
\end{example}



Next, we recall a Leibniz rule
\begin{align*}
    \Gamma(AB,C)= B\Gamma(A,C)+A\Gamma(B,C),\qquad A,B,C\in \mathcal{A}
\end{align*}
as an easy consequence of the diffusion property \eqref{diffusionproperty_gamma}.  It implies the integration-by-part formula
\begin{align}
\label{intbypartformula_gammaoperator}
    \E{\Gamma(A,B)C}=-\E{A\brac{CLB+\Gamma(B,C)}} ,\qquad A,B,C\in \mathcal{A}. 
\end{align}

Moreover, let us introduce the space of smooth random variables 
\begin{align*}
    \mathbb{D}^\infty:=\left\{L^kF\in \bigcap_{p\in (1,\infty)} L^p\cap \operatorname{Dom}(L),k\in \N \right\}. 
\end{align*}
	

The following theorem    is known for random variables on Wiener chaos (\cite{nualart2006malliavin}) and is a straightforward extension. However, for the purpose of completeness, we provide a  proof here. 
	\begin{theorem}
		\label{theorem_herrydensity}	Let $F\in \mathcal{A}$ and $\Gamma(F)^{-1}\in L^{4N+4}(E,\mu)$ for some $N\in \N$. Then $F$ admits a density that is N-time differentiable. Moreover for $0\leq k\leq N$, the k-th derivative of the density function has the representation
		\begin{align}
			p_F^{(k)}(x)=(-1)^k\E{\mathds{1}_{\{F>x\}}M_{k+1}}=(-1)^k\E{\mathds{1}_{\{F>x\}}\frac{R_{k+1}}{\Gamma(F)^{2(k+1)}}}\,, 
		\end{align}
where the family of random variables $\{R_k:k\in \N \}$ are recursively defined by 		\begin{align*}
	R_0:=1,\qquad		R_{k+1}:=\Gamma(F)\brac{-R_kLF-\Gamma(F,R_k)}+(2k+1)R_k\Gamma(F,\Gamma(F)). 
		\end{align*}
	\end{theorem}
\begin{proof}
We emphasize that the idea of the proof is from \cite[Proposition 2.1]{herry2023regularity}, however since the former is presented in the setting of Wiener space, we sketch the proof  in this  more general setting of Markov diffusion chaos. 

    As the first step, we will show a useful formula that is an easy consequence of \eqref{intbypartformula_gammaoperator}: suppose $\psi\in C^\infty_c(\R)$ and $G\in \mathbb{D}^\infty$ with $\Gamma(F)^{-1}\in L^4(E,\mu)$ then
    \begin{align}
		\label{intbypart_gammaoperator_second}
		\E{\psi'(F)G}&=\E{\Gamma\brac{\psi(F),F}\frac{G}{\Gamma(F)} }\nonumber\\
        &=-\E{\psi(F)\brac{\frac{GLF}{\Gamma(F)}+\Gamma\brac{\frac{G}{\Gamma(F)},F} }}.
	\end{align}
 In the second step, we will use \eqref{intbypart_gammaoperator_second} to derive the representation for $p_F$. We have
\begin{align}
\label{densityF_firststep}
    \E{\psi'(F)}&=-\E{\psi(F)\brac{\frac{1}{\Gamma\brac{F}}LF+\Gamma\brac{F, \frac{1}{\Gamma\brac{F}} }    }  }\nonumber\\
    &=-\E{\psi(F)\brac{ \frac{LF}{\Gamma\brac{F}}   -\frac{\Gamma\brac{F, \Gamma\brac{F} } }{\Gamma\brac{F}^2}     } }. 
\end{align}
 Furthermore, it is a standard argument (see for instance \cite[Proof of Proposition 2.1.1]{nualart2006malliavin}) that: if there is a square integrable random variable $H$ such that $\E{\psi'(F)}=\E{\psi(F)H}$ for any smooth function with compact support $\psi$, then $F$ admits a density with the representation
    \begin{align}
    \label{factdensity}
        p_F(x)=\E{\mathds{1}_{\{F>x \}}H }. 
    \end{align}
Thus, \eqref{densityF_firststep} and \eqref{factdensity} give us the desired representation for $p_F$. 

In the third step, we assume that the $(k-1)$-th and $k$-th derivatives of $p_F$ exist, and our goal is to show the following statement: suppose there exists a random variable $H$ such that $p_F^{(k-1)}(x)= \E{\mathds{1}_{ \{x\leq F\} }H  }$, then it holds that 
		\begin{align}
        \label{recursiverelationfordensity}
			p_F^{(k)}(x)= \E{\mathds{1}_{ \{x\leq F\} } \brac{\frac{HLF}{\Gamma(F)}+\Gamma\brac{\frac{H}{\Gamma(F)},F }  }  }. 
		\end{align}
Indeed,	the formula $ p_F^{(k-1)}(x)= \E{\mathds{1}_{ \{x\leq F\} }H  }$ implies for $\psi\in C^\infty_c(\R)$, 
		\begin{align*}
			\E{\psi^{(k)}(F) }&=\int_\R \psi^{(k)}(x)p_F(x)dx\\
			&=(-1)^{k-1}\int_\R \psi'(x) p_F^{(k-1)}(x)dx\\
			&=(-1)^{k-1}\int_\R \psi'(x) \E{\mathds{1}_{ \{x\leq F\} }H  } dx\\
			&=(-1)^{k-1}\E{\brac{\int_\R \psi'(x) \mathds{1}_{ \{x\leq F\} }dx} H   }\\
			&=(-1)^{k-1}\E{\psi(F)H}. 
		\end{align*}
Combine this  with \eqref{intbypart_gammaoperator_second} to get
		\begin{align*}
			\E{\psi^{(k+1)}(F) }&=(-1)^{k-1}\E{\psi'(F)H}\\
			&=(-1)^{k}\E{\psi(F)\brac{\frac{HLF}{\Gamma(F)}+\Gamma\brac{\frac{H}{\Gamma(F)},F} } }\\
			&=(-1)^{k}\int_\R \psi'(x) \E{\mathds{1}_{ \{x\leq F\} } \brac{\frac{HLF}{\Gamma(F)}+\Gamma\brac{\frac{H}{\Gamma(F)},F} }  }dx\\
			&=(-1)^{2k}\int_\R \psi^{(k+1)}(x) h(x)dx=\int_\R \psi^{(k+1)}(x) h(x)dx,
		\end{align*}
		where $h$ is a function such that $h^{(k)}(x)= \E{\mathds{1}_{ \{x\leq F\} }  \brac{\frac{HLF}{\Gamma(F)}+\Gamma\brac{\frac{H}{\Gamma(F)},F} }}$. Thus, we have just obtained \eqref{recursiverelationfordensity}. 

As the final step, we use the representation for $p_F$ in the first step, formula \eqref{recursiverelationfordensity} and an induction argument to deduce the desired representation for $p_F^{(k)}$ for any $k\in \N$. The details of this final step are omitted. 
\end{proof}

\section{Some density representation for the target distribution} 
\label{section_somedensity}	

For a random variable $F$ in a Markov chaos,  Theorem 
\ref{theorem_herrydensity} gives a sufficient condition 
such that the probability law of $F$ is absolutely continuous with 
respect to the Lebesgue measure  and the density $p_F(x)$ is given by 
\eqref{recursiverelationfordensity}.  We want to use these  formulas  to obtain the difference between   this density $p_F(x)$ 
and its derivatives with a targeted density $p$ and its derivatives.  To this end, we need to have a   representation similar to \eqref{recursiverelationfordensity} for  the targeted density.  Of course we may still use 
equation \eqref{recursiverelationfordensity},    which however,  does not incorporate the specific form of the known targeted density.  In this section we obtain a 
representation of the targeted density for a very general distribution. When the targeted  distribution is normal 
then, we obtain the Stein's formula  and when    
the targeted  distribution is Gamma distribution, we obtain 
generalization of the Stein's formula (e.g. \cite{BDH24}
ad references therein).


\begin{proposition}
\label{prop_densityrepfortarget}
Assume $F$ is a random variable that admits a density function $p$ with support on an interval $(\ell,u), -\infty\leq \ell<u\leq \infty$. 
\begin{enumerate}[label=\roman*)]
	\item   If $p(x)$ is continuously differentiable on $(\ell,u)$,  then for every $x\in (\ell,u)$
\begin{equation}
	p(x)=\EE \left[ \mathds{1}_{\{ F>x\}} h(F) \right]\,,\quad \hbox{where}\quad  
	h(x)=h_0(x)=-\frac{d}{dx} \log (p(x))\,. \label{e.1} 
	\end{equation}  
\item Let $k\in \N$.  If $p(x)$ has continuous derivative up to order $k+1$ and moreover $\lim_{x\to u^{-}} p^{(k)}(x)=0$,  then for every $x\in (\ell,u)$    
\begin{equation}
	p^{(k)}(x)=\EE \left[ \mathds{1}_{\{ F>x\}} h_k(F) \right]\,,\quad \hbox{where}\quad 
	h_k(x)=-\frac{  p^{(k+1)} (x)}{p(x)} \,. \label{e.2}  
\end{equation} 
\item The sequence $h_k$'s  satisfy  the following recursive formula:
\begin{equation}
	h_{k+1}(x)=h_k'(x)+h_k(x)\frac{d}{dx} \log p(x) =h_k'(x)-h_k(x) h_0(x)\,. 
	\label{e.3} 
\end{equation}
\end{enumerate}
\end{proposition} 

\begin{proof} 
We have  for any $x\in (\ell, u)$, 
	\[
	\begin{split}
	\EE \left[ \mathds{1}_{\{ F>x\}} h_0(F) \right]
	=& \int _x^{  u}   	 h_0(y) p(y) dy 
	= - \int_x^{  u}  	 \frac{ p'(y)}{p(y)} p(y) dy\\
	=& -\int _x^{  u}   p'(y) dy = p(x)\,
	\end{split}
	\] 
which is \eqref{e.1}.  

Next, regarding \eqref{e.2},  the fact that $\lim_{x\to u^{-}} p^{(k)}(x)=0$ implies
	\[
	\begin{split}
		\EE \left[ \mathds{1}_{\{ F>x\}} h_k (F) \right]
		=& \int _x^{  u}   	 h_k (y) p(y) dy 
		= - \int_x^{  u}  	 \frac{ p^{(k+1}(y)}{p(y)} p(y) dy\\
		=& -\int _x^{  u}    p^{(k+1)}(y) dy = -p^{(k)} (x)\,. 
	\end{split}
	\]
	
Finally, we can differentiate $h_k(x)$ to get the recursive relation
	\[
	h_k'(x)=-\frac{p^{(k+2)}}{p(x)}+\frac{p^{(k+1)} p'(x)}{p^2(x)}
	=h_{k+1}(x)-h_k(x)\frac{d}{dx} \log p(x)\,. 
	\]
	This proves the proposition. 
\end{proof}	

\begin{example}  If  $p(x)=e^{-\frac{x^2}{2}}/\sqrt{2\pi}$
is the density for normal distribution, then   $\log p(x)=-\frac{x^2}{2} -\log (\sqrt{2\pi})$.
\[
\frac{d}{dx} \log p(x)=-x\,. 
\]
Thus for a standard normal random variable
\begin{equation}
	p(x)=\EE\left[ \mathds{1}_{\{N>x\}} N\right]\,. 
\end{equation}
This is the  Stein's original  formula  for normal distribution.    
\end{example}

\begin{example}   If   $p(x)=\frac{1}{\Gamma(\alpha)\theta^\alpha}x^{\alpha-1} e^{-x/\theta}$, $x>0$ is the Gamma  density,  then 
\[
\frac{d}{dx} \log p(x)= -\frac{1}{\theta}+\frac{\alpha-1}{x}\,. 
\]
Thus for a Gamma random variable  $G$, we have 
\begin{equation*}
	p(x)=\EE\left\{ \mathds{1}_{\{G>x\}} \left[ \frac{1}{\theta}+\frac{1-\alpha }{G}\right] \right\}\,.
\end{equation*} 
and when $\theta=1$, 
\begin{equation}
	p(x)=\EE\left\{ \mathds{1}_{\{G>x\}} \left[ 1+\frac{1-\alpha }{G}\right] \right\}\,.
\end{equation} 
This is the formula used intensely in \cite{BDH24}.  
\end{example}


\section{General diffusion}
\label{section_generaldiffusion}
~\

The following is the standing set of assumptions for all results in this section. 
\begin{customassumption}{A}~
		\label{assum_generaldiffusion}

\begin{itemize}
   \item $\nu$ is a distribution with zero mean,  finite variance and admits a density $p$ on $\R$. $p$ is strictly positive on $(\ell,u),-\infty \leq \ell<u\leq \infty$   and zero elsewhere. Moreover, $p$ is bounded, twice differentiable with continuous first derivative. Finally, $\lim_{x\to u^{-}} p(x)=0$.

    \item $a:\R \to \R$ is continuous on $(\ell,u)$ and is differentiable on $(-\infty,\ell]\cup [u,\infty)$. Furthermore, there exists a real number $k\in (\ell,u)$ such that 
\begin{align}
\label{condition_drifta}
    a(x)>0,\ell<x<k,\qquad a(x)<0,k<x<u. 
\end{align} 
Finally, the function $ap$ is bounded on $(\ell,u)$ and 
\begin{align*}
   \int_\ell^u a(x)p(x)dy=0. 
\end{align*}

\item   $\int_\ell^u |a(x)p(x)|dx<\infty$  and we denote  \begin{align}
    \label{def_driftb_gendif}
     b(x)=\frac{\int_\ell^x2a(y) p(y)dy}{p(x)}, \quad x\in 
    (\ell, u)\,. 
\end{align}

If  $u=\infty$, 
\begin{align}
    \label{condition_diffusionb}
    \liminf_{x\to u^-} b(x)=\lim_{\epsilon\to 0}\inf_{x\in (u-\epsilon,u)} b(x) >0. 
\end{align}
\end{itemize}

\end{customassumption}
Observe that the above set of assumptions is similar in parts to those assumed  in \cite{tudorkusuoka2012,eden2015nourdin, tudorkusuoka2018characterization,bourguintaqqu2019}. In particular, we need the additional assumptions that $p$ has second derivative and that $a$ is differentiable on $(-\infty,\ell]\cup [u,\infty)$. 


Under Assumption \ref{assum_generaldiffusion} and per \cite{bibby2005diffusion}, the stochastic differential equation
\begin{align}
\label{generaldiffusion}
    dX_t=a(X_t)dt+\sqrt{\mathds{1}_{(\ell,u)}(X_t)b(X_t)}dB_t
\end{align} 
has a unique global Markovian weak solution 
on $(\ell,u)$ and moreover, $\nu$ is its 
unique invariant distribution.

The next lemma is an easy consequence of Proposition \ref{prop_densityrepfortarget}.
\begin{lemma}
\label{lemma_densityrep_generaldiffision}
Assume Assumption \ref{assum_generaldiffusion}. Let $\mathcal{Y}$ be a random variable distributed as $\nu$. Then the density of $\mathcal{Y}$ has the representation
    \begin{align*}
        p(x)=\begin{cases} \mathbb{E}\big[\mathds{1}_{\{\mathcal{Y}>x \}} \frac{-2a(\mathcal{Y})+b'(\mathcal{Y})}{b(\mathcal{Y})}\big] &\qquad x>\ell\,;  \\
         \mathbb{E}\big[\mathds{1}_{\{\mathcal{Y}<x \}} \frac{-2a(\mathcal{Y})+b'(\mathcal{Y})}{b(\mathcal{Y})}\big] &\qquad x\leq \ell\,. \\
        \end{cases}
    \end{align*}
\end{lemma}

\begin{proof}
It is well-known (see for instance \cite[(6.22)]{karlin1981second}) that the invariant probability measure  $p$   of \eqref{generaldiffusion}   
can be written as 
\begin{align*}
    p(x)=\frac{p(x_0)b(x_0)}{b(x)}\exp\brac{\int_{x_0}^x \frac{2a(y)}{b(y)}dy }, 
\end{align*}
where $x_0$ is any point in $(\ell,u)$. We can further compute that 
\begin{align*}
    -\frac{d}{dx}\log p(x)=&-\frac{d}{dx} \left( \int_{x_0}^x \frac{2a(y)}{b(y)}dy +\log (p(x_0)b(x_0))-\log b(x) \right) \\
    =&\frac{-2a(x)+b'(x)}{b(x)}, 
\end{align*}
so that via part i of Proposition \ref{prop_densityrepfortarget}, we have the desired representation of $p$ for $x\in (\ell,u)$. The formula when $x\geq u$ or $x\leq \ell$ is immediate since the support of $p$ is the interval $(\ell,u)$, so that for $x\geq u$,
\begin{align*}
   p(x)=\mathbb{E}\left[\mathds{1}_{\{\mathcal{Y}>x \}} \frac{-2a(\mathcal{Y})+b'(\mathcal{Y})}{b(\mathcal{Y})}\right]=0
\end{align*}
and for $x\leq \ell$,
\begin{align*}
   p(x) =\mathbb{E}\left[\mathds{1}_{\{\mathcal{Y}<x \}} \frac{-2a(\mathcal{Y})+b'(\mathcal{Y})}{b(\mathcal{Y})}\right]=0. 
\end{align*}
This completes the proof of the lemma. 
\end{proof}

For the current section, we will need a slightly different version of the density representation in Theorem \ref{theorem_herrydensity}. The formula in Theorem \ref{theorem_herrydensity} will be useful later (Section \ref{section_densityeigenfunction}) in the more special case where the functional $F$ is an eigenfunction of $L$.

\begin{lemma}
\label{lemma_densitygeneralfunctional}
  Let $F\in \mathcal{A}$ and $a$ be a real-valued function such that $\Gamma\brac{F,-L^{-1}a(F)}^{-1}\in L^{4}(E,\mu)$. Set $h_F:= \Gamma\brac{F,-L^{-1}2a(F)}$. Then $F$ admits a density with the representation
\begin{align*}
   p_F(x)=  \E{\mathds{1}_{\{F>x\}}\brac{\frac{-2\E{a(F)}+2a(F)}{h_F}+\frac{\Gamma\brac{-L^{-1}2a(F),h_F }}{h_F^2}}  }. 
\end{align*}
\end{lemma}

\begin{proof}

The proof is the same as the second step in
 the proof of Theorem \ref{theorem_herrydensity}, except that we replace \eqref{densityF_firststep} with
\begin{align*}
    \E{\psi'(F)}&=\E{\Gamma\brac{\psi(F),-L^{-1}2a(F)}\frac{1}{ \Gamma\brac{F,-L^{-1}2a(F)}} }\\
    &=\E{\Gamma\brac{\psi(F),-L^{-1}2a(F)}\frac{1}{ h_F }} \\
    &=-\E{\psi(F)\brac{\frac{1}{h_F}L(-L^{-1})2a(F)+\Gamma\brac{-L^{-1}2a(F), \frac{1}{h_F} }    }  }\\
    &=-\E{\psi(F)\brac{\frac{2\E{a(F)}-2a(F)}{h_F}-\frac{\Gamma\brac{-L^{-1}2a(F),h_F }}{h_F^2}  } }. 
\end{align*} 
This proves the lemma. 
\end{proof}

\subsection{Stein's method}
\label{section_steinmethod_gendif}
~\

We will present Stein's method for approximating $\mathcal{Y}$.  Let $a, b , p_\mathcal{Y}$ be   related by Assumption
A (e.g. \eqref{def_driftb_gendif}). 
Given an appropriate test function $h$  such that $\E{h(\mathcal{Y})}<\infty$, a Stein's equation for approximation of $\mathcal{Y}$ is
\begin{align}
\label{equation_stein}
  2a(y)g(y)+\mathds{1}_{\{y\in (\ell,u) \} }b(y)g'(y)=h(y)-\E{h(\mathcal{Y})}. 
\end{align}
By a standard integrating factor argument, we obtain the solution to  \eqref{equation_stein} with initial condition $g(\ell)=0$ as 
\begin{align}
\label{solution_steinequation}
    g(y)
    & = \mathds{1}_{\{y\in(\ell,u) \}}\frac{1}{b(y)p_\mathcal{Y}(y)}\int_{\ell}^y \brac{h(w)- \E{h(\mathcal{Y})}}p_\mathcal{Y}(w)dw \nonumber\\
    &\qquad \quad +\mathds{1}_{\{y\notin(\ell,u) \}}\frac{h(y)- \E{h(\mathcal{Y})}}{-2a(y)} \nonumber \\
  &=-\mathds{1}_{\{y\in(\ell,u) \}}\frac{1}{b(y)p_\mathcal{Y}(y)}\int_{y}^u \brac{h(w)- \E{h(\mathcal{Y})}}p_\mathcal{Y}(w)dw\nonumber\\
    &\qquad \quad+\mathds{1}_{\{y\notin(\ell,u) \}}\frac{h(y)- \E{h(\mathcal{Y})}}{-2a(y)}. 
\end{align}
Then $g$ is a map from $\R\setminus \brac{\left\{y\in(\ell,u)^\complement:a(y)=0 \right\}\cup\left\{y\in(\ell,u):b(y)=0 \right\}}$ to $\R$. The next three lemmas are devoted to bounding $g$ and its first derivative. 

\begin{lemma}
\label{lemma_xinlandu_gendif}
Fix  $x\in (\ell, u)$ and $k\in \N$. Let the function $h$ in the ODE \eqref{equation_stein} be 
    \begin{align*}
        h(y):=\mathds{1}_{\{y>x\}}\frac{-2a(y)+b'(y)}{b(y)},\quad y\in \R\setminus \left\{y\in(x,\infty):b(y)=0 \right\}.
    \end{align*} 
Then a solution $g_x: \R\setminus \brac{\left\{y\in(\ell,u)^\complement:a(y)=0 \right\}\cup\left\{y\in(\ell,u):b(y)=0 \right\}}\to\R$ to \eqref{equation_stein} exists and satisfies for every $y$ in its domain, 

\begin{align*}
    & \abs{g_x(y)}\\
    &\leq \mathcal{V}_1(x,y):=\mathds{1}_{ \{y\geq u \} }\frac{\abs{-2a(y)+b'(y)}}{\abs{b(y)}2\abs{a(y)} }+\mathds{1}_{ \{y\notin (\ell,u) \} }\frac{p_\mathcal{Y}(x) }{2\abs{a(y)}}+\mathds{1}_{\{y\in(\ell,u) \}}C(x); 
\end{align*}
and 
\begin{align*}
    &\abs{g_x'(y)}\\
    &\leq \mathcal{V}_2(x,y):= \mathds{1}_{\{y\in(\ell,u) \}}\brac{C(x)\frac{\abs{a(y)}}{\abs{b(y)}}+\frac{\abs{-2a(y)+b'(y)}}{\abs{b(y)}^2}+\frac{p_\mathcal{Y}(x)}{\abs{b(y)}}}\\
    &+\mathds{1}_{ \{y\geq u\} } 2\abs{\frac{(-2a'(y)+b''(y))b(y)-(-2a(y)+b'(y))b'(y) }{b(y)^2a(y)}}\\
   &\qquad+\mathds{1}_{ \{y\geq u \} }
\abs{\frac{-2a(y)+b'(y)}{b(y)4a(y)^2} }+\mathds{1}_{ \{y\notin (\ell,u) \} }\frac{p_\mathcal{Y}(x) }{4a(y)^2}. 
\end{align*} 

Here $C(x)>0$ is a function that is well-defined on $(\ell,u)$ and does not depend on $y$.

\end{lemma}

\begin{proof} We divide the proof into two   steps. 

\underline{Step  1:}  Bounding $g$.  
When $x\in (\ell,u)$, we know that $\E{h(\mathcal{Y})}=p_\mathcal{Y}(x)$ so that
\begin{align}
\label{steinsolution_expand_gendif}
    g_x(y) =&\mathds{1}_{ \{y\in(\ell,u) \} }\frac{1}{b(y)p_\mathcal{Y}(y)}\brac{\int_{y}^u -h(w)p_\mathcal{Y}(w)+p_{\mathcal{Y}}(x)p_\mathcal{Y}(w)dw}\nonumber\\
    &\quad +\mathds{1}_{ \{y\notin(\ell,u) \} } \frac{h(y)-p_\mathcal{Y}(x)}{-2a(y)}\,. 
\end{align}

One can see right away that 
\begin{align}
\label{steinsolution_easyterm_gendif}
    \abs{\mathds{1}_{ \{y\notin(\ell,u) \} }\frac{h(y)-p_\mathcal{Y}(x)}{2a(y)}} \leq \mathds{1}_{ \{y\geq u \} }\frac{\abs{-2a(y)+b'(y)}}{\abs{b(y)}2\abs{a(y)} }+\mathds{1}_{ \{y\notin (\ell,u) \} }\frac{p_\mathcal{Y}(x) }{2\abs{a(y)}}. 
\end{align}
In particular, we have $y\geq u$ for the first term on the right hand side as a combination of $y\notin(\ell,u)$,$x\in (\ell,u)$ and $y> x$.

Next, let us set 
\begin{align*}
    \mathcal{A}(y):= \frac{1}{b(y)p_\mathcal{Y}(y)}\brac{\int_{y}^u -h(w)p_\mathcal{Y}(w)+p_\mathcal{Y}(x) p_\mathcal{Y}(w)dw},\qquad y\in (\ell,u)\,. 
\end{align*}
To bound $\mathcal{A}(y)$ for a fixed $x\in (\ell,u)$,  we will follow the argument in \cite[Proof of Proposition 2]{tudorkusuoka2012}. The fact that $b(y)$ and $p_{\mathcal{Y}}(y)$ are  strictly positive and differentiable on $(\ell,u)$ implies that for a fixed $x\in (\ell,u)$, $\mathcal{A}(y)$ is continuous on $(\ell,u)$. Thus, if we can bound $\limsup_{y\to \ell^+} \abs{\mathcal{A}(y)}$ and $\liminf_{y\to u^-} \abs{\mathcal{A}(y)}$, then it would lead to the uniform bound
\begin{align}
\label{uniformbound}
    \sup_{y\in (\ell,u)} \abs{\mathcal{A}(y)}\leq C(x)
\end{align}
for a function $C$ that is well defined on $(\ell,u)$. 

By L'h\^{o}pital's rule and the fact that $b(y)p_\mathcal{Y}(y)=2\int_\ell^y a(w) p_{\mathcal{Y}}(w)dw$, we can write
\begin{align}
\label{term_lhopitale}
     \limsup_{y\to \ell^+} \abs{\mathcal{A}(y)}&=\limsup_{y\to \ell^+}\abs{\frac{h(y)p_\mathcal{Y}(y)-p_\mathcal{Y}(x) p_\mathcal{Y}(y)  }{2a(y)p_\mathcal{Y}(y) }}\nonumber\\
     &=\limsup_{y\to \ell^+}\abs{\frac{h(y)-p_\mathcal{Y}(x)  }{2a(y) }} =\limsup_{y\to \ell^+}\abs{\frac{p_\mathcal{Y}(x)  }{2a(y) }} \leq Cp_\mathcal{Y}(x), 
\end{align} 
where $C$ is some constant independent of $x$ and $y$. To get the last line, we use \eqref{condition_drifta} which implies $\liminf_{y\to u^-} a(y)>0$ and $\limsup_{y\to \ell^+}a(y)<0$. To get the second to last line, observe that $x>\ell$ implies for any $\epsilon>0$ such that $\ell+\epsilon<x$, we have $\sup_{y\in (\ell,\ell+\epsilon)}\abs{h(y)}=\sup_{y\in (\ell,\ell+\epsilon)}\abs{\mathds{1}_{\{y>x \}}\frac{-2a(y)+b'(y)}{b(y)}}=0 $. This further implies 
\begin{align*}
    \limsup_{y\to \ell^+}\abs{h(y)}= \lim_{\epsilon\to 0} \sup_{y\in (\ell,\ell+\epsilon)}\abs{h(y)}=0. 
\end{align*}

 Next, let us bound $  \liminf_{y\to u^-} \abs{\mathcal{A}(y)}$. By Lemma \ref{lemma_densityrep_generaldiffision}, it is possible to write 
 \begin{align*}
     \mathcal{A}(y)&= \frac{1}{b(y)p_\mathcal{Y}(y)}\int_y^u -(1-\mathds{1}_{\{w\leq x \}} )\frac{-2a(w)+b'(w)}{b(w)}p_\mathcal{Y}(w)+p_{\mathcal{Y}}(x)p_\mathcal{Y}(w) dw\\
     &=-\frac{p_\mathcal{Y}(y)}{b(y)p_\mathcal{Y}(y)}+\frac{1}{b(y)p_\mathcal{Y}(y)}\int_y^u \mathds{1}_{\{w\leq x \}} \frac{-2a(w)+b'(w)}{b(w)}p_\mathcal{Y}(w)+p_{\mathcal{Y}}(x)p_\mathcal{Y}(w) dw. 
 \end{align*}
 Therefore,
 \begin{align*}
     \liminf_{y\to u^-} \abs{\mathcal{A}(y)}&\leq \liminf_{y\to u^-}\frac{1}{\abs{b(y)}}\\
     &+\liminf_{y\to u^-} \abs{\frac{1}{b(y)p_\mathcal{Y}(y)}\int_y^u \mathds{1}_{\{w\leq x \}} \frac{-2a(w)+b'(w)}{b(w)}p_\mathcal{Y}(w)+p_{\mathcal{Y}}(x)p_\mathcal{Y}(w) dw}. 
 \end{align*}
 The first term on the right hand side is finite due to \eqref{condition_diffusionb}. For the second term, we use L'h\^{o}pital's rule and $b(y)p_\mathcal{Y}(y)=2\int_\ell^y a(w) p_{\mathcal{Y}}(w)dw$ to get
 \begin{align*}
  &\liminf_{y\to u^-}  \mathcal{A}_1(y)\\
  &:=   \liminf_{y\to u^-} \abs{\frac{1}{2\int_\ell^y a(w) p_{\mathcal{Y}}(w)dw}\int_y^u \mathds{1}_{\{w\leq x \}} \frac{-2a(w)+b'(w)}{b(w)}p_\mathcal{Y}(w)+p_{\mathcal{Y}}(x)p_\mathcal{Y}(w) dw}\\
      & \leq  \liminf_{y\to u^-} \abs{\frac{ \mathds{1}_{\{y\leq x \}} \frac{-2a(y)+b'(w)}{b(y)} }{2a(y)} }+ p_{\mathcal{Y}}(x)\liminf_{y\to u^-}\frac{1}{\abs{a(y)}}.  
 \end{align*}
The fact that $x<u$ implies $\liminf_{y\to u^-}  \mathds{1}_{\{y\leq x \}} \frac{-2a(y)+b'(w)}{b(y)}=0$. This combined with \eqref{condition_drifta} leads   to 
\begin{align*}
      \liminf_{y\to u^-}  \mathcal{A}_1(y)<\infty
\end{align*}
 and hence $\liminf_{y\to u^-}  \mathcal{A}(y)<\infty$. Consequently, we arrive at the bound \eqref{uniformbound}. Finally, a combination of \eqref{steinsolution_expand_gendif}, \eqref{steinsolution_easyterm_gendif} and \eqref{uniformbound} yields the desired bound on $g$. 
~\\

\underline{Step  2:} bounding $g'$. 
From \eqref{equation_stein} and $p_\mathcal{Y}(x)=\E{h(\mathcal{Y})}$ when $x\in (\ell,u)$, we can write
\begin{align*} 
	 g'_x(y)=&\mathds{1}_{\{y\in(\ell,u) \}}\brac{-\frac{2a(y)}{b(y)} g(y) +\frac{  h(y)-p_\mathcal{Y}(x)}{ b(y)}}\\
     &+\mathds{1}_{\{y\notin(\ell,u) \}} \frac{-h'(y)2a(y)-(h(y)-p_\mathcal{Y}(x))}{4a(y)^2},
\end{align*}
noting that $g'(y)$ is well defined on $\R\setminus \brac{\left\{y\in(\ell,u)^\complement:a(y)=0 \right\}\cup\left\{y\in(\ell,u):b(y)=0 \right\}}$.

The bound on $g(y),y\in (\ell,u)$ has been done in   Step   1, which leads to 
\begin{align*}
   & \mathds{1}_{\{y\in(\ell,u) \}}\abs{-\frac{2a(y)}{b(y)} g(y) +\frac{  h(y)-p_\mathcal{Y}(x)}{ b(y)}}\\
   &\leq  \mathds{1}_{\{y\in(\ell,u) \}}\brac{C(x)\frac{\abs{a(y)}}{\abs{b(y)}}+\frac{\abs{-2a(y)+b'(y)}}{\abs{b(y)}^2}+\frac{p_\mathcal{Y}(x)}{\abs{b(y)}}}. 
\end{align*}


Next, we have for $y\notin (\ell,u)$ that 
\begin{align*}
    \abs{h'(y)}&\leq \mathds{1}_{ \{y\geq u\} }\abs{\frac{(-2a'(y)+b''(y))b(y)-(-2a(y)+b'(y))b'(y) }{b(y)^2}}.
\end{align*}
In particular, \eqref{def_driftb_gendif} and the fact that $a$ is differentiable on $(u,\infty)$, $p$ is twice differentiable on $\R$ from Assumption \ref{assum_generaldiffusion} imply $b''$ exists. It follows that 
\begin{align*}
    &\mathds{1}_{\{y\notin(\ell,u) \}}\abs{ \frac{-2h'(y)a(y)-(h(y)-p_\mathcal{Y}(x))}{4a(y)^2}}\\
    &\leq \mathds{1}_{ \{y\geq u\} } 2\abs{\frac{(-2a'(y)+b''(y))b(y)-(-2a(y)+b'(y))b'(y) }{b(y)^2a(y)}}\\
   &\qquad+\mathds{1}_{ \{y\geq u \} }
\abs{\frac{-2a(y)+b'(y)}{b(y)4a(y)^2} }+\mathds{1}_{ \{y\notin (\ell,u) \} }\frac{p_\mathcal{Y}(x) }{4a(y)^2}. 
\end{align*}

Combining the previous bounds   leads  us to the desired bound on $g'$.  
\end{proof}

\begin{lemma}
\label{lemma_steinmethod_xlessthanl_gendif} 
 Fix  $x\leq \ell$ and $k\in \N$. Let the function $h$ in the ODE \eqref{equation_stein} be 
    \begin{align*}
        h(y):=\mathds{1}_{\{y\leq x\}}\frac{-2a(y)+b'(y)}{b(y)},\quad y\in \R\setminus \left\{y\in(-\infty,x]:b(y)=0 \right\}.
    \end{align*} 
Then a solution $g_x: \R\setminus \brac{\left\{y\in(\ell,u)^\complement:a(y)=0 \right\}\cup\left\{y\in(\ell,u)\cup (-\infty,x]:b(y)=0 \right\}}\to\R$ to \eqref{equation_stein} exists and satisfies for every $y$ in its domain, 
\begin{align}
     \abs{ g_x(y)}&\leq \mathcal{V}_3(y):= \mathds{1}_{\{y\leq \ell\}} \frac{\abs{-2a(y)+b'(y)}}{\abs{b(y)}\abs{2a(y)} }\,;  
    \end{align}
and
\begin{align}
    \abs{ g'_x(y)}&\leq \mathcal{V}_4(y):= \mathds{1}_{\{y\leq \ell\} } \brac{\abs{\frac{(1+b''(y))b(y)-(-2a(y)+b'(y))b'(y) }{b(y)^2a(y)}}+\abs{\frac{-2a(y)+b'(y)}{b(y)4a(y)^2} }}. 
\end{align}

\end{lemma}

\begin{proof}
    The fact that $x\leq \ell$ and $\operatorname{supp}(\mathcal{Y})\subseteq (\ell,u)$ imply $\E{h(\mathcal{Y})}=0$. Then it follows from \eqref{solution_steinequation} that 
    \begin{align*}
        g_x(y)&=\mathds{1}_{\{y\notin (\ell,u)\}}\frac{h(y)}{-2a(y)}
    \end{align*}
    and hence
    \begin{align*}
     \abs{ g_x(y)}&\leq  \mathds{1}_{\{y\leq \ell\} } \frac{\abs{-2a(y)+b'(y)}}{\abs{b(y)}\abs{2a(y)} }. 
    \end{align*}
$g'_x(y)$ can be bounded similarly since
\begin{align*}
    g'_x(y)&=\mathds{1}_{ \{y\notin(\ell,u) \} } \frac{-h'(y)2a(y)-h(y)}{4a(y)^2}. 
\end{align*}
This completes  the proof of the lemma. 
\end{proof}

The next lemma can be proved in the same way as  for  the previous lemma
and we omit the details.

\begin{lemma}
\label{lemma_steinmethod_xmorethanu_gendif}
  Fix  $x\geq u$ and $k\in \N$. Let the function $h$ in the ODE \eqref{equation_stein} be 
    \begin{align*}
        h(y):=\mathds{1}_{\{y> x\}}\frac{-2a(y)+b'(y)}{b(y)},\quad y\in \R\setminus \left\{y\in (x,\infty):b(y)=0 \right\}
    \end{align*} 
Then a solution $g_x: \R\setminus \brac{\left\{y\in(\ell,u)^\complement:a(y)=0 \right\}\cup\left\{y\in(\ell,u)\cup (x,\infty):b(y)=0 \right\}}\to\R$ exists and satisfies for every $y$ in its domain, 
\begin{align*}
     \abs{ g(y)}&\leq  \mathcal{V}_5(y):=\mathds{1}_{\{y>u\}} \frac{\abs{-2a(y)+b'(y)}}{\abs{b(y)}\abs{2a(y)} }
    \end{align*}
and
\begin{align*}
    &\abs{ g'(y)}\\
    &\leq  \mathcal{V}_6(y):=\mathds{1}_{\{y>u\} } \brac{\abs{\frac{(1+b''(y))b(y)-(-2a(y)+b'(y))b'(y) }{b(y)^22a(y)}}+\abs{\frac{-2a(y)+b'(y)}{b(y)4a(y)^2} }}. 
\end{align*}

\end{lemma}


We rely on the previous lemmas to derive a Stein's bound for a general diffusion $\mathcal{Y}$.  To  state  the main result  of this section  we define,  
  for any $k\in \N$,   a function from  $\R\setminus \{y\in \R: b(y)=0 \}\to \R$:
\[ h(y)=\begin{cases} 
    \mathds{1}_{\{y\leq x\}}\frac{-2a(y)+b'(y)}{b(y)} &\text{ when } x\leq \ell\,;  \\
      \mathds{1}_{\{y> x\}}\frac{-2a(y)+b'(y)}{b(y)} &\text{ when } x>\ell\,. 
   \end{cases}\] 
\begin{proposition}
\label{prop_steinmethod_gendif}
Let $X$ be a centered random variable in $L^2\brac{E,\mu}$ and $F=X+m$.  Choose any $x\in \R$. Assume   the following moments  
conditions:
\begin{align*}
\E{\mathcal{V}_1(x,F)^2}+\E{\mathcal{V}_2(x,F)^2}+\sum_{i=3}^6 \E{\mathcal{V}_i(F)^2}<\infty\,, 
\end{align*}
where $\mathcal{V}_i,1\leq i\leq 6$ are the functions defined in Lemmas \ref{lemma_xinlandu_gendif}, \ref{lemma_steinmethod_xlessthanl_gendif} and \ref{lemma_steinmethod_xmorethanu_gendif}.  
Then we have the estimate 
\begin{align*}
   &\abs{ \E{h\brac{F}}-\E{h\brac{\mathcal{Y}}}}\\
     &\leq  \brac{\E{\mathcal{V}_2(x,F)^2}^{1/2}+\E{\mathcal{V}_4(F)^2}^{1/2}+\E{\mathcal{V}_6(F)^2}^{1/2} }\\
   &\hspace{16em}\times \E{\brac{ b(F)+  \Gamma\brac{L^{-1}F,F}   }^2 }^{1/2}. 
\end{align*}
\end{proposition}
\begin{proof}
In this proof, $h$ is the function defined above and $g_x$ denotes the solution to the Stein equation   \eqref{solution_steinequation}.       First, observe that for a bounded function $\varphi$ with bounded derivative and random variables $G,H\in L^2(E,\mu)$ such that $\E{G}=0$,
    \begin{align}
    \label{ledouxintbypart}
        \E{G\varphi(H)}=&\E{(LL^{-1}G)\ \varphi(H) }=\E{\Gamma\brac{-L^{-1}G,\varphi(H) }}\nonumber\\
        =&\E{\varphi'(H)\Gamma\brac{-L^{-1}G,H}}. 
    \end{align}
In fact,  the above integration by part argument for random variables on Markov diffusion chaos was first done in the seminal paper \cite{Led12}.  Second, the moment assumption in the current lemma in combination with the bound on $g_x$ and $g_x'$ in Lemmas \ref{lemma_xinlandu_gendif}, \ref{lemma_steinmethod_xlessthanl_gendif} and \ref{lemma_steinmethod_xmorethanu_gendif} imply
\begin{align*}
    \E{\brac{\frac{\alpha}{\alpha+\beta}-F }g_x\brac{F}},\E{g_x'\brac{F}\Gamma\brac{L^{-1}F,F}}<\infty. 
\end{align*}
The above facts and the dominated convergence theorem lead to
\begin{align}
\label{equation_ledoux_appliedtogh}
    \E{\brac{m-F }g_x\brac{F}}=\E{g_x'\brac{F}\Gamma\brac{L^{-1}F,F}}. 
\end{align}

In the case $x> 0$, we have 
\begin{align*}
   &\abs{ \E{h\brac{F}}-\E{h\brac{\mathcal{Y}}}}\\
   &=\abs{\E{b(F)g_x'\brac{F}+(m-F)g_x\brac{F} }}\\
   &\leq \E{g_x'\brac{F}^2  }^{1/2}\E{\brac{ b(F)+  \Gamma\brac{L^{-1}F,F}   }^2 }^{1/2}\\
   &\leq  \brac{\E{\mathcal{V}_2(x,F)^2}^{1/2}+\E{\mathcal{V}_4(F)^2}^{1/2}+\E{\mathcal{V}_6(F)^2}^{1/2} }\\
   &\hspace{16em}\times \E{\brac{ b(F)+  \Gamma\brac{L^{-1}F,F}   }^2 }^{1/2}. 
\end{align*}
The second line is due to the Stein equation \eqref{equation_stein}. The third line is due to Equation \eqref{equation_ledoux_appliedtogh}. The last line is a consequence of Lemma \ref{lemma_xinlandu_gendif} and Lemma \ref{lemma_steinmethod_xmorethanu_gendif} which provide a bound on $\abs{g_x'(y)}$.

In the case $x<\ell$, we have
\begin{align*}
    & \E{h\brac{F}}-\E{h\brac{\mathcal{Y}}}\\
    &=\E{\mathds{1}_{\{y> x\}}\frac{F-m+b'(F)}{b(F)}}-\E{\mathds{1}_{\{y> x\}}\frac{-2a(y)+b'(\mathcal{Y})}{b(\mathcal{Y})}}\\
    &=-\E{\mathds{1}_{\{y\leq x\}}\rho_{k+1}(F)}+\E{\mathds{1}_{\{y\leq x\}}\frac{-2a(y)+b'(\mathcal{Y})}{b(\mathcal{Y})}}.
\end{align*}
At this point, we can use the Stein equation \eqref{equation_stein} with $h(y)=\mathds{1}_{\{y\leq x\}}\frac{-2a(y)+b'(y)}{b(y)}$, Equation \eqref{equation_ledoux_appliedtogh} and Lemma \ref{lemma_steinmethod_xlessthanl_gendif} to obtain the same estimate as the previous one.  
\end{proof}


\subsection{Approximating general diffusion}

\begin{theorem}
\label{theorem_generaldiffusion}
    Let $X$ be a random variable that belongs to $\mathcal{A}$, which has zero mean and admits a density. Let $\mathcal{Y}$ be  a random variable distributed as $\nu$ 
    described in  Assumption \ref{assum_generaldiffusion}. Set $F=X+m$ for $m=\E{\mathcal{Y}}$. Choose $x\in \R$ and assume the moment condition 
\begin{align*}
\E{\mathcal{V}_1(x,F)^2}+\E{\mathcal{V}_2(x,F)^2}+\sum_{i=3}^6 \E{\mathcal{V}_i(F)^2}<\infty \,,  
\end{align*}
where $\mathcal{V}_i,1\leq i\leq 6$ are the functions defined in Lemmas \ref{lemma_xinlandu_gendif}, \ref{lemma_steinmethod_xlessthanl_gendif} and \ref{lemma_steinmethod_xmorethanu_gendif}.  Set $h_F=\Gamma\brac{F,-L^{-1}2a(F)}$ and
\begin{align*}
    &P_1=h_F+b(F),\\
    &P_2=\Gamma(-L^{-1}2a(F),P_1)=\Gamma(-L^{-1}2a(F), h_F)+b'(F)h_F. 
\end{align*} 
Then we have the pointwise estimate for $x\in\R$
    \begin{align*}
            \big|p_F(x)&- p_{\mathcal{Y}}(x)\big|\nonumber\\
        \leq &\brac{\E{\mathcal{V}_2(x,F)^2}^{1/2}+\E{\mathcal{V}_4(F)^2}^{1/2}+\E{\mathcal{V}_6(F)^2}^{1/2} }\E{P_1^2}^{1/2}\nonumber\\
        &+\E{\brac{\frac{2a(F)-b'(F)}{h_F b(F)}}^2 } ^{1/2}\E{P_1^2}^{1/2}\nonumber\\
        &+\E{\frac{\Gamma(-L^{-1}2a(F) )}{h_F^4}}^{1/2}\E{(LP_1)^2}^{1/4}\E{P_1^2}^{1/4}+\E{\frac{2}{\abs{h_F}}}\abs{\E{a(F)}}. 
\end{align*}
\end{theorem}

\begin{remark}
    The appearance of the term $\abs{\E{a(F)}}$ in the bound is standard since if $F\sim \mathcal{Y}$ then $\E{a(F)}=\E{a(\mathcal{Y})}=0$ per the assumption on the function $a$ in Assumption \ref{assum_generaldiffusion}. 
\end{remark}

\begin{remark}  The authors of \cite{HLN14} obtain 
estimates for the difference between the densities of Gaussian
functionals and the normal density and in particular, their estimates are uniform with respect to $x\in \R$. It is hence natural to ask whether
we can do the same for a general target described in  Assumption \ref{assum_generaldiffusion}. The
answer to this question relies on whether Lemmas \ref{lemma_xinlandu_gendif} can be improved from a pointwise estimate
to a uniform estimate, i.e., whether one can control the supremum over
$x>0$ of $\abs{g_x(y)}$ defined therein. We believe  that  this should  be true for some class of target distributions  but  we decide not to tackle this question since the focus of this paper is to  provide  a framework for  density convergence for a wide range of target distributions. Of course,  this question about uniform density convergence for non-normal targets is a very interesting and important  future research topic.  
\end{remark}

\begin{proof}      Via the density representation of $F$ in Lemma \ref{lemma_densitygeneralfunctional}, we can write
    \begin{align}
 \label{generalfunctional_densitydecompose}
         p_F&(x)\nonumber\\
         =&\mathbb{E}\bigg[\mathds{1}_{\{F>x\}} \bigg(\frac{-2a(F)+b'(F)}{b(F)}+\frac{-2\E{a(F)}+2a(F)}{h_F}+\frac{\Gamma\brac{-L^{-1}2a(F),h_F }}{h_F^2}\nonumber\\
        &\qquad\qquad\qquad\qquad\qquad\qquad\qquad\qquad\qquad\qquad- \frac{-2a(F)+b'(F)}{b(F)}  \bigg)\bigg]\nonumber\\
         =&\mathbb{E}\bigg[\mathds{1}_{\{F>x\}} \bigg(\frac{-2a(F)+b'(F)}{b(F)}-\frac{2\E{a(F)}}{h_F}+\brac{\frac{2a(F)}{h_F}+\frac{2a(F)}{b(F)}}\nonumber\\
        &\qquad\qquad\qquad\qquad\qquad+\frac{\Gamma(-L^{-1}2a(F),h_F)+b'(F)h_F }{h_F^2}-\brac{\frac{b'(F)}{b(F)}+\frac{b'(F)}{h_F}} \bigg)\bigg]\nonumber\\
         =&\mathbb{E}\bigg[\mathds{1}_{\{F>x\}} \bigg(\frac{-2a(F)+b'(F)}{b(F)}-\frac{2\E{a(F)}}{h_F} +\frac{2a(F)}{h_Fb(F)}P_1+\frac{1}{h_F^2}P_2-\frac{b'(F)}{b(F)h_F}P_1\bigg)\bigg]. 
    \end{align}

Now to bound $\abs{p_F(x)-p_{\mathcal{Y}}(x)}$, we will consider three cases that are  $x\in (\ell,u)$, $x\leq \ell$ and $x\geq u$. 

\underline{Step 1}: $x\in (\ell,u)$. \ 
 We have from \eqref{generalfunctional_densitydecompose} and Lemma \ref{lemma_densityrep_generaldiffision} that 
\begin{align}
\label{prelimbound_generalfunctional}
    &\abs{p_F(x)-p_{\mathcal{Y}}(x)}\nonumber\\
    &\leq \abs{\E{\mathds{1}_{\{F>x\}} \brac{\frac{-2a(F)+b'(F)}{b(F)}}- \mathds{1}_{\{\mathcal{Y}>x\}} \brac{\frac{-2a(\mathcal{Y})+b'(\mathcal{Y})}{b(\mathcal{Y})} }} }\nonumber\\
    &+\E{\brac{\frac{2a(F)-b'(F)}{h_Fb(F)}}^2 } ^{1/2}\E{P_1^2}^{1/2}+\E{\frac{1}{h_F^2}\abs{P_2}}+\E{\frac{2}{\abs{h_F}}}\abs{\E{a(F)}}\,. 
\end{align}

 Proposition \ref{prop_steinmethod_gendif} provides   a bound on the first term on the right hand side that is 
 \begin{align}
 \label{firstterm_generalfunctional}
   &\abs{\E{\mathds{1}_{\{F>x\}} \brac{\frac{-2a(F)+b'(F)}{b(F)}}- \mathds{1}_{\{\mathcal{Y}>x\}} \brac{\frac{-2a(\mathcal{Y})+b'(\mathcal{Y})}{b(\mathcal{Y})} }} }\nonumber\\
     &\qquad \leq  \brac{\E{\mathcal{V}_2(x,F)^2}^{1/2}+\E{\mathcal{V}_4(F)^2}^{1/2}+\E{\mathcal{V}_6(F)^2}^{1/2} }\E{P_1^2}^{1/2}. 
\end{align}
Regarding the second to last term on the right hand side of \eqref{prelimbound_generalfunctional}, we use $\Gamma(F,G)^2\leq \Gamma(F)\Gamma(G) $ to obtain
\begin{align*}
\E{\frac{1}{h_F^2}\abs{P_2}}&=\E{\frac{1}{h_F^2}\abs{\Gamma\brac{-L^{-1}2a(F),P_1 }}}\\
&\leq \E{\frac{1}{h_F^2}\Gamma(L^{-1}2a(F) )^{1/2}\Gamma(P_1)^{1/2}  }\\
&\leq \E{\frac{\Gamma(L^{-1}2a(F))}{h_F^4}}^{1/2}\E{\Gamma(P_1)}^{1/2}. 
\end{align*}
Moreover, 
\begin{align*}
    \E{\Gamma(P_1)}=\E{-P_1LP_1}\leq \E{(LP_1)^2}^{1/2}\E{P_1^2}^{1/2}  
\end{align*}
which leads to 
\begin{align}
\label{secondterm_generalfunctional}
    \E{\frac{1}{h_F^2}\abs{P_2}}&\leq \E{\frac{\Gamma(-L^{-1}2a(F) )}{h_F^4}}^{1/2}\E{(LP_1)^2}^{1/4}\E{P_1^2}^{1/4}. 
\end{align}
Combining \eqref{prelimbound_generalfunctional}, \eqref{firstterm_generalfunctional} and \eqref{secondterm_generalfunctional} yields 
\begin{align}
\label{estimatewhenxinellu}
           & \abs{p_F(x)-p_{\mathcal{Y}}(x)}\nonumber\\
       &\leq \brac{\E{\mathcal{V}_2(x,F)^2}^{1/2}+\E{\mathcal{V}_4(F)^2}^{1/2}+\E{\mathcal{V}_6(F)^2}^{1/2} }\E{P_1^2}^{1/2}\nonumber\\
        &+\E{\brac{\frac{2a(F)-b'(F)}{h_F b(F)}}^2 } ^{1/2}\E{P_1^2}^{1/2}\nonumber\\
        &+\E{\frac{\Gamma(-L^{-1}2a(F) )}{h_F^4}}^{1/2}\E{(LP_1)^2}^{1/4}\E{P_1^2}^{1/4}+\E{\frac{2}{\abs{h_F}}}\abs{\E{a(F)}}. 
\end{align}

\underline{Step  2}: $x<\ell$.\      
Via \eqref{intbyparts_gammaandL} and \eqref{relation_Landinverse}, we have
\begin{align*}
   &\E{ \frac{-2\E{a(F)}+2a(F)}{h_F}+\frac{\Gamma\brac{-L^{-1}2a(F),h_F }}{h_F^2}}\\
   &= \E{ \frac{-2\E{a(F)}+2a(F)}{h_F}+\Gamma\brac{-L^{-1}2a(F),-1/h_F} }\\
   &=\E{ \frac{-2\E{a(F)}+2a(F)}{h_F}+\frac{1}{h_F}L(-L^{-1})2a(F) }\\
   &=\E{ \frac{-2\E{a(F)}+2a(F)}{h_F}-\frac{1}{h_F}(2a(F)-2\E{a(F)})}=0,  
\end{align*}
which together with \eqref{generalfunctional_densitydecompose}  imply
\begin{align*}
    p_F(x)&=\E{\mathds{1}_{\{F>x\}}\brac{ \frac{-2\E{a(F)}+2a(F)}{h_F}+\frac{\Gamma\brac{-L^{-1}2a(F),h_F }}{h_F^2}} }\\
   &= \E{\brac{1-\mathds{1}_{\{F\leq x\}}}\brac{ \frac{-2\E{a(F)}+2a(F)}{h_F}+\frac{\Gamma\brac{-L^{-1}2a(F),h_F }}{h_F^2}} }\\
   &=-\E{\mathds{1}_{\{F\leq x\}}\brac{ \frac{-2\E{a(F)}+2a(F)}{h_F}+\frac{\Gamma\brac{-L^{-1}2a(F),h_F }}{h_F^2}} }\\
    &=-\mathbb{E}\bigg[\mathds{1}_{\{F\leq x\}} \bigg(\frac{-2a(F)+b'(F)}{b(F)}-\frac{2\E{a(F)}}{h_F} +\frac{2a(F)}{h_Fb(F)}P_1+\frac{1}{h_F^2}P_2-\frac{b'(F)}{b(F)h_F}P_1\bigg)\bigg]. 
\end{align*}
Moreover, since the support of $\mathcal{Y}$ is the interval $(\ell,u)$, we can write for $x<\ell$,
\begin{align*}
    p_{\mathcal{Y}}(x)&=-\E{\mathds{1}_{\{\mathcal{Y}\leq x\}} \brac{\rho_1(\mathcal{Y})}}=0. 
\end{align*}
Then we can estimate the quantity $\abs{p_F(x)-p_{\mathcal{Y}}(x)}$
for $x<\ell$ in the same way as Step  1. This will yield the same bound as \eqref{estimatewhenxinellu}.

\underline{Step  3}: $x>u$.\  
This part is similar to Step 2 with  
\begin{align*}
     p_{\mathcal{Y}}(x)&=-\E{\mathds{1}_{\{\mathcal{Y}> x\}} \brac{\rho_1(\mathcal{Y})}}=0, \quad x>u
\end{align*}
and yield the same bound on  $\abs{p_F(x)-p_{\mathcal{Y}}(x)}$: 
    \begin{align*}
           & \abs{p_F(x)-p_{\mathcal{Y}}(x)}\nonumber\\
       &\leq \brac{\E{\mathcal{V}_2(x,F)^2}^{1/2}+\E{\mathcal{V}_4(F)^2}^{1/2}+\E{\mathcal{V}_6(F)^2}^{1/2} }\E{P_1^2}^{1/2}\nonumber\\
        &+\E{\brac{\frac{2a(F)-b'(F)}{h_F b(F)}}^2 } ^{1/2}\E{P_1^2}^{1/2}\nonumber\\
        &+\E{\frac{\Gamma(-L^{-1}2a(F) )}{h_F^4}}^{1/2}\E{(LP_1)^2}^{1/4}\E{P_1^2}^{1/4}+\E{\frac{2}{\abs{h_F}}}\abs{\E{a(F)}}. 
\end{align*}
This completes the proof.  
\end{proof}

The following result is Theorem \ref{theorem_generaldiffusion} rewritten in sequential form.

\begin{corollary}
     Let $X_n,n\geq 1$ be a sequence of random variables such that each $X_n$ belongs to $\mathcal{A}$ has zero mean and admits a density. Set $F_n=X_n+m$ for $m=\E{\mathcal{Y}}$ and $h_{F_n}=\Gamma\brac{F_n,-L^{-1}2a(F_n)}$. Denote the density function of $F_n$ by $p_{F_n}$. Assume the moment bound
   \begin{align*}
&\sup_{n\geq 1}\Bigg\{\mathds{1}_{\{F_n\geq u \} }\Bigg(\E{\abs{\frac{-2a'(F_n)+b''(F_n)}{b(F_n)^2a(F_n)^2}} } +\E{\abs{\frac{(-2a(F_n)+b'(F_n))^2b'(F_n)^2}{b(F_n)^4a(F_n)^2}} }   \\
&\qquad+\E{\abs{\frac{-2a(F_n)+b'(F_n)}{b(F_n)^2a(F_n)^4}} }+\E{\abs{\frac{-2a(F_n)+b'(F_n)}{b(F_n)^2a(F_n)^2}} }\Bigg)\\
&\qquad\quad+\mathds{1}_{\{F_n\notin (\ell,u) \} }\E{\abs{a(F_n)}^{-4}}\\
&\qquad\quad\quad+\mathds{1}_{\{F_n\leq \ell \} }\Bigg(\E{\abs{\frac{-2a'(F_n)+b''(F_n)}{b(F_n)^2a(F_n)^2}} } +\E{\abs{\frac{(-2a(F_n)+b'(F_n))^2b'(F_n)^2}{b(F_n)^4a(F_n)^2}} }   \\
&\qquad\quad\quad\quad+\E{\abs{\frac{-2a(F_n)+b'(F_n)}{b(F_n)^2a(F_n)^4}} }+\E{\abs{\frac{-2a(F_n)+b'(F_n)}{b(F_n)^2a(F_n)^2}} }\Bigg)\\
&\qquad\quad\quad\quad\quad+\E{\Gamma\brac{-L^{-1}2a(F_n),F_n}^2}+\E{\brac{\frac{2a(F_n)-b'(F_n)}{h_{F_n} b(F_n)}}^2 }\\
&\qquad\qquad\qquad\qquad+\E{\frac{\Gamma(-L^{-1}2a(F_n) )}{h_{F_n}^4}}+\E{(L(h_{F_n}+b(F_n)))^2}+\E{\frac{2}{\abs{h_{F_n}}}}\\
&\qquad\quad+\mathds{1}_{\{F_n\in (\ell,u) \} }\brac{\E{\abs{\frac{a(F_n)}{b(F_n)}}^2}+\E{\frac{\abs{-2a(F_n)+b'(F_n)}^2}{\abs{b(F_n)}^4}}+\E{\abs{b(F_n)}^{-2}}}\Bigg\}<\infty; 
\end{align*}
and also that
\begin{align*}
&\E{\brac{\Gamma(-L^{-1}2a(F_n), h_F)+b'(F_n)h_{F_n}}^2}\to 0;\\
    &\E{\brac{h_{F_n}+b(F_n)}^2}\to 0;\quad\E{a(F_n)}\to 0. 
\end{align*} 
Then it holds for $x\in \R$ that 
\begin{align*}
    \lim_{n\to \infty}p_{F_n}(x)=p_{\mathcal{Y}}(x). 
\end{align*}

\end{corollary}

\begin{proof}
    The Corollary follows directly from Theorem \ref{theorem_generaldiffusion}. The only thing we need to spell out are sufficient conditions such that
\begin{align*}
    \E{\mathcal{V}_1(x,F)^2}+\E{\mathcal{V}_2(x,F)^2}+\sum_{i=3}^6 \E{\mathcal{V}_i(F)^2}<\infty \, . 
\end{align*}


    Per Lemma \ref{lemma_xinlandu_gendif}, $\E{\mathcal{V}_1(x,F_n)^2}<\infty$ as long as 
    \begin{align*}
        \mathds{1}_{\{F_n\geq u \} }\E{\abs{\frac{-2a(F_n)+b'(F_n)}{b(F_n)^2a(F_n)^2}} }+ \mathds{1}_{\{F_n\notin (\ell,u) \} }\E{\abs{a(F_n)}^{-2}}<\infty;
    \end{align*}
while $\E{\mathcal{V}_2(x,F_n)^2}<\infty$ as long as 
\begin{align*}
&\mathds{1}_{\{F_n\geq u \} }\bigg(\E{\abs{\frac{-2a'(F_n)+b''(F_n)}{b(F_n)^2a(F_n)^2}} } +\E{\abs{\frac{(-2a(F_n)+b'(F_n))^2b'(F_n)^2}{b(F_n)^4a(F_n)^2}} }   \\
&\qquad+\E{\abs{\frac{-2a(F_n)+b'(F_n)}{b(F_n)^2a(F_n)^4}} }\bigg)+\mathds{1}_{\{F_n\notin (\ell,u) \} }\E{\abs{a(F_n)}^{-4}}\\
&\qquad\quad+\mathds{1}_{\{F_n\in (\ell,u) \} }\brac{\E{\abs{\frac{a(F_n)}{b(F_n)}}^2}+\E{\frac{\abs{-2a(F_n)+b'(F_n)}^2}{\abs{b(F_n)}^4}}+\E{\abs{b(F_n)}^{-2}}}<\infty. 
\end{align*}

Next, per Lemma \ref{lemma_steinmethod_xlessthanl_gendif}, $\E{\mathcal{V}_3(F_n)^2}<\infty$ as long as 
\begin{align*}
    \mathds{1}_{\{F_n\leq  \ell \} }\E{\abs{\frac{-2a(F_n)+b'(F_n)}{b(F_n)^2a(F_n)^2}} }<\infty;
\end{align*}
while 
$\E{\mathcal{V}_4(F_n)^2}<\infty$ as long as 
\begin{align*}
&\mathds{1}_{\{F_n\leq  \ell \} }\bigg(\E{\abs{\frac{-2a'(F_n)+b''(F_n)}{b(F_n)^2a(F_n)^2}} } +\E{\abs{\frac{(-2a(F_n)+b'(F_n))^2b'(F_n)^2}{b(F_n)^4a(F_n)^2}} }   \\
&\qquad\qquad\qquad\qquad\qquad\qquad+\E{\abs{\frac{-2a(F_n)+b'(F_n)}{b(F_n)^2a(F_n)^4}} }\bigg)<\infty. 
\end{align*}
Finally, per Lemma \ref{lemma_steinmethod_xmorethanu_gendif}, $\E{\mathcal{V}_5(F_n)^2}<\infty$ as long as 
\begin{align*}
        \mathds{1}_{\{F_n\geq u \} }\E{\abs{\frac{-2a(F_n)+b'(F_n)}{b(F_n)^2a(F_n)^2}} }<\infty;
    \end{align*}
while $\E{\mathcal{V}_6(F_n)^2}<\infty$ as long as 
\begin{align*}
&\mathds{1}_{\{F_n\geq u \} }\bigg(\E{\abs{\frac{-2a'(F_n)+b''(F_n)}{b(F_n)^2a(F_n)^2}} } +\E{\abs{\frac{(-2a(F_n)+b'(F_n))^2b'(F_n)^2}{b(F_n)^4a(F_n)^2}} }   \\
&\qquad\qquad\qquad\qquad\qquad\qquad+\E{\abs{\frac{-2a(F_n)+b'(F_n)}{b(F_n)^2a(F_n)^4}} }\bigg)<\infty. 
\end{align*}
Now an application of Theorem \ref{theorem_generaldiffusion} proves the corollary. 
\end{proof}


\section{Pearson diffusion}	
\label{section_pearson}

\
 
The focus on this section is  the  Pearson  diffusion  \eqref{generaldiffusion},  where the function $a$ is linear and the function $b$ is quadratic. The stationary measure of a Pearson diffusion is a Pearson distribution.

Specifically, a Pearson diffusion is the   solution   to the stochastic differential equation
\begin{align}
\label{sde_pearson}
    dZ_t=\theta(m-Z_t)dt+\sqrt{\mathds{1}_{(\ell,u)}(Z_t)2\theta b(Z_t)}dB_t\,, 
\end{align}
where $b(x)=b_2x^2+b_1x+b_0$; $(\ell, u)\,,  \infty\leq \ell<u\leq \infty$ is an open interval; $m$ is the stationary mean and $\theta$ is the speed of mean reversion. Under some well-known conditions on the coefficients we know that the solution $Z_t$ lives 
in the interval $(\ell, u)$ so that in fact we have
$\mathds{1}_{(\ell,u)}(Z_t)=1$.  For convenience, we will assume $\theta=1/2$. Furthermore, we suppose that  
\begin{align}
\label{assumption_diffusion}
    &b(x)>0,\qquad \forall x\in (\ell,u)\,. 
\end{align}

Per the survey paper \cite{Sorensen}, the stationary density associated with \eqref{sde_pearson} is
    \begin{align}
    \label{densityanalyticformula_pearson}
    p_\mathcal{Z}(x)=\mathds{1}_{ (\ell,u)}(x)\frac{C}{b_2x^2+b_1x+b_0}\exp\brac{-\int_{x_0}^x\frac{y-m}{b^2y^2+b_1y+b_0} dx},
    \end{align}
where $C$ is the normalizing constant. Moreover, condition \eqref{def_driftb_gendif} becomes
\begin{align}
\label{relation_driftanddiffusion}
    b(x)= \frac{2\int_\ell^x (m-w)p_Z(w)dw}{p_Z(x)},\qquad \forall x\in (\ell,u)  
\end{align}
which is easy to verify.

For any $k\in\N$, set
\begin{align}
\label{def_rho}
 \rho_{k+1}(x):=   \frac{\sum_{j=0}^{k+1} c^{k+1}_j x^j }{b(x)^{k+1}} .
\end{align}
Here $\{c^k_j:k\geq 1,k\geq j\geq 0 \}$ is a family of real coefficients that satisfy 
\begin{align}
\label{recursiverelation_c}
    c^{k+1}_\ell=c^k_j(m+b_1(j-k-1) )-c^k_{j-1}\brac{(2k+3-j)b_2+1 }+b_0 c^k_{j+1}(j+1)\,,  
\end{align}
and $c_1^1=2b_2+1,c_0^1=b_1-m$. Moreover, we employ the convention that when $j<0$ or $j>k$, $c_j^k=0$.

\begin{lemma}
\label{lemma_densityrep_pearson}
   Assume $\mathcal{Z}$ is   the Pearson distribution  associated with the Pearson diffusion   
   \eqref{sde_pearson} (namely, the law of $\mathcal{Z}$  is given by \eqref{densityanalyticformula_pearson}). Then the density function $p_Z$ admits the representation
    \begin{align*}
    p_\mathcal{Z}(x)&=\begin{cases}\E{ \mathds{1}_{\{x\leq \mathcal{Z}  \}}\rho_{1}(\mathcal{Z})}, &\quad x>\ell\,, \\
    \E{ \mathds{1}_{\{x> \mathcal{Z}  \}}\rho_{1}(\mathcal{Z})},&\quad x\leq \ell\,, \end{cases}
\end{align*}
where $\rho_1(y)= \frac{(2b_2+1)y+(b_1-m)}{b(y)}$. Moreover, if $p_\mathcal{Z}(x)$ has continuous derivatives up to order $k+1$ and $p^{(k)}(u)=0$, then we have 
 \begin{align*}
    p^{(k)}_\mathcal{Z}(x)&=\begin{cases}\E{ \mathds{1}_{\{x\leq Z  \}}\rho_{k+1}(\mathcal{Z})}, &\quad x>\ell\,, \\
    \E{ \mathds{1}_{\{x> \mathcal{Z}  \}}\rho_{k+1}(\mathcal{Z})}, &\quad x\leq \ell\,, 
    \end{cases}
\end{align*}
where $\rho_k$'s satisfy the recursive relation
\begin{align*}
      \rho_{k+1}(x)=-\rho_{k }(x)\frac{x-m+b'(x) }{b(x)}+\rho'_{k }(x) \,,\quad k=1, 2, \cdots
\end{align*}
\end{lemma}

\begin{proof}
    The density representation of $\mathcal{Z}$ is just Lemma \ref{lemma_densityrep_generaldiffision} with $b(y)=b_2y+b_1y+b_0$. The density representation for the derivatives of the density is Part (iii) of Proposition \ref{prop_densityrepfortarget}.
\end{proof}

\subsection{Density representation for eigenfunctions}
~\
\label{section_densityeigenfunction}

Let $X$ be an eigenfunction of the generator $L$ corresponding to the eigenvalue $-q$ and $F=X+m$. Set $S_1:=\Gamma(F)$, $Q_1:=S_1-qb(F)$ and recursively for $n\geq 2$,  
\begin{align}
\label{def_SnandQ_n}
 S_n:=\Gamma(F,S_{n-1}),\qquad   Q_n:=\Gamma(F,Q_{n-1}). 
\end{align}

\begin{proposition}
\label{prop_densitydecompose}
    Let $X$ be an eigenfunction of the generator $L$ corresponding to the eigenvalue $-q$ and $F=X+m$. Assume $F\in \mathcal{A}$ and $S_1^{-1}=\Gamma(F)^{-1}\in L^{4N+4}(E,\mu)$ for some $N\in \N$. Then we have for $k\leq N$
    \begin{align*}
        p^{(k)}_F(x)=\E{\mathds{1}_{\{F>x\}}\brac{\rho_{k+1}(F)+T_{k+1}}}\,, 
    \end{align*}
where the function $\rho_{k+1}$ is defined by  \eqref{def_rho}. Moreover, the term $T_{k+1}$ has the form
\begin{align*}
    T_{k+1}=\sum_{P_{k+1}} Q_{i_1}\frac{c\brac{i_1,i_2,i_3,i_4,j_1,\ldots,j_{k+1},m_1,\ldots,m_{k+1}} F^{i_4}\prod_{r=1}^{k+1}S_{j_r}^{n_r} }{S_1^{i_2}b(F)^{i_3}}. 
\end{align*}
The coefficients $c\brac{i_1,i_2,i_3,i_4,j_1,\ldots,j_{k+1},n_1,\ldots,n_{k+1}}$ are real numbers; while $P_{k+1}$ is the set of multi-indices $\brac{i_1,i_2,i_3,i_4,j_1,\ldots,j_{k+1},n_1,\ldots,n_{k+1}}\in \N^{\otimes 4+2(k+1)}$ that satisfies the following set of conditions. 
\begin{enumerate}[label=\roman*)]
\item $1\leq i_1\leq k+2$;
\item $0\leq i_2\leq 2(k+1)$;
\item $0\leq i_3\leq k+1$;
\item $0\leq i_4\leq k+1$;
\item $2\leq j_r\leq k+1$ for $1\leq r\leq k+1$ and $j_r\neq j_s$ when $r\neq s$;
\item $0\leq n_r\leq k+1$ for $1\leq r\leq k+1$. 
\end{enumerate}
If a sum can be written like the right hand side of $T_{k+1}$, we say  it satisfies \textbf{Con}$-(k+1)$.
\end{proposition}

\begin{proof}
In the first step, we will verify the representation for $p_Z(x)$. According to Theorem \ref{theorem_herrydensity}, we have
\begin{align*}
    &p_F(x)=\E{\mathds{1}_{\{F>x\}}\frac{-S_1LF+S_2}{S_1^2} }\\
    &\quad=\E{\mathds{1}_{\{F>x\}}\brac{\rho_1(F)+\frac{-S_1LF+S_2}{S_1^2}-\frac{b'(F)+F-m}{b(F)}  } }\\
    &\quad=\E{\mathds{1}_{\{F>x\}}\brac{\rho_1(F)+ \frac{1}{S_1^2b(F)}\brac{qS_1(F-m)b(F)+S_2b(F)-S_1^2b'(F)-(F-m)S_1^2 }  } }. 
\end{align*}
In the last line, we have used the fact that $-LF=q(F-m)$. Next, we observe 
\begin{align*}
    &S_2b(F)-S_1^2b'(F)\\
    &=S_2b(F)-S_1qb(F)b'(F)+S_1qb(F)b'(F)-S_1^2b'(F)\\
    &=b(F)Q_2-S_1b'(F)Q_1
\end{align*}
and 
\begin{align*}
    &qS_1(F-m)b(F)-(F-m)S_1^2\\
    &=-S_1(F-m)Q_1
\end{align*}
in order to deduce that 
\begin{align}
\label{0thderivativedensity}
    p_F(x)&=\E{\mathds{1}_{\{F>x\}} \brac{\rho_1(F)+Q_2\frac{b_2F^2+b_1F+b_0}{S_1^2b(F)}-Q_1\frac{2b_2F+b_1+F-m}{S_1b(F)} } }\nonumber\\
   & =\E{\mathds{1}_{\{F>x\}} \brac{\rho_1(F)+\frac{Q_2}{S_1^2}-Q_1\frac{2b_2F+b_1+F-m}{S_1b(F)} } }
\end{align}
In the next step, we will verify the representation for $p^{(k)}_Z(x), k\geq 1$ via an induction argument. We assume as our induction hypothesis that the formula holds up to $k$, that is 
\begin{align*}
     p^{(k)}_F(x)=\E{\mathds{1}_{\{F>x\}}\brac{\rho_{k+1}(F)+T_{k+1}}}. 
\end{align*}
Then it follows from Lemma \ref{lemma_densityrep_pearson}  that 
\begin{align*}
    p_F^{(k+1)}(x)=\E{\mathds{1}_{\{F>x\}}\brac{\frac{\brac{\rho_{k+1}(F)+T_{k+1}}LF}{S_1}+\Gamma\brac{\frac{\rho_{k+1}(F)+T_{k+1}}{S_1},F } } }\,. 
\end{align*}
Since $-LF=q(F-m)$, we can write
\begin{align*}
    &\frac{\brac{\rho_{k+1}(F)+T_{k+1}}LF}{S_1}\\
    &=-\frac{(F-m)\rho_{k+1}(F) }{b(F)}+Q_1\frac{(F-m)\rho_{k+1}(F)}{S_1b(F)}-\frac{q(F-m)T_{k+1}}{S_1}. 
\end{align*}
Furthermore,   by the product rule and chain rule for $\Gamma$, 
\begin{align*}
    &\Gamma\brac{\frac{\rho_{k+1}(F)+T_{k+1}}{S_1},F }\\
    &=\brac{\rho_{k+1}(F)+T_{k+1} }\Gamma\brac{\frac{1}{S_1},F}+\frac{1}{S_1}\Gamma\brac{\rho_{k+1}(F)+T_{k+1},F } \\
    &=-\frac{S_2}{S_1^2}\rho_{k+1}(F)+\rho'_{k+1}(F)-\frac{S_2}{S_1^2}T_{k+1}+\frac{1}{S_1}\Gamma(T_{k+1},F)\\
    &=-\rho_{k+1}(F)\frac{b'(F)}{b(F)}-\rho_{k+1}(F)\brac{\frac{S_2}{S_1^2}-\frac{qb'(F)S_1}{S_1^2}+\frac{qb'(F)}{S_1}-\frac{b'(F)}{b(F)}}\\
    &\hspace{13em}+\rho'_{k+1}(F)-\frac{S_2}{S_1^2}T_{k+1}+\frac{1}{S_1}\Gamma(T_{k+1},F)\\
    &=-\rho_{k+1}(F)\frac{b'(F)}{b(F)}-\rho_{k+1}(F)\brac{\frac{Q_2}{S_1^2}-Q_1\frac{b'(F)}{S_1b(F)} }\\
    &\hspace{13em}+\rho'_{k+1}(F)-\frac{S_2}{S_1^2}T_{k+1}+\frac{1}{S_1}\Gamma(T_{k+1},F). 
\end{align*}
As a result,
\begin{align*}
    &p_F^{(k+1)}(x)\\
    &=\mathbb{E}\Bigg[\mathds{1}_{\{F>x\}}\Bigg(\brac{-\rho_{k+1}(F)\frac{F-m }{b(F)}+\rho'_{k+1}(F)-\rho_{k+1}(F)\frac{b'(F)}{b(F)}  }+   \\
    &-\rho_{k+1}(F)\brac{\frac{Q_2}{S_1^2}-Q_1\frac{b'(F)}{S_1b(F)} } +\rho'_{k+1}(F)-\frac{S_2}{S_1^2}T_{k+1}+\frac{1}{S_1}\Gamma(T_{k+1},F)\\
    &+Q_1\frac{(F-m)\rho_{k+1}(F)}{S_1b(F)}-\frac{q(F-m)T_{k+1}}{S_1}\Bigg)\Bigg]. 
\end{align*}
Per Lemma \ref{lemma_densityrep_pearson}, we have
\begin{align*}
    -\rho_{k+1}(F)\frac{F-m }{b(F)}+\rho'_{k+1}(F)-\rho_{k+1}(F)\frac{b'(F)}{b(F)} =\rho_{k+2}(F). 
\end{align*}
What remains of our induction argument is to show that the sum 
\begin{align*}
    \mathcal{A}:= &-\rho_{k+1}(F)\brac{\frac{Q_2}{S_1^2}-Q_1\frac{b'(F)}{S_1b(F)} } -\frac{S_2}{S_1^2}T_{k+1}+\frac{1}{S_1}\Gamma(T_{k+1},F)\\
    &+Q_1\frac{(F-m)\rho_{k+1}(F)}{S_1b(F)}-\frac{q(F-m)T_{k+1}}{S_1}
\end{align*}
satisfies \textbf{Con}$-(k+2)$. We will only verify $\frac{1}{S_1}\Gamma(T_{k+1},F)$ can be written as a sum satisfying \textbf{Con}$-(k+2)$, noting that the other terms can be handled similarly. In the upcoming calculation, we will write $c$ as the short form of $c\brac{i_1,i_2,i_3,i_4,j_1,\ldots,j_{k+1},m_1,\ldots,m_{k+1}}$.  The chain rule plus product rule for $\Gamma$ imply that
\begin{align*}
    &\frac{1}{S_1}\Gamma(T_{k+1},F)\\
    &=\frac{1}{S_1}\Gamma\brac{\sum_{P_{k+1}} Q_{i_1}\frac{c F^{i_4}\prod_{r=1}^{k+1}S_{j_r}^{n_r} }{S_1^{i_2}b(F)^{i_3}},F}\\
    &=\sum_{P_{k+1}}Q_{i_1+1}\frac{c F^{i_4}\prod_{r=1}^{k+1}S_{j_r}^{n_r} }{S_1^{i_2+1}b(F)^{i_3}}+\sum_{P_{k+1}}Q_{i_1}\frac{ci_4 F^{i_4-1}\prod_{r=1}^{k+1}S_{j_r}^{n_r} }{S_1^{i_2}b(F)^{i_3}}-\sum_{P_{k+1}}Q_{i_1}\frac{c i_2F^{i_4}S_2\prod_{r=1}^{k+1}S_{j_r}^{n_r} }{S_1^{i_2}b(F)^{i_3}}\\
    &-\sum_{P_{k+1}}Q_{i_1}\frac{c F^{i_4}\prod_{r=1}^{k+1}S_{j_r}^{n_r} \brac{2b_2F+b_1}}{S_1^{i_2}b(F)^{i_3+1}}+\sum_{P_{k+1}}Q_{i_1}\frac{c F^{i_4}\brac{\prod_{\substack{1\leq r\leq k+1;r\neq q}} S_{j_r}^{n_r}}S_{j_q}^{n_q-1}S_{j_q+1} }{S_1^{i_2}b(F)^{i_3}}. 
\end{align*}
Based on this calculation, we can see $\frac{1}{S_1}\Gamma(T_{k+1},F)$ is indeed a sum satisfying \textbf{Con}$-(k+2)$. This completes our induction proof.  
\end{proof}


\subsection{Stein's method}
\label{section_steinmethod_pearson}
~\

In order to bound the difference between an arbitrary distribution and the  target Pearson density,  in this section we study the solution of the Stein's equation \eqref{equation_stein}
in the special case that $a(y)=\frac{1}{2}(m-y)$ and $b(y)=b_2y^2+b_1y+b_0$, that is 
\begin{align}
\label{equation_stein_pearson}
  (m-y)g(y)+\mathds{1}_{\{y\in (\ell,u) \} }(b_2y^2+b_1y+b_0)g'(y)=h(y)-\E{h(\mathcal{Y})}. 
\end{align}

The content of this section is fairly similar to Section \ref{section_steinmethod_gendif}, except that here we consider a special case that $b(y)=b_2y^2+b_1y+b_0$ but more complicated test functions $h$ which will allow us to obtain convergence of the higher derivatives of the density functions.

In the upcoming lemmas, we will derive estimates on the solution to the Stein's equation and its first derivative. $C$ will denote a positive constant that may  change from line to line.
\begin{lemma}
\label{lemma_xinlandu}
Fix  $x\in (\ell, u)$ and $k\in \N$. Let the function $h$ in the ODE \eqref{equation_stein} be 
    \begin{align*}
        h(y)=h_{k,x}(y):=\mathds{1}_{\{y>x\}}\rho_{k+1}(y),\qquad y\in \R\setminus \left\{y\in [u,\infty):b(y)=0 \right\}\,, 
    \end{align*} 
where the function $\rho$ is defined by  \eqref{def_rho}. Then the solution $g_{k,x}:\R\to\R$ to \eqref{equation_stein_pearson} exists and satisfies for every $y\in \R$, 

\begin{align*}
    & \abs{g_{k,x}(y)}\\
    &\leq \mathcal{U}_1(k,x,y):=\mathds{1}_{ \{y\geq u \} }C\frac{\sum_{j=0}^{k+1}\abs{y}^j }{\abs{b(y)}^{k+1}\abs{y-m} }+\mathds{1}_{ \{y\notin (\ell,u) \} }\frac{p_Z^{(k)}(x) }{\abs{y-m}}+\mathds{1}_{\{y\in(\ell,u) \}}C(x);
\end{align*}
and 
\begin{align*}
    &\abs{g'_{k,x}(y)}\\
    &\leq \mathcal{U}_2(k,x,y):=\mathds{1}_{\{y\in(\ell,u) \}}\brac{C(x)\frac{\abs{y-m}}{\abs{b(y)}}+\frac{p_\mathcal{Z}^{(k)}(x)}{\abs{b(y)}}+C'\frac{\sum_{j=0}^{k+1}\abs{y}^j}{\abs{b(y)}^{k+1}} }\\
    &\qquad\qquad\qquad\qquad\qquad+\mathds{1}_{ \{y\geq u \} }C'\frac{\sum_{j=0}^{k+1}\abs{y}^j }{\abs{b(y)}^{k+1}\abs{y-m}^2 }+\mathds{1}_{ \{y\notin (\ell,u) \} }C'\frac{p_Z^{(k)}(x) }{\abs{y-m}^2}\,,  
\end{align*} 
\end{lemma}
where $C(x)>0$ is a function  defined on $(\ell,u)$ and does not depend on $y$; and $C'>0$ is a constant independent of $x,y$. 

\begin{proof} \underline{Step 1:}  bound for  $g_{k,x}$. 
Assumption \eqref{assumption_diffusion} implies that the domain of $h$ is $\R\setminus \left\{y\in [u,\infty):b(y)=0 \right\}$.  
When $x\in (\ell,u)$, we know that $\E{h(Z)}=p_Z^{(k)}(x)$ so that
\begin{align}
\label{steinsolution_expand}
    g_{k,x}(y)&=\mathds{1}_{ \{y\in(\ell,u) \} }\brac{-\frac{1}{b(y)p_Z(y)}\int_{y}^u h(w)p_Z(w)dw+\frac{p^{(k)}_Z(x)}{b(y)p_Z(y)}\int_{y}^u p_Z(w)dw}\nonumber\\
    &\qquad +\mathds{1}_{ \{y\notin(\ell,u) \} } \frac{h(y)-p_Z^{(k)}(x)}{y-m} . 
\end{align}
The domain of $g_{k,x}$ is $\R$ due to the assumption \eqref{assumption_diffusion} and $m\in (\ell,u)$. 
%
It is easy to see 
\begin{align}
\label{steinsolution_easyterm}
    \abs{\mathds{1}_{ \{y\notin(\ell,u) \} }\frac{h(y)-p_Z^{(k)}(x)}{y-m}} \leq \mathds{1}_{ \{y\geq u \} }C'\frac{\sum_{j=0}^{k+1}\abs{y}^j }{\abs{b(y)}^{k+1}\abs{y-m} }+\mathds{1}_{ \{y\notin (\ell,u) \} }\frac{p_Z^{(k)}(x) }{\abs{y-m}}. 
\end{align}
In particular, we have $y\geq u$ for the first term on the right hand side as a combination of $y\notin(\ell,u)$,$x\in (\ell,u)$ and $y> x$.

Next, let us set 
\begin{align*}
    \mathcal{A}(y):= -\frac{1}{b(y)p_Z(y)}\int_{y}^u h(w)p_Z(w)dw+\frac{p^{(k)}_Z(x)}{b(y)p_Z(y)}\int_{y}^u p_Z(w)dw,\qquad y\in (\ell,u)\,. 
\end{align*}
The argument to bound $\mathcal{A}(y)$ follows the same line as the one in the proof of Lemma \ref{lemma_xinlandu_gendif} and gives
\begin{align}
\label{uniformbound_pearson}
    \sup_{y\in (\ell,u)} \abs{\mathcal{A}(y)}\leq C(x)
\end{align}
for a function $C(x)$ that is well defined and positive on $(\ell,u)$.  
A combination of \eqref{steinsolution_expand}, \eqref{steinsolution_easyterm} and \eqref{uniformbound_pearson} yields the desired bound on $g$. 
~\\

\underline{Step  2:} bounding $g'$.  
Based on \eqref{equation_stein_pearson}, we can write
\begin{align}
\label{steinderivative}
    g'(y)&=\mathds{1}_{\{y\in(\ell,u) \}}\brac{-\frac{y-m}{b(y)} g(y) +\frac{  h(y)-p_\mathcal{Z}^{(k)}(x)}{ b(y)}}\nonumber\\
    &\qquad +\mathds{1}_{ \{y\notin(\ell,u) \} } \frac{h'(y)(y-m)-(h(y)-p_\mathcal{Z}^{(k)}(x))}{(y-m)^2}. 
\end{align}
Regarding the first term on the right hand side of \eqref{steinderivative}, we use the bound on $g(y),y\in (\ell,u)$ from Step  1 to write 
\begin{align*}
&\abs{\mathds{1}_{\{y\in(\ell,u) \}}\brac{-\frac{y-m}{b(y)} g(y) +\frac{  h(y)-p_\mathcal{Z}^{(k)}(x)}{ b(y)}}}\\
&\qquad \leq \mathds{1}_{\{y\in(\ell,u) \}}\brac{C(x)\frac{\abs{y-m}}{\abs{b(y)}}+\frac{p_\mathcal{Z}^{(k)}(x)}{\abs{b(y)}}+C'\frac{\sum_{j=0}^{k+1}\abs{y}^j}{\abs{b(y)}^{k+1}} }\,, 
\end{align*}
where the function $C(x)$ is as in Step  1, and $C'>0$ is a constant independent of $x,y$. 

A bound on the second term on the right hand side of \eqref{steinderivative} can be derived similarly to \eqref{steinsolution_easyterm}, noting that 
\begin{align*}
    \abs{h'(y)}&\leq \mathds{1}_{ \{y\geq x\} }C'\brac{\frac{\sum_{j=1}^{k+1}\abs{y}^{j-1} }{\abs{b(y)}^{k+1}}+\sum_{j=0}^{k+1}\frac{\abs{y}^{j+1}+\abs{y}^j}{\abs{b(y)}^{k+2}} }\\
    &\leq \mathds{1}_{ \{y\geq x\} }C'\sum_{j=0}^{k+1}\frac{\abs{y}^{j+1}+\abs{y}^j}{\abs{b(y)}^{k+2}} . 
\end{align*}
Based on the previous calculations, we arrive at the desired bound on $g'(y)$.  
\end{proof}

\begin{lemma}
\label{lemma_steinmethod_xlessthanl}
  Fix  $x\leq \ell$ and $k\in \N$. Let the function $h$ in the ODE \eqref{equation_stein} be 
    \begin{align*}
        h(y)=h_{k,x}(y):=\mathds{1}_{\{y\leq x\}}\rho_{k+1}(y), \qquad y\in \R \{ y\in (-\infty,x]:b(y)=0\}. 
    \end{align*} 
Then the solution $g_{k,x}:\R\to\R$ to \eqref{equation_stein_pearson} exists and satisfies for every $y\in \R$
\begin{align*}
     \abs{ g_{k,x}(y)}&\leq \mathcal{U}_3(k,y):= C\mathds{1}_{\{y\leq \ell\}}  \frac{\sum_{j=0}^{k+1} \abs{y}^j }{\abs{b(y)}^{k+1}\abs{y-m} }
    \end{align*}
and
\begin{align*}
    \abs{ g'_{k,x}(y)}&\leq \mathcal{U}_4(k,y)\\
    &:= C \mathds{1}_{\{y\leq \ell\} }  \brac{\frac{\sum_{j=0}^{k+1} \abs{y}^j }{\abs{b(y)}^{k+1}\abs{y-m}^2 } +\frac{\sum_{j=0}^{k} \abs{y}^j }{\abs{b(y)}^{k+1}\abs{y-m} }+\frac{\abs{b'(y)}\sum_{j=0}^{k+1} \abs{y}^j }{\abs{b(y)}^{k+2}\abs{y-m} }}. 
\end{align*}

\end{lemma}

\begin{proof}
    The fact that $x\leq \ell$ and $\operatorname{supp}(Z)\subseteq (\ell,u)$ imply $\E{h(Z)}=0$. Then it follows from \eqref{solution_steinequation} that 
    \begin{align*}
        g_{k,x}(y)&=\mathds{1}_{\{y\notin (\ell,u)\}}\frac{h(y)}{y-m}
    \end{align*}
    and hence
    \begin{align*}
     \abs{ g_{k,x}(y)}&\leq  C \mathds{1}_{\{y\leq \ell\} }  \frac{\sum_{j=0}^{k+1} \abs{y}^j }{\abs{b(y)}^{k+1}\abs{y-m} }. 
    \end{align*}
From \eqref{steinderivative}, we have 
\begin{align*}
    g'_{k,x}(y)&=\mathds{1}_{ \{y\notin(\ell,u) \} } \frac{h'(y)(y-m)-h(y)}{(y-m)^2}. 
\end{align*}
$g'_{k,x}(y)$ can then be bounded in a similar way as what was done for $g_{k,x}$.  
\end{proof}

The next lemma can be proved in the same way as   the previous lemma.

\begin{lemma}
\label{lemma_steinmethod_xmorethanu}
  Fix $x\geq u$ and $k\in \N$. Let the function $h$ in the ODE \eqref{equation_stein} be 
    \begin{align*}
        h(y)=h_{k,x}(y):=\mathds{1}_{\{y> x\}}\rho_{k+1}(y), \qquad y\in  
         \{ y\in (x,\infty):b(y)=0\}. 
    \end{align*} 
Then the solution $g_{k,x}:\R\to\R$ to \eqref{equation_stein_pearson} exists and satisfies for every $y\in \R$, 
\begin{align*}
     \abs{ g_{k,x}(y)}&\leq  \mathcal{U}_5(k,y):=\mathds{1}_{\{y>u\}} C\frac{\sum_{j=0}^{k+1} \abs{y}^j }{\abs{b(y)}^{k+1}\abs{y-m} }
    \end{align*}
and
\begin{align*}
    \abs{ g'_{k,x}(y)}&\leq  \mathcal{U}_6(k,y)]]\\
    &:=\mathds{1}_{\{y>u\} } C\brac{\frac{\sum_{j=0}^{k+1} \abs{y}^j }{\abs{b(y)}^{k+1}\abs{y-m}^2 } +\frac{\sum_{j=0}^{k} \abs{y}^j }{\abs{b(y)}^{k+1}\abs{y-m} }+\frac{\abs{b'(y)}\sum_{j=0}^{k+1} \abs{y}^j }{\abs{b(y)}^{k+2}\abs{y-m} }}. 
\end{align*}

\end{lemma}


As the main result of this section, we rely on the previous lemmas to derive a Stein's bound for the Pearson distribution  $\mathcal{Z}$. 
\begin{proposition}
\label{prop_steinmethod}
Let $X$ be a centered random variable in $\mathcal{A}$ and $F=X+m$. Fix $x\in \R$.  Assume further that the following moment assumption is satisfied.  
\begin{align*}
\E{\mathcal{U}_1(k,x,F)^2}+\E{\mathcal{U}_2(k,x,F)^2}+\sum_{3\leq i\leq 6}\E{\mathcal{U}_i(k,F)^2}<\infty\,,  
\end{align*}
where $U_i,1\leq i\leq 6$ are   defined in Lemmas \ref{lemma_xinlandu}, \ref{lemma_steinmethod_xlessthanl} and \ref{lemma_steinmethod_xmorethanu}.  
For any $k\in \N$, define  
\[ h(y)=h_k(y)=\begin{cases} 
    \mathds{1}_{\{y\leq x\}}\rho_{k+1}(y) &\text{ when } x\leq \ell\,,  \\
      \mathds{1}_{\{y> x\}}\rho_{k+1}(y) &\text{ when } x>\ell\,. 
   \end{cases}\] 
Then,  we have the estimate 
\begin{align*}
   &\abs{ \E{h\brac{F}}-\E{h\brac{\mathcal{Z}}}}\\
     &\leq  \brac{\E{\mathcal{U}_2(k,x,F)^2}^{1/2}+\E{\mathcal{U}_4(k,F)^2}^{1/2}+\E{\mathcal{U}_6(k,F)^2}^{1/2} }\\
   &\hspace{16em}\times \E{\brac{ b(F)+  \Gamma\brac{L^{-1}F,F}   }^2 }^{1/2}. 
\end{align*}
\end{proposition}
\begin{proof}
The proof here is the same as that of Proposition \ref{prop_steinmethod_gendif}, with the use of Lemmas \ref{lemma_xinlandu}, \ref{lemma_steinmethod_xlessthanl} and \ref{lemma_steinmethod_xmorethanu} instead of Lemmas \ref{lemma_xinlandu_gendif}, \ref{lemma_steinmethod_xlessthanl_gendif} and \ref{lemma_steinmethod_xmorethanu_gendif}.  
\end{proof}

\begin{remark}
    Similar results about Stein's equation for the purpose of density approximation have been obtained in \cite{HLN14,BDH24} for normal and Gamma distributions. In particular, in the case where the Pearson distribution  is a normal distribution, \cite[Section 2.3]{HLN14} provides bounds on the solution of the Stein's equation that is uniform with respect to $x\in \R$, which allows them to obtain density bound in the uniform metric later on. In the case where the Pearson distribution is a Gamma distribution with parameter $\alpha>0$, \cite[Section 4]{BDH24} provides bounds on the solution of the Stein's equation that depends on $\alpha$, and in particular do not depend on the degree of the derivative of the density functions. Hence, the bounds in \cite{HLN14,BDH24} are better than the bounds we have obtained in the current section. Nevertheless, this is not surprising   given the level of generality of this section, and one would expect that if a single Pearson target is being studied at a time, one should be able to develop a tailored argument for such target that leads to better bounds.
\end{remark}


\subsection{Distance between the density of an eigenfunction and  Pearson density} 

\begin{proposition}
    \label{prop_estimatecarreduchamp_0thderivative}
Let $\mathcal{Z}$ be the Pearson  distribution  assosiated with \eqref{sde_pearson} with mean $m$. Suppose that $X\in \mathcal{A}$  is an eigenfunction of the generator $L$ corresponding to the eigenvalue $-q$. Set   $F=X+m$ and recall the definitions of $S_n,Q_n$ at \eqref{def_SnandQ_n}. Assume that  $S_1^{-1}\in L^{4}(E,\mu)$ and that 
\begin{align*}
\E{\mathcal{U}_1(0,x,F)^2}+\E{\mathcal{U}_2(0,x,F)^2}+\sum_{3\leq i\leq 6}\E{\mathcal{U}_i(0,F)^2}<\infty\,,  
\end{align*}
where $U_i(0),1\leq i\leq 6$ are from in Lemmas \ref{lemma_xinlandu}, \ref{lemma_steinmethod_xlessthanl} and \ref{lemma_steinmethod_xmorethanu}. \
Then $F$ admits a density and we have the following estimate at every $x\in \R$:   
\begin{align*}
     &\abs{p_F(x)-p_\mathcal{Z}(x)}\\
     &\leq \E{\frac{1}{S_1^3}}^{1/2}\E{\brac{LQ_1}^2}^{1/4}\E{Q_1^2}^{1/4}+\E{ \brac{\frac{2b_2F+b_1+F-m}{S_1b(F)}}^2}^{1/2}\E{Q_1^2}^{1/2}\\
&+\brac{\E{\mathcal{U}_2(0,x,F)^2}^{1/2}+\E{\mathcal{U}_4(0,F)^2}^{1/2}+\E{\mathcal{U}_6(0,F)^2}^{1/2} }\frac{1}{q^2}\E{Q_1^2 }^{1/2}. 
\end{align*}

\end{proposition}

\begin{proof} 

The assumptions $X\in \mathcal{A}$ and $S_1^{-1}\in L^{4}(E,\mu)$ imply $F$ admits a density per Theorem \ref{theorem_herrydensity}. Now to bound $\abs{p_F(x)-p_\mathcal{Z}(x)}$, we split the proof into several steps. 

~\ \underline{Step  1}: $x\in (\ell,u)$. 
 Under the assumption $F\in \mathcal{A}$ and $S_1^{-1}\in L^{4}(E,\mu)$, we have previously obtained at \eqref{0thderivativedensity} that

 
\begin{align*}
    p_F(x)=\E{\mathds{1}_{\{F>x\}} \brac{\rho_1(F)+\frac{Q_2}{S_1^2}-Q_1\frac{2b_2F+b_1+F-m}{S_1b(F)} } }  
\end{align*}
which implies
\begin{align}
\label{step_decompose}
    &\abs{p_F(x)-p_\mathcal{Z}(x)}\nonumber\\
    &=\abs{\E{\mathds{1}_{\{F>x\}}Q_1\frac{2b_2F+b_1+F-m}{S_1b(F)} }}+\abs{\E{\mathds{1}_{\{F>x\}}\rho_1(F)-\mathds{1}_{\{\mathcal{Z}>x\}}\rho_1(\mathcal{Z}) }} \nonumber\\
    &\qquad\qquad\qquad\qquad\qquad\qquad\qquad\qquad\qquad+\abs{\E{\mathds{1}_{\{F>x\}}\frac{Q_2}{S_1^2}}}\nonumber\\
    &\leq \E{ \brac{\frac{2b_2F+b_1+F-m}{S_1b(F)}}^2}^{1/2}\E{Q_1^2}^{1/2}+\abs{\E{\mathds{1}_{\{F>x\}}\rho_1(F)-\mathds{1}_{\{\mathcal{Z}>x\}}\rho_1(\mathcal{Z}) }}\nonumber\\
    &\qquad\qquad\qquad\qquad\qquad\qquad\qquad\qquad\qquad\qquad\qquad\qquad+\E{\abs{\frac{Q_2}{S_1^2}} }. 
\end{align}

We will leave the first term on the right hand side unchanged. The  second term has been bounded in Proposition \ref{prop_steinmethod} via Stein's method as 
\begin{align*}
&\abs{\E{\mathds{1}_{\{F>x\}}\rho_1(F)-\mathds{1}_{\{\mathcal{Z}>x\}}\rho_1(\mathcal{Z})}}\\
       &\leq  \brac{\E{\mathcal{U}_2(0,x,F)^2}^{1/2}+\E{\mathcal{U}_4(0,F)^2}^{1/2}+\E{\mathcal{U}_6(0,F)^2}^{1/2} }\frac{1}{q^2}\E{Q_1^2 }^{1/2}. 
\end{align*} 
To bound the third term on the right hand side of \eqref{step_decompose}, we recall that $\Gamma$ is a positive
definite  operator (\cite[(1.4.3)]{bakry2014analysis}) which implies 
\begin{align*}
    \Gamma(F,G)^2\leq \Gamma(F)\Gamma(G). 
\end{align*}
Therefore,  
\begin{align*}
    \E{\abs{\frac{Q_2}{S_1^2}} }=\E{\abs{\frac{\Gamma(F,Q_1)}{S_1^2}} }&\leq  \E{\frac{S_1^{1/2}\Gamma(Q_1)^{1/2} }{S_1^2}}\leq \E{\frac{1}{S_1^3}}^{1/2} \E{\Gamma(Q_1) }^{1/2}. 
\end{align*} 
    At this point, we can use the integration by part formula \eqref{intbyparts_gammaandL}   for  $\Gamma$ to get
\begin{align*}
    \E{\Gamma(Q_1) }
     &\leq \E{-Q_1LQ_1}\nonumber\\
    &\leq \E{\brac{LQ_1}^2}^{1/2}\E{Q_1^2}^{1/2}, 
\end{align*}
so that
\begin{align}
\label{step_boundW2overW1square}
    \E{\abs{\frac{Q_2}{S_1^2}} }
    &\leq \E{\frac{1}{S_1^3}}^{1/2}\E{\brac{LQ_1}^2}^{1/4}\E{Q_1^2}^{1/4} . 
\end{align}

Combining the previous calculations yields the desired estimate in the case $x\in (\ell,u)$.

\underline{Step  2}: $x<\ell$.

We have 
\begin{align*}
  \E{\frac{-S_1LF+S_2}{S_1^2}}  &=\E{ \frac{-LF}{S_1^2}}+\E{\frac{\Gamma(F,S_1)}{S_1^2}}\\
  &=\E{ \frac{-LF}{S_1}}+\E{\Gamma\brac{F,\frac{-1}{S_1}}}=\E{ \frac{-LF}{S_1}}+\E{ \frac{LF}{S_1}}=0
\end{align*}
which implies
\begin{align*}
    p_F(x)&=\E{\mathds{1}_{\{F>x\}}\frac{-S_1LF+S_2}{S_1^2} }\\
    &=\E{\brac{1-\mathds{1}_{\{F\leq x\}}}\frac{-S_1LF+S_2}{S_1^2} }\\
    &=-\E{\mathds{1}_{\{F\leq x\}}\frac{-S_1LF+S_2}{S_1^2} }=-\E{\mathds{1}_{\{F\leq x\}}\brac{\rho_1(F)+T_1 } }. 
\end{align*}
The last line is due to Proposition \ref{prop_densitydecompose}. 

Moreover, since the support of $\mathcal{Z}$ is the interval $(\ell,u)$, we can write for $x<\ell$,
\begin{align*}
    p_{\mathcal{Z}}(x)&=-\E{\mathds{1}_{\{\mathcal{Z}\leq x\}} \brac{\rho_1(\mathcal{Z})}}=0. 
\end{align*}
Then we can estimate the quantity
\begin{align*}
    &\abs{p_F(x)-p_{\mathcal{Z}}(x)}=\abs{\E{\mathds{1}_{\{F\leq x\}}\brac{\rho_1(F)+T_1 }}-\E{\mathds{1}_{\{\mathcal{Z}\leq x\}} \brac{\rho_1(\mathcal{Z})}}  }
\end{align*}
for $x<\ell$ in the same way as Step  1 for $x\in (\ell,u)$.

\underline{Step  3}: $x>u$. This case can be handled 
in a way   similar   to Step  2. 
\end{proof}

In fact, if the $k$-th derivative of the density of the eigen-funciton $F$ exists, it is possible to get a carr\'{e} du champ estimate on the $k$-th derivative of the densities, as the next result   shows. 
\begin{proposition}
    \label{prop_estimatecarreduchamp_kthderivative}
Assume  $q, N\in \N$. Let $\mathcal{Z}$  be  the  Pearson  distribution  associated with \eqref{sde_pearson} that has mean $m$, and satisfy   $\lim_{x\to u^{-}} p^{(k)}(x)=0, \forall k\leq N$ and admit densities that are differentiable up to the $N$-th order. Suppose that $X$ belongs to $\mathcal{A}$, and is an eigenfunction of the generator $L$ corresponding to the eigenvalue $-q$ and admits a density differentiable up to the $N$-th order. Set $F=X+m$ and recall the definitions of $S_n,Q_n$ at \eqref{def_SnandQ_n}. Assume that  $S_1^{-1}\in L^{4N+4}(E,\mu)$ and that 
\begin{align*}
\E{\mathcal{U}_1(N,x,F)^2}+\E{\mathcal{U}_2(N,x,F)^2}+\sum_{3\leq i\leq 6}\E{\mathcal{U}_i(N,F)^2}<\infty\,,  
\end{align*}
where $U_i,1\leq i\leq 6$ are the functions defined in Lemmas \ref{lemma_xinlandu}, \ref{lemma_steinmethod_xlessthanl} and \ref{lemma_steinmethod_xmorethanu}. \    
Then $F$ admits a density that is differentiable up to the $N$-th order. Moreover, for $k\leq N$, the following pointwise density estimate holds for $x\in\R$:   
    \begin{align*}
      &\abs{p_F^{(k)}(x)-p_\mathcal{Z}^{(k)}(x)}\\
      &\leq   \brac{\E{\mathcal{U}_2(k,x,F)^2}^{1/2}+\E{\mathcal{U}_4(k,F)^2}^{1/2}+\E{\mathcal{U}_6(k,F)^2}^{1/2} }\frac{1}{q^2} \E{Q_1 ^2}^{1/2}
      \\
     &+\sum_{P_{k+1}}\E{\abs{\frac{c\brac{i_1,i_2,i_3,i_4,j_1,\ldots,j_{k+1},n_1,\ldots,n_{k+1}}{\prod_{r=1}^{k+1}S_{j_r}^{n_r} } F^{i_4}}{S_1^{i_2}b(F)^{i_3} }}^2 S_1}^{1/2} \\ &\hspace{10em}\brac{\prod_{m=2}^{i_1-1}\E{\abs{LQ_m}^2 S_1}^{\frac{1}{2^{i_1+1-m}}}}\E{(LQ_1)^2}^{\frac{1}{2^{i_1-4}}}\E{Q_1^2}^{\frac{1}{2^{i_1-4}}}\,,  
    \end{align*}
where  $P_{k}$ is the set of multi-indices specified in Proposition \ref{prop_densitydecompose}. 
\end{proposition}
%

\begin{proof}
The assumptions $F\in \mathcal{A}$ and $S_1^{-1}\in L^{4N+4}(E,\mu)$ imply $F$ admits a density differentiable up to the $N$-th order per Theorem \ref{theorem_herrydensity}. Now we split the proof into several steps. 
~\ 
\underline{Step  1}: $x\in(\ell,u)$.  \
    Based on Proposition \ref{prop_densityrepfortarget} and Proposition \ref{prop_densitydecompose} (the latter requiring that $F\in \mathcal{A}$ and $S_1^{-1}\in L^{4N+4}(E,\mu)$), we can write
    \begin{align}
    \label{prelimbound_kthderivative}
         \abs{p_F^{(k)}(x)-p_\mathcal{Z}^{(k)}(x)} =  \E{ \mathds{1}_{\{F>x\}}\brac{\rho_{k+1}\brac{F}-\rho_{k+1}\brac{\mathcal{Z}}}} +\E{\mathds{1}_{\{F>x\}}T_{k+1} }.
    \end{align}

The first term on the right hand side has been bounded in Proposition \ref{prop_steinmethod} via Stein's method as follows.
\begin{align}
\label{boundkthderivative_firstterm}
    &\sup_{x\in\R}\abs{ \E{ \mathds{1}_{\{F>x\}}\brac{\rho_{k+1}\brac{F}-\rho_{k+1}\brac{\mathcal{Z}}}} }\nonumber\\
    &\leq\brac{\E{\mathcal{U}_2(k,x,F)^2}^{1/2}+\E{\mathcal{U}_4(k,F)^2}^{1/2}+\E{\mathcal{U}_6(k,F)^2}^{1/2} }\frac{1}{q^2}\E{Q_1^2 }^{1/2}. 
\end{align}

What remains is then to study the term $\E{\mathds{1}_{\{F>x\}}T_{k+1} }$. Proposition \ref{prop_densitydecompose} states 
\begin{align}
\label{step_estimateremainderterm}
    \abs{{\E{\mathds{1}_{\{F>x\}}T_{k+1} }}}\leq \sum_{P_{k+1}}\E{\abs{\frac{c \prod_{r=1}^{k+1}S_{j_r}^{n_r}F^{i_4}} {S_1^{i_2} b(F)^{i_3} } Q_{i_1}}}, 
\end{align}
where noting that we write $c$ in place of $c\brac{i_1,\ldots, i_{k+1}, j_1,\ldots,j_{k+4}}$ to lighten the notation. Observe that any individual term on the right hand side of \eqref{step_estimateremainderterm} has the form
\begin{align*}
    \E{\abs{AQ_n}}
\end{align*}
for some integer $n\geq 1$ and some random variable $A$. Per \cite[(1.4.3)]{bakry2014analysis}, $\Gamma(F,G)^2\leq \Gamma(F)\Gamma(G)$ so that 
\begin{align*}
\E{\abs{AQ_n}}&=\E{\abs{A\Gamma(F,Q_{n-1})}} \leq \E{\abs{A}S_1^{1/2}\Gamma(Q_{n-1})^{1/2}}\nonumber\\
&\leq\E{\abs{A}^2 S_1}^{1/2}\E{\Gamma(Q_{n-1})}^{1/2}.
\end{align*}
Then by applying the integration by part formula \eqref{intbyparts_gammaandL} of $\Gamma$, we get
\begin{align*}
    \E{\abs{AQ_n}}\leq \E{\abs{A}^2 S_1}^{1/2}\E{(-LQ_{n-1})Q_{n-1}}^{1/2}. 
\end{align*}
The term $\E{(-LQ_{n-1})Q_{n-1}}$ on the right hand side can be bounded similarly  to how $\E{\abs{AQ_n}}$ was   bounded just above, yielding 
\begin{align*}
    \E{\abs{AQ_n}}\leq\E{\abs{A}^2 S_1}^{1/2} \E{\abs{LQ_{n-1}}^2S_1 }^{1/4}\E{(-LQ_{n-2})Q_{n-2}}^{1/4}. 
\end{align*}
By iterating the previous procedure, we are able to obtain
\begin{align*}
    \E{\abs{AQ_n}}\leq \E{\abs{A}^2 S_1}^{1/2} \brac{\prod_{i=3}^{n-1}\E{\abs{LQ_i}^2 S_1}^{\frac{1}{2^{n+1-i}}}} \E{(-LQ_2)Q_2}^{\frac{1}{2^{n-2}}}. 
\end{align*}
At this point, observe that 
\begin{align*}
     \E{(-LQ_2) Q_2}&=\E{(-LQ_2)S_1^{1/2}\frac{Q_2}{S_1^{1/2}}}\\
     &\leq \E{(LQ_2)^2S_1}^{1/2}\E{\frac{Q_2^2}{S_1}}^{1/2}
\end{align*}
and by applying the estimate at\eqref{step_boundW2overW1square} for $\E{\frac{Q_2^2}{S_1}}$, we get
\begin{align*}
    \E{(-LQ_2) Q_2}\leq \E{(LQ_2)^2S_1}^{1/2} \E{\brac{LQ_1}^{2}}^{1/4}\E{Q_1^2}^{1/4}. 
\end{align*}
This leads to
\begin{align}
\label{step_boundAQ_n_normaltarget}
     &\E{\abs{AQ_n}}\nonumber\\&\leq  \E{\abs{A}^2 S_1}^{1/2} \brac{\prod_{i=2}^{n-1}\E{\abs{LQ_i}^2 S_1}^{\frac{1}{2^{n+1-i}}}} \E{(LQ_1)^2}^{\frac{1}{2^{n-4}}}\E{Q_1^2}^{\frac{1}{2^{n-4}}}.
\end{align}

Finally, we combine \eqref{step_estimateremainderterm} and \eqref{step_boundAQ_n_normaltarget} to deduce that
\begin{align}
\label{boundkthderivative_secondterm}
       \abs{{\E{\mathds{1}_{\{F>x\}}T_{k+1} }}} 
    &\leq \sum_{P_{k+1}} \E{\abs{\frac{c{\prod_{r=1}^{k+1}S_{j_r}^{n_r} } F^{i_4}}{S_1^{i_2}b(F)^{i_3} }  } \abs{Q_{i_1}}}\nonumber\\
    &\leq \sum_{P_{k+1}}\E{\abs{\frac{c{\prod_{r=1}^{k+1}S_{j_r}^{n_r} } F^{i_4}}{S_1^{i_2}b(F)^{i_3} }}^2 S_1}^{1/2} \brac{\prod_{m=2}^{i_1-1}\E{\abs{LQ_m}^2 S_1}^{\frac{1}{2^{i_1+1-m}}}}\nonumber\\ 
    &\hspace{8em}\times \E{(LQ_1)^2}^{\frac{1}{2^{i_1-4}}}\E{Q_1^2}^{\frac{1}{2^{i_1-4}}}. 
\end{align}

A combination of \eqref{prelimbound_kthderivative}, \eqref{boundkthderivative_firstterm} and \eqref{boundkthderivative_secondterm} yields the desired estimate on $\abs{p_F^{(k)}(x)-p_\mathcal{Z}^{(k)}(x)}$ when $x\in (\ell,u)$. 

\underline{Step  2}: $x<\ell$. \ 
Set 
\begin{align*}
    H:=\frac{R_{k}}{\Gamma(F)^{2k}} \,, 
\end{align*}
where $R_k$ is defined in Theorem \ref{theorem_herrydensity}, then one can verify that 
\begin{align*}
    \frac{R_{k+1}}{\Gamma(F)^{2(k+1)}}=\frac{HLF}{\Gamma(F)}+\Gamma\brac{\frac{H}{\Gamma(F)},F } . 
\end{align*}
We also know via the integration by part formula of $\Gamma$ that $\E{\Gamma\brac{\frac{H}{\Gamma(F)},F }}=-\E{\frac{HLF}{\Gamma(F)}}$,
 which implies 
\begin{align*}
    \E{\frac{R_{k+1}}{\Gamma(F)^{2(k+1)}}}= 0\,. 
\end{align*}
Then 
\begin{align*}
			p_F^{(k)}(x)&=(-1)^k\E{\mathds{1}_{\{F>x\}}\frac{R_{k+1}}{\Gamma(F)^{2(k+1)}}} =(-1)^k\E{\brac{1-\mathds{1}_{\{F\leq x\}}}\frac{R_{k+1}}{\Gamma(F)^{2(k+1)}}}\\
            &=(-1)^{k+1}\E{\mathds{1}_{\{F\leq x\}}\frac{R_{k+1}}{\Gamma(F)^{2(k+1)}}} =-\E{\mathds{1}_{\{F\leq x\}}\brac{\rho_{k+1}(F)+T_{k+1} }}. 
		\end{align*}
The last equality  is due to Proposition \ref{prop_densitydecompose}. 
Moreover, since the support of $\mathcal{Z}$ is the interval $(\ell,u)$, we can write for $x<\ell$,
\begin{align*}
    p_{\mathcal{Z}}^{(k)}(x)&=-\E{\mathds{1}_{\{\mathcal{Z}\leq x\}} \brac{\rho_{k+1}(\mathcal{Z})}}=0. 
\end{align*}
Now we can estimate the quantity
\begin{align*}
    &\abs{p_F^{(k)}(x)-p_{\mathcal{Z}}^{(k)}(x)}=\abs{\E{\mathds{1}_{\{F\leq x\}}\brac{\rho_{k+1}(F)+T_{k+1} }}-\E{\mathds{1}_{\{\mathcal{Z}\leq x\}} \brac{\rho_{k+1}(\mathcal{Z})}}  }
\end{align*}
for $x<\ell$ in the same way as  for  the case $x\in (\ell,u)$
 in Step  1.

\underline{Step  3}: $x>u$.\ 
This case  is similar to Step  2. 
\end{proof}

Based on the above results, we are going to state   a four moment estimate for approximation of a Pearson  distribution. 
Let $X$ be  a \emph{chaotic} eigenfunction of chaos grade $q$  
 and  $F=X+\alpha$. Let 
  $\mathcal{M}(q)$  be  a linear combination of the first four moments   defined according to the chaos grade $q$ as follows.  
\begin{align}
    \label{def_Mq}
\mathcal{M}(q):= \begin{cases} 
      2\brac{1-b_2-\frac{q}{4} }\E{V(F)} &\text{ when } q
      \leq 2\brac{1-b_2}\,,   \\
      2\brac{1-b_2-\frac{q}{4} }\E{V(F)}+\frac{\brac{q-2(1-b_2)}(1-b_2)}{2}\E{W^2(F)} &\text{ when } q> 2\brac{1-b_2}\,, 
   \end{cases}
\end{align}
where 
\begin{align*}
    W(x)&:=x^2+\frac{2(b_1+m)}{2b_2-1}x+\frac{1}{b_2-1}\brac{b_0+\frac{m(b_1+m)}{2b_2-1}};\\
     V(x)&:=(1-b_2)W^2(x)-\frac{1}{12}\brac{W'(x)}^3(x-m). 
\end{align*} 
\begin{theorem}
\label{theorem_fourmoment}
    In addition to the hypothesis of Proposition \ref{prop_estimatecarreduchamp_kthderivative}, let us assume further that $X$ is a {chaotic} eigenfunction of chaos grade $q$. Then,   we have the pointwise estimate with respect to $x\in \R$ for the densities of $F=X+\alpha$ and $\mathcal{Z}$.
\begin{align}
     &\abs{p_F(x)-p_\mathcal{Z}(x)}\nonumber \\
     &\leq \E{\frac{1}{S_1^3}}^{1/2}\E{\brac{LQ_1}^2}^{1/4}\mathcal{M}(q)^{1/4}+\E{ \brac{\frac{2b_2F+b_1+F-m}{S_1b(F)}}^2}^{1/2}\mathcal{M}(q)^{1/2}\nonumber\\
&\qquad +\brac{\E{\mathcal{U}_2(0,x,F)^2}^{1/2}+\E{\mathcal{U}_4(0,F)^2}^{1/2}+\E{\mathcal{U}_6(0,F)^2}^{1/2} }\frac{1}{q^2}\mathcal{M}(q)^{1/2}. \label{5.38} 
\end{align}
where $\mathcal{M}(q)$ is defined at \eqref{def_Mq}. 

Moreover,  for any integer $k$ such that $1\leq k\leq N$, we have
\begin{align}
      &\abs{p_F^{(k)}(x)-p_\mathcal{Z}^{(k)}(x)}\nonumber \\
      &\leq   \brac{\E{\mathcal{U}_2(k,x,F)^2}^{1/2}+\E{\mathcal{U}_4(k,F)^2}^{1/2}+\E{\mathcal{U}_6(k,F)^2}^{1/2} }\frac{1}{q^2} \mathcal{M}(q)^{1/2}
 \nonumber     \\
     &\qquad +\sum_{P_{k+1}}\E{\abs{\frac{c\brac{i_1,i_2,i_3,i_4,j_1,\ldots,j_{k+1},n_1,\ldots,n_{k+1}}{\prod_{r=1}^{k+1}S_{j_r}^{n_r} } F^{i_4}}{S_1^{i_2}b(F)^{i_3} }}^2 S_1}^{1/2} \nonumber\\ 
     &\qquad\quad  \qquad\quad \times \brac{\prod_{m=2}^{i_1-1}\E{\abs{LQ_m}^2 S_1}^{\frac{1}{2^{i_1+1-m}}}}\E{(LQ_1)^2}^{\frac{1}{2^{i_1-4}}}\mathcal{M}(q)^{\frac{1}{2^{i_1-4}}}\,,  
     \label{e.5.39}  
    \end{align}     
 where $P_{k+1}$ is the set of multi-indices specified in Proposition \ref{prop_densitydecompose}.  
 \end{theorem}
%

\begin{proof}
    This theorem is a consequence of Proposition \ref{prop_estimatecarreduchamp_0thderivative}, Proposition \ref{prop_estimatecarreduchamp_kthderivative} and \cite[Theorem 3.8]{bourguintaqqu2019}. In particular, the latter provides a crucial bound for $\E{Q_1^2 }$ using the first four moments of the random variable $F$.  
    \end{proof}

\begin{remark} \cite{HLN14,BDH24} consider density convergence of random variables on Wiener chaos to respectively a normal target and a Gamma target. The convergence rates in Theorem \ref{theorem_fourmoment} are slower than those in the aforementioned references, for instance the right hand side of our estimate \eqref{5.38} has order $\mathcal{M}(q)^{1/4}$, while the estimate in \cite[Corollary 4.3]{HLN14} has order $\mathcal{M}(q)^{1/2}$. This has been brought up in the introduction and it is because we substitute the product formula on Wiener chaos (which is not available on Markov diffusion chaos) with the identity $\Gamma(F,G)^2\leq \Gamma(F)\Gamma(G) $ for the carr\'{e} du champ operator $\Gamma$. Thus, it is an open question whether one can achieve the same convergence rate for density four moment theorem on Markov diffusion chaos as the rate on Wiener chaos. 
\end{remark}

\begin{remark}
In the special case when $\mathcal{Z}$ follows a normal distribution, the bound in Proposition \ref{prop_steinmethod} can be improved to a uniform bound over $x\in \R$, and consequently the density estimates in Theorem \ref{theorem_fourmoment} can also be improved to uniform estimates. It is an open question whether uniform density estimates are possible for non-normal targets. 
\end{remark}

The bounds in Theorem \ref{theorem_fourmoment} are rather complicated. As an example, in the upcoming result, we will specialize the bound \eqref{5.38} in Theorem \ref{theorem_fourmoment} to the case where the Markov diffusion chaos are Wiener chaos and the Markov generator $L$ is the (infinite-dimensional) Ornstein-Uhlenbeck generator. In this setting, the eigenfunctions $L$ are the multiple Wiener integrals (more details are provided in Appendix \ref{appendix_malliavin}). In fact, we know from \cite[Section 4.1]{ACP14} that multiple Wiener integrals are chaotic, and from \cite{tudorkusuoka2012} that among the Pearson targets, a sequence of multiple Wiener integrals can only converge to either a normal distribution or a Gamma distribution. Our focus of the upcoming result is therefore on these two targets. We will recover the main results in \cite{HLN14} and \cite{BDH24}. 
\begin{corollary}
    \label{cor_fourmomentwienerchaos}
Assume $L$ is the infinite-dimensional Ornstein-Uhlenbeck generator (defined in Appendix \ref{appendix_malliavin}) and $\Gamma$ is the associated carr\'{e} du champ operator. $X_n,n\geq 1$ is a sequence of eigenfunctions of $L$ with the same eigenvalue $q$ (in other words, $X_n,n\geq 1$ is a sequence of multiple Wiener integrals of a fixed order order $q$). Let  $\mathcal{Z}$ be a random variable distributed either as a normal distribution with density $\frac{1}{\sqrt{2\pi}}e^{-\frac{(x-\alpha)^2}{2}}$ or a Gamma distribution with density $\frac{1}{\Gamma(\alpha)}x^{\alpha-1}e^{-x}\mathds{1}_{ \{x> 0\} }$. Set $F_n=Z_n+\alpha,n\geq 1$. 

Under the conditions that 
\begin{itemize}
  \item the first four moments of $F_n$ converge to the first four moments of $\mathcal{Z}$, 
    \item $\sup_{n\geq 1} \E{\frac{1}{\Gamma(F_n)^3}}<\infty$ in the case $\mathcal{Z}$ is normally distributed; and \\$\sup_{n\geq 1} \brac{\E{\frac{1}{F_n^6}\mathds{1}_{\{F_n\leq 0 \}}}+\E{\frac{1}{\Gamma(F_n)^4}}+\E{\frac{1}{F_n^4}} }<\infty $ in the case $\mathcal{Z}$ is Gamma distributed,
\end{itemize}
then the densities $p_{F_n}(x)$ converges to the density $p_\mathcal{Z}(x)$ for every $x\in \R$.
    
\end{corollary}

\begin{remark}
    When $\mathcal{Z}$ is a normal distribution, a recent paper \cite{polyherry2023centralsuperconvergence} shows that convergence of the first four moments of $F_n$ to the first four moments of $\mathcal{Z}$, without any additional assumptions, is sufficient to imply convergence of densities of $F_n$ to the density of $\mathcal{Z}$. They coin this the \textit{superconvergence} phenomenon on Wiener chaos and it is a significant improvement of Theorem 4.1 in \cite{HLN14} and our result above in the case the target $\mathcal{Z}$ is normal. 
\end{remark}

\begin{proof}
In the upcoming proof, we will use notations for Malliavin calculus summarized in Appendix \ref{appendix_malliavin}. 

The fact that the first four moments of $F_n$ converge to the first four moments of $\mathcal{Z}$ implies $F_n$ converges in distribution to $\mathcal{Z}$ per the classic references \cite{NP09main}. $F_n,n\geq 1$ is then a tight sequence of multiple Wiener integrals, and by \cite[Part 3 of Exercise 2.8.17]{nourdinpeccatibook}, it holds for every $r>0$ that 
\begin{align}
\label{uniformmoment}
    \sup_{n\geq 1} \E{\abs{F_n}^r}< \infty. 
\end{align}

\underline{When $\mathcal{Z}$ has Gamma distribution:} we have $b(x)=x$ in this case. $\mathcal{M}(q)$ in the bound \eqref{5.38} goes to $0$ due to the convergence of the first four moments, so it is sufficient to check that the remaining terms are bounded uniformly. 

It is easy to see $\sup_{n\geq 1}\brac{\E{\mathcal{U}_2(0,x,F_n)^2}^{1/2}+\E{\mathcal{U}_4(0,F_n)^2}^{1/2}+\E{\mathcal{U}_6(0,F_n)^2}^{1/2}}<\infty $ when $\sup_{n\geq 1} \brac{\E{\frac{1}{F_n^6}\mathds{1}_{\{F_n\leq 0 \}}}+\E{\frac{1}{F_n^2}\mathds{1}_{\{F_n> 0 \}}}}<\infty$. 

Moreover, $b(x)=x$ means $b_2=0,b_1=1$ so that 
\begin{align*}
    \E{ \brac{\frac{2b_2F_n+b_1+F_n-m}{S_1b(F_n)}}^2}\leq \E{\brac{1+\frac{1-\alpha}{F_n}}^4}^{1/2}\times \E{\frac{1}{S_1^4}}^{1/2}.
\end{align*}
Hence, when $\sup_{n\geq 1} \brac{\E{\frac{1}{\Gamma(F_n)^4}}+\E{\frac{1}{F_n^4}} }<\infty $, we have $\E{ \brac{\frac{2b_2F_n+b_1+F_n-m}{S_1b(F_n)}}^2}<\infty$.

Finally, we consider the term $\E{\brac{LQ_1}^2}$ where $Q_{1,n}=\Gamma(F_n)-q F_n$. Per \eqref{functionalinequality}, it can be bounded as 
\begin{align}
\label{boundLQ1}
    \E{\brac{LQ_{1,n}}^2}\leq C\brac{\E{Q_{1,n}^2}+\E{\norm{DQ_{1,n}}^2_\frak{H}}+\E{\norm{D^2Q_{1,n}}^2_{\frak{H}^{\otimes 2}}} }.
\end{align}
Per \cite[Lemma 3.8 and Lemma 5.6]{BDH24}, $\E{\norm{DQ_{1,n}}^2_\frak{H}}$ and $\E{\norm{D^2Q_{1,n}}^2_{\frak{H}^{\otimes 2}}} $ are bounded from above by some multiple of $\E{Q_{1,n}^2}$, which in turns is bounded by the first four moments of $F_n$ per the classic reference \cite{NP09main}. Consequently, under our assumption of the convergence of the first four moments of $F_n$,  $\sup_{n\geq 1}\E{\brac{LQ_{1,n}}^2}<\infty$. 

Combining the previous arguments yields the sufficient condition that  \\$\sup_{n\geq 1} \brac{\E{\frac{1}{F_n^6}\mathds{1}_{\{F_n\leq 0 \}}}+\E{\frac{1}{\Gamma(F_n)^4}}+\E{\frac{1}{F_n^4}} }<\infty $ in the case $\mathcal{Z}$ is Gamma distributed. 

\underline{When $\mathcal{Z}$ has normal distribution:} we have $b(x)=1$ in this case. Similarly to the earlier argument, it is sufficient to check that the terms beside $\mathcal{M}(q)$ on the right hand side of \eqref{5.38} are bounded uniformly.

We have $\sup_{n\geq 1}\brac{\E{\mathcal{U}_2(0,x,F_n)^2}^{1/2}+\E{\mathcal{U}_4(0,F_n)^2}^{1/2}+\E{\mathcal{U}_6(0,F_n)^2}^{1/2}}<\infty $ as long as $\sup_{n\geq 1} \E{F_n^2}<\infty$. The latter has been shown at \eqref{uniformmoment}.

Next, $b(x)=1$ means $b_2=b_1=0$ and $b_0=1$ so that 
\begin{align*}
    \E{ \brac{\frac{2b_2F_n+b_1+F_n-m}{S_1b(F_n)}}^2}\leq \E{\frac{1}{\Gamma(F_n)^3}}^{2/3}\times \E{(F_n-\alpha)^6}^{1/3}. 
\end{align*}
We know $\sup_{n\geq 1}\E{(F_n-\alpha)^6}<\infty$ per \eqref{uniformmoment}. Then as long as $\sup_{n\geq 1}\E{\frac{1}{\Gamma(F_n)^3}}<\infty$, we have $\E{ \brac{\frac{2b_2F_n+b_1+F_n-m}{S_1b(F_n)}}^2}< \infty$. 

Finally, regarding the term $\E{\brac{LQ_{1,n}}^2}$ where $Q_{1,n}=\Gamma(F_n)-q$, we bound it the same way as \eqref{boundLQ1}. Per the calculation in \cite[Lemma A.1]{HLN14}, $\E{\norm{DQ_{1,n}}^2_\frak{H}}$ and $\E{\norm{D^2Q_{1,n}}^2_{\frak{H}^{\otimes 2}}} $ are bounded from above by some multiple of $\E{Q_{1,n}^2}$, which in turns is bounded by the first four moments of $F_n$ per the classic reference \cite{NP09main}. Then  under our assumption that the first four moments of $F_n$ converge, we get $\sup_{n\geq 1}\E{\brac{LQ_{1,n}}^2}<\infty$. 

Combining the previous arguments yields the sufficient condition that $\sup_{n\geq 1} \E{\frac{1}{\Gamma(F_n)^3}}<\infty$ in the case $\mathcal{Z}$ is normally distributed. 
\end{proof}

As an second example, in the upcoming result, we will specialize the bound \eqref{5.38} in Theorem \ref{theorem_fourmoment} to the case where the Markov diffusion chaos are the Laguerre chaos. Specifically, let $\gamma,\alpha>0$ and   denote by $\nu_\gamma$ a Gamma distribution with density $\frac{1}{\Gamma(\gamma)}x^{\gamma-1}e^{-x}\mathds{1}_{ \{x> 0\} }$. $Y_i,i\geq 1$ is a sequence of i.i.d. random variables whose laws are $\nu_\gamma$. Moreover, $\mathcal{Z}$ is a random variable whose law follows $\nu_\alpha$. We will assume the infinite product Dirichlet structure associated with $Y_i,i\geq 1$ as described in \cite[Section 1.15.3 and Section 2.7.3]{bakry2014analysis} (see also \cite[Chapter 5, Section 2.2]{bouleauhirsch2010dirichlet}). The carr\'{e} du champ operator $\Gamma$ acts on the basis elements $Y_i,i\geq 1$ as follows.
\begin{align}
\label{laguerreoperator_basiselement}
 \Gamma(Y_i):=\Gamma(Y_i,Y_i)=Y_i ,\quad   \Gamma(Y_i,Y_j)=0, \quad i\neq j. 
\end{align}
This can be viewed as an infinite-dimensional extension of the one-dimensional Laguerre structure in Example \ref{example_laguerre}. Let $X$ be an eigenfunction associated with $\Gamma$ then $X$ is known to be chaotic per \cite[Section 4.2]{ACP14}.

\begin{corollary}
       \label{cor_gammfourmoment}
Assume the above setup. Under the conditions that  the first four moments of $F_n$ converge to the first four moments of $\mathcal{Z}$ and 
\begin{align}
\label{sufficientmomentbound_density}
    \sup_{n\in\N}\brac{\E{\frac{1}{F_n^6}\mathds{1}_{\{F_n\leq 0 \}}}+\E{\frac{1}{\Gamma(F_n)^4}}+\E{\frac{1}{F_n^4}}+\E{\brac{L(\Gamma(F_n)-F_n}^2}}<\infty, 
\end{align}
then the densities $p_{F_n}(x)$ converges to the density $p_\mathcal{Z}(x)$ for every $x\in \R$. 
    
\end{corollary}

\begin{proof}

 $\mathcal{M}(q)$ in the bound \eqref{5.38} goes to $0$ due to the convergence of the first four moments, so it is sufficient to check that the remaining terms are bounded uniformly in order to have the right hand side of \eqref{5.38} go to $0$.


$\mathcal{Z}$ is Gamma distribution so that $b(x)=b_2x^2+b_1x+b_0=x$. Then if we assume $\sup_{n\geq 1} \brac{\E{\frac{1}{F_n^6}\mathds{1}_{\{F_n\leq 0 \}}}+\E{\frac{1}{F_n^2}\mathds{1}_{\{F_n> 0 \}}}}<\infty$, we will get the uniform bound \\$\sup_{n\geq 1}\brac{\E{\mathcal{U}_2(0,x,F_n)^2}^{1/2}+\E{\mathcal{U}_4(0,F_n)^2}^{1/2}+\E{\mathcal{U}_6(0,F_n)^2}^{1/2}}<\infty $.     

Next, by H\"{o}lder inequality, $ \E{ \brac{\frac{2b_2F_n+b_1+F_n-m}{S_1b(F_n)}}^2}\leq \E{\brac{1+\frac{1-\alpha}{F_n}}^4}^{1/2}\times \E{\frac{1}{S_1^4}}^{1/2}$. Thus, by assuming  $\sup_{n\geq 1} \brac{\E{\frac{1}{\Gamma(F_n)^4}}+\E{\frac{1}{F_n^4}} }<\infty $, we get $\sup_{n\geq 1}\E{ \brac{\frac{2b_2F_n+b_1+F_n-m}{S_1b(F_n)}}^2}<\infty$. 

Finally, we assume $\sup_{n\geq 1}\E{\brac{L(\Gamma(F_n)-F_n}^2}<\infty$. In view of \eqref{5.38}, this completes the proof.  
\end{proof}


\section{Application to convergence of weighted sum of i.i.d. Gamma distribution.}
\label{section_superconvergencetogamma}

Assume the setup prior to Corollary \ref{cor_gammfourmoment}. The upcoming result is the main result of this section. It says under suitable conditions, convergence in distribution of a weighted sum of i.i.d. Gamma random variables to some other Gamma random variable also implies convergence of the densities. 
\begin{theorem}
\label{theorem_superconvergencetogamma}
Assume the setup above. Let $\{X_n:n\geq 1\}$ be a sequence of weighted sum
\begin{align*}
    X_n=\sum_{i=1}^{k_n}\lambda_{i,n} (Y_i-\gamma)
\end{align*}
with the ordering $\abs{\lambda_{1,n}}\geq \abs{\lambda_{2,n}}\geq \ldots>0$. Assume further that the sequence $X_n+\alpha,n\geq 1$ converges in distribution to $G_\alpha$. If $\alpha>6$, then we have
\begin{align*}
    \lim_{n\to\infty} p_{X_n+\alpha}(x)=p_{G_\alpha}(x). 
\end{align*}

    \end{theorem}

    \begin{remark}
    Let $F$ be any random variable in the second Wiener chaos then it is well-known that $F$ can be written as a weighted sum of i.i.d. chi-square random variables of 1 degree of freedom. Hence, the above Theorem is applicable to a sequence of random variables in the second Wiener chaos. 
    \end{remark}

Next, we present two lemmas that will facilitate the proof of Theorem \ref{theorem_superconvergencetogamma} at the end of this section. 
\begin{lemma}
\label{lemma_factwhenFnconvergetogamma}
 Assume the setup of Theorem \ref{theorem_superconvergencetogamma}.  The fact that $X_n+\alpha\xrightarrow{\operatorname{dist}} G_\alpha$ implies the following. 
    \begin{enumerate}[label=\roman*)]
        \item $ \Gamma(X_n)\xrightarrow{\operatorname{dist}} G_\alpha$. 
        \item for any integer $q$ satisfying $1\leq q< \alpha/\gamma$, we have
         \begin{align*}
          \lim_{n\to\infty}   M_{q+1,n}:=\frac{\alpha}{2} \brac{\frac{\alpha}{2}-\lambda_{1,n}^2}\dots \brac{\frac{\alpha}{2}-\sum_{j=1}^{q\wedge k_n}\lambda_{j,n}^2}> 0.
         \end{align*}
        \item  $\lim_{n\to \infty }\inf_{j\in \N} \lambda_{j,n}\geq 0$,
so that we can assume for $N$ sufficiently large and any $n\geq N$, 
\begin{align*}
    \lambda_{j,n}>0,\quad \forall   j\in \N. 
\end{align*}
    \end{enumerate}
\end{lemma}

\begin{proof}
    Regarding Part $i)$, the fact that  $X_n+\alpha\xrightarrow{\operatorname{dist}} G_\alpha$ implies $\{X_n,n\geq 1\}$ is a tight sequence of random variables. Moreover, take $\mathcal{P}_1$ to be the space spanned by $Y_i,i\geq 1$ then per \cite[Lemma 1.4]{herry2023regularity}, $\mathcal{P}_1$ has the hypercontractivity property: all the $L^p$ norms are equivalent on $\mathcal{P}_1$. Per \cite[Part 3 of Exercise 2.8.17]{nourdinpeccatibook}, tightness and hypercontractivity lead to the uniform moment estimate 
    \begin{align*}
      \sup_{n\in\N}  \E{\abs{X_n+\alpha}^p}<\infty,\quad \forall p>0. 
    \end{align*}
    This in turn implies for any $p\in \N$, $\E{\abs{X_n+\alpha}^p}\to \E{\abs{G_\alpha}^p}$ as $n\to\infty$ (\cite[Theorem 5.9]{gut2006probabilitybook}).
    
 Then per \cite[Theorem 3.4]{ACP14}, convergence of the first four moments leads to
\begin{align}
\label{estimate_forconvergenceinlaws}
    \E{\brac{\Gamma(X_n)-(X_n+\alpha)}^2}\to 0,
\end{align}
and consequently, we get $ \Gamma(X_n)\xrightarrow{\operatorname{dist}} G_\alpha$.

Regarding Part $ii)$, we suppose the opposite statement holds and $\lim_{n\to\infty}M_{q,n}=0$ for some $q$ in $[1,k_n\wedge \brac{\alpha/\gamma})$, which implies
$\lim_{n\to\infty}\sum_{i=1}^{q'}\lambda_{i,n}^2= \alpha/2$ for some $q'\leq q$. Due to the ordering $\abs{\lambda_{1,n}}\geq \abs{\lambda_{2,n}}\geq \dots$ and the fact that 
\begin{align}
\label{application_variance}
    \lim_{n\to \infty} \E{X_n^2}=\lim_{n\to \infty} \sum_{i=1}^{k_n} \lambda_{i,n}^2=\E{G_\alpha^2}=\alpha/2, 
\end{align}
we can deduce $0=\lim_{n\to\infty}\lambda_{q'+1,n}=\ldots=\lim_{n\to\infty}\lambda_{8,n}=\lim_{n\to\infty}\lambda_{9,n}=\ldots$ and
\begin{align*}
  \lim_{n\to\infty}\prod_{i=q'+1}^{k_n} \brac{1-it\lambda_{i,n}^2}^{-\gamma}=1.
\end{align*}
Therefore, given that $q'<\alpha/\gamma$ which implies $-\alpha<-\gamma q$, the characteristic functions of $\Gamma(X_n),n\geq 1$ which are
\begin{align*}
    \E{e^{it\Gamma(X_n)}}= \prod_{i=1}^{k_n}\E{\exp\brac{it\lambda_{i,n}^2Y_i}}&= \prod_{i=1}^{q'} \brac{1-it\lambda_{i,n}^2}^{-\gamma}\prod_{i=q'+1}^{k_n} \brac{1-it\lambda_{i,n}^2}^{-\gamma}\\
    &=\prod_{i=1}^{q'} \brac{1-it\lambda_{i,n}^2}^{-\gamma}
\end{align*}
 cannot converge to $\E{\exp\brac{it G_\alpha}}= (1-it)^{-\alpha}$. This contradicts Part $i)$ of the current lemma, and we can conclude $\lim_{n\to\infty}M_{q+1,n}> 0$. 

Finally, let us prove Part $iii)$ via a contradiction argument. Notice that $X_n+\alpha\xrightarrow{\operatorname{dist}} G_\alpha$ implies
\begin{align}
\label{fact}
    \lim_{n\to\infty } \Pr\brac{X_n\leq -2\alpha}=\Pr\brac{G_\alpha-\alpha\leq -2\alpha }=0. 
\end{align}
Suppose Part $iii)$ is not true and for every $k\in \N$, there exists some $k_0>k$ such that at least one eigenvalue $\lambda_{j,k_0}<0$ among $\{\lambda_{j,k_0}:j\geq 1 \}$. More precisely, we have
\begin{align*}
    F_{k_0}=\sum_{j\in J_1}\lambda_{j,k_0}Y_j+\sum_{j\in J_2}\lambda_{j,k_0}Y_j\,, 
\end{align*}
where $J_1,J_2$ are the sets
\begin{align*}
    J_1=J_{1, k_0} :=\left\{j\in \N:\lambda_{j,k_0}> 0  \right\},\qquad J_2=J_{2, k_0} :=\left\{j\in \N:\lambda_{j,k_0}< 0  \right\}
\end{align*}
and $J_2$ is non-empty. Then 
\begin{align*}
    &\Pr\brac{ F_{k_0}\leq -2\alpha}\\
    &\geq \Pr\brac{F_{k_0}\leq -2\alpha; \ 0\leq  \sum_{j\in J_1}\lambda_{j,k_0}X_j\leq \alpha/2}\\
    &\geq \Pr\brac{\sum_{j\in J_2}\lambda_{j,k_0}Y_j+\alpha/2\leq -2\alpha; 0\leq  \sum_{j\in J_1}\lambda_{j,k_0}Y_j\leq \alpha/2}\\
    &=\Pr\brac{ \sum_{j\in J_2}\lambda_{j,k_0}Y_j\leq -5\alpha/2}\cdot \Pr\brac{ 0\leq \sum_{j\in J_1}\lambda_{j,k_0}Y_j\leq \alpha/2}.
\end{align*}
The last line is due to pairwise independence among $\{Y_j:j\geq 1 \}$. One can see from this calculation that  $\Pr\brac{ F_{k_0}\leq -2\alpha}>0$, which contradicts \eqref{fact}, and thus we have $\lim_{n\to \infty }\inf_{j\in \N} \lambda_{j,n}\geq 0$. The second half of Part $iii)$ is because by definition, $\lambda_{j,n}\neq 0$ for all $j,n\geq 1$.  
\end{proof}

\begin{lemma}
    \label{lemma_gammaapplication_momentbound}
Assume the setup of Theorem \ref{theorem_superconvergencetogamma}. Let $q$ be any positive number satisfying $q<\alpha$. Then we have the moment bounds 
\begin{align*}
   \lim_{n\to\infty}  \E{\Gamma(X_n)^{-q}} <\infty,\qquad \lim_{n\to \infty} \E{ (X_n+\alpha)^{-q}}<\infty. 
\end{align*}
\end{lemma}

\begin{proof}

Part $iii)$ of Lemma \ref{lemma_factwhenFnconvergetogamma} only holds for $n$ sufficiently large, however to simplify our argument, let us assume $   \lambda_{i,n}>0;\forall i,n\geq 1 $ in the rest of the proof.

As in \cite{polyherry2023centralsuperconvergence,herry2023regularity}, with respect to $q,n\geq 1$ and the set $\{\lambda_{i,n}:1\leq i\leq k_n\}$, let us define 
\begin{align*}
    \mathcal{R}_{q,n}:=\sum_{i_1\neq i_2\neq \cdots\neq i_q }\lambda_{i_1,n }^2\ldots \lambda_{i_q,n }^2;\qquad \mathcal{S}_{q,n}:=\sum_{i_1\neq i_2\neq \cdots\neq i_q }\lambda_{i_1,n }\ldots \lambda_{i_q,n }.
\end{align*}

Via \cite[(22) of Lemma 25]{polyherry2023centralsuperconvergence}, Part $ii)$ of our Lemma \ref{lemma_factwhenFnconvergetogamma} and \eqref{application_variance}, we have for any integer $p$ in the interval $[0,\alpha/\gamma)$, 
\begin{align}
\label{lowerbound_Rp}
     \lim_{n\to \infty}    \mathcal{R}_{p+1,n}\geq \lim_{n\to\infty }M_{p+1,n}>0.
\end{align}

In addition, \eqref{application_variance} is equivalent to 
\begin{align}
    \label{lowerbound_R1}
    \lim_{n\to \infty}    \mathcal{R}_{1,n}=\alpha/2>0. 
\end{align}
Thus, we can use the above lower estimates and and proceed to bound $ \E{\Gamma(X_n)^{-4}}$ similar to how it is done in \cite[Proof of Proposition 27]{polyherry2023centralsuperconvergence}. For $t\geq 0$ and a positive integer $m$ to be specified later, we have
\begin{align}
\label{estimate_laplacetransformGammaF}
    \E{\exp\brac{-\frac{t^m}{2}\Gamma(X_n) }}&=\E{\exp\brac{ -\frac{t^m}{2}\brac{\sum_{i=1}^{k_n} 2\lambda_{i,n}^2Y_i }}}\nonumber \\
    &\leq \prod_{i=1}^{k_n} \frac{1}{(1+2t^m\lambda_{i,n}^2)^{\gamma}}\leq \frac{1}{t^{m q\gamma}\mathcal{R}_{q}(A_n)^{\gamma}} 
\end{align}
which holds for any $q\leq k_n$. We also know 
\begin{align}
\label{equation_relatemomenttolaplacetransform}
    x^{-q}&=\frac{1}{(q-1)!}\int_0^\infty t^{q-1}e^{-tx}dt=\frac{m}{2(q-1)!}\int_0^\infty t^{mq-1}e^{-t^m x/2}dt, \qquad x>0. 
\end{align}
Moreover, we assume $\gamma<q<\alpha$, the second half of which implies $(\alpha-q)/\gamma>1$, so it is possible to find an integer $p$ such that $q/\gamma-1<p<\alpha/\gamma$. Hence, 
\begin{align}
\label{estimate_negmomentofgammaFbyspectralremainder}
    &\E{ \Gamma(X_n)^{-q}}\nonumber\\
    &= \frac{2}{(q-1)!}\int_\R \abs{t}^{2q-1} \E{\exp\brac{-\frac{t^m}{2}\Gamma(X_n) }}dt\nonumber\\
&\leq \frac{m}{2(q-1)!}\int_0^1 t^{mq-1}\frac{1}{t^{m\gamma} \mathcal{R}_{1,n}^{\gamma} }dt+\frac{m}{2(q-1)!}\int_1^\infty t^{mq-1}\frac{1}{t^{m\gamma(p+1)} \mathcal{R}_{p+1,n}^{\gamma} }dt\nonumber\\
&=\frac{m}{2(q-1)!m(q-\gamma)}\frac{1}{ \mathcal{R}_{1,n}^{\gamma} }+\frac{m}{2(q-1)!m(\gamma(p+1)-q)}\frac{1}{ \mathcal{R}_{p+1,n}^{\gamma} }. 
\end{align}
In particular, since $q>\gamma$, we choose $m$ big enough such that 
\begin{align}
\label{choiceofm}
    q>\frac{1}{m}+\gamma,
\end{align}
which is equivalent to $mq-1-m\gamma>0$. Furthermore, the integer $p$ satisfies $mq-1-m\gamma(p+1)<-1$. Consequently, both integrals in the second-to-last line in \eqref{estimate_negmomentofgammaFbyspectralremainder} converge. Finally, by \eqref{lowerbound_Rp} and \eqref{lowerbound_R1}, we   conclude that 
\begin{align}
\label{gammaapp_negativemomentofgammaX_n}
  \lim_{n\to\infty}  \E{\Gamma(X_n)^{-q}} <\infty. 
\end{align}

Next, we consider $\E{(X_n+\alpha)^{-q}}$. It follows from \eqref{application_variance} and our assumption {$\alpha\geq 1/2$} that
\begin{align}
\label{secondchaos_comparealphatosumofeigenvalues}
 \lim_{n\to\infty}   \sum_{i=1}^{k_n} \lambda_{i,n} \leq \lim_{n\to\infty} \brac{\sum_i \lambda_{i,n}^2}^{1/2}=\sqrt{\alpha/2}\leq  \alpha. 
\end{align}
This implies for $t\geq 0$ and integer $m$ as chosen in \eqref{choiceofm},
\begin{align*}
   \lim_{n\to\infty} &   \E{\exp\brac{-\frac{t^m}{2}(X_n+\alpha) }}\\
  &= \lim_{n\to\infty} \E{\exp\brac{ -\frac{t^m}{2}\brac{\sum_{i=1}^{k_n} \lambda_{i,n}Y_i }}}\cdot  \lim_{n\to\infty} \exp\brac{-\frac{t^m}{2}\brac{\alpha- \sum_i \lambda_{i,n}}} \\
    &\leq  \lim_{n\to\infty} \prod_{i=1}^{k_n} \frac{1}{(1+t^m\lambda_{i,n})^\gamma}\cdot 1\leq  \lim_{n\to\infty} \frac{1}{t^{mq\gamma}\mathcal{S}_{q}(A_n)^{\gamma}}  
\end{align*}
for any integer $q$ in $[1, k_n]$. Then by 
\eqref{equation_relatemomenttolaplacetransform}, we have for any integer $q$ satisfying $\gamma<q<\alpha$,   
\begin{align}
\label{upperbound_negativemomentFn_beforelast}
&\E{ (X_n+\alpha)^{-q}}\nonumber\\
&= \frac{m}{2(q-1)!}\int_\R \abs{t}^{mq-1} \E{\exp\brac{-\frac{t^m}{2}(X_n+\alpha) }}dt\nonumber\\
&\leq \frac{m}{2(q-1)!}\int_0^1 t^{mq-1}\frac{1}{t^{m\gamma} \mathcal{S}_{1,n}^{\gamma} }dt+\frac{m}{2(q-1)!}\int_1^\infty t^{mq-1}\frac{1}{t^{m\gamma(p+1)} \mathcal{S}_{p+1,n}^{\gamma} }dt\nonumber\\
&=\frac{m}{2(q-1)!m(q-\gamma)}\frac{1}{ \mathcal{S}_{1,n}^{\gamma} }+\frac{m}{2(q-1)!m(\gamma(p+1)-q)}\frac{1}{ \mathcal{S}_{p+1,n}^{\gamma} }. 
\end{align}
Next, due to \eqref{secondchaos_comparealphatosumofeigenvalues}, we can write for any positive integer $q$ 
\begin{align*}
    \mathcal{R} _{q,n}\leq \max_{1\leq i\leq  k_n} \lambda_{i,n}^q S_{q,n}\leq \alpha^q S_{q,n}. 
\end{align*}
This combined with \eqref{lowerbound_Rp} and \eqref{lowerbound_R1} give us
\begin{align*}
    \lim_{n\to \infty}    \mathcal{S}_{q,n}>0,\quad \lim_{n\to \infty}    \mathcal{S}_{1,n}>0,\quad \gamma<q<\alpha. 
\end{align*}
Then from \eqref{upperbound_negativemomentFn_beforelast}, we can conclude that 
\begin{align}
\label{gammaapp_negativemomentofX_n}
    \lim_{n\to \infty} \E{ (X_n+\alpha)^{-q}}<\infty. 
\end{align}
As the final step, notice we have obtained the moment bounds \eqref{gammaapp_negativemomentofgammaX_n} and \eqref{gammaapp_negativemomentofX_n} for $q$ satisfying $\gamma<q<\alpha$. However, due to Lyapunov's inequality, the requirement $q>\gamma$ is redundant and the moment bound holds for $0<q<\alpha$. This completes the proof.  
\end{proof}

\begin{lemma}
\label{lemma_momentbound_LQ1}
    Assume the setup of Theorem \ref{theorem_superconvergencetogamma}. Then we have the moment bounds 
\begin{align*}
\sup_{n\in\N}\E{\brac{L\brac{\Gamma(X_n)-(X_n+\alpha)}}^2}<\infty. 
\end{align*}
\end{lemma}

\begin{proof}
    We will follow the strategy in \cite{ACP14} and bound $\E{\brac{L\brac{\Gamma(X_n)-2(X_n+\alpha)}}^2}$ by the combination of the first four moments of $X_n+\alpha$. Note that unlike the four moment bound on $\E{\brac{\Gamma(X_n)-2(X_n+\alpha)}^2}$, our upcoming four moment bound of $\E{\brac{L\brac{\Gamma(X_n)-2(X_n+\alpha)}}^2}$ does not necessarily equal $0$ when $X_n+\alpha$ is distributed as the law of  $G_\alpha$.

    We have $  L\brac{\Gamma\brac{X_n+\alpha}-(X_n+\alpha)}=L(L+2\operatorname{Id})P_2^{\alpha-1}(X_n+\alpha)$ where $P_2^{\alpha-1}(x)=\frac{x^2}{2}-(\alpha+1)x+\frac{\alpha(\alpha+1)}{2}$ is the second Laguerre polynomial of parameter $\alpha-1$. Then by the integration by part formula \eqref{intbyparts_gammaandL}, 
    \begin{align*}
       &\E{\brac{L\brac{\Gamma(X_n)-(X_n+\alpha)}}^2}\\
       &= \E{L(L+2\operatorname{Id})P_2^{\alpha-1}(X_n+\alpha) L(L+2\operatorname{Id})P_2^{\alpha-1}(X_n+\alpha)}\\
       &=\E{P_2^{\alpha-1}(X_n+\alpha)L^2(L+2\operatorname{Id})(L+2\operatorname{Id}) P_2^{\alpha-1}(X_n+\alpha) }\\
       &=\E{P_2^{\alpha-1}(X_n+\alpha)L^3(L+2\operatorname{Id}) P_2^{\alpha-1}(X_n+\alpha)}\\
       &\qquad\qquad\qquad+2\E{P_2^{\alpha-1}(X_n+\alpha)L^2(L+2\operatorname{Id})P_2^{\alpha-1}(X_n+\alpha)}\\
       &=\E{L^3 P_2^{\alpha-1}(X_n+\alpha)(L+2\operatorname{Id}) P_2^{\alpha-1}(X_n+\alpha)}\\
       &\qquad\qquad\qquad+2\E{LP_2^{\alpha-1}(X_n+\alpha)(L+2\operatorname{Id})LP_2^{\alpha-1}(X_n+\alpha)}:=\mathcal{A}_1+\mathcal{A}_2. 
    \end{align*}
We will bound individual terms on the right hand side of the above equation. 

Regarding $\mathcal{A}_1$, it is easy to check (see for instance \cite[Section 4.2]{ACP14}) that $X_n^2 \in \bigoplus_{k=0}^2 \ker(L+k\operatorname{Id})$. Denote $J_k,k\in \N$ the projection map into the eigenspace associated with the eigenvalue $k$ of $L$. Then we get  $L^3P_2^{\alpha-1}(X_n+\alpha)=-J_1(P_2^{\alpha-1}(X_n+\alpha))-8J_2(P_2^{\alpha-1}(X_n+\alpha))$ and $L^2(L+2\operatorname{Id})P_2^{\alpha-1}(X_n+\alpha)=2\E{P_2^{\alpha-1}(X_n+\alpha)}+J_1(P_2^{\alpha-1}(X_n+\alpha))$. We also know that the eigenspaces corresponding to different eigenvalues of $L$ are orthogonal, so that 
\begin{align*}
    \mathcal{A}_1=-\E{J_1(P_2^{\alpha-1}(X_n+\alpha))^2 }\leq 0. 
\end{align*}

Regarding $\mathcal{A}_2$, we first show that $LP_2^{\alpha-1}(X_n+\alpha)$ has the form $Q(X_n+\alpha)$ where $Q$ is a second-degree polynomial. By the chain rule \eqref{eq:30} and \eqref{laguerreoperator_basiselement},
\begin{align*}
    LP_2^{\alpha-1}(X_n+\alpha)&=(P_2^{\alpha-1})'(X_n+\alpha)L(X_n+\alpha)+(P_2^{\alpha-1})''(X_n+\alpha)\Gamma(X_n+\alpha)\\
    &=-P_1^{\alpha}(X_n+\alpha)\brac{-\sum_{i=1}^{k_n}\lambda_{i,n}Y_i}+\sum_{i=1}^{k_n}\lambda_{i,n}Y_i\\
    &=-P_1^{\alpha}(X_n+\alpha)\brac{X_n-\alpha+\sum_{i=1}^{k_n}\lambda_{i,n}}+\brac{X_n-\alpha+\sum_{i=1}^{k_n}\lambda_{i,n}}\\
    &:=Q(X_n+\alpha), 
\end{align*}
where $Q(x)$ is a second degree polynomial. Then per \cite[Lemma 3.14]{ACP14},
\begin{align*}
    \mathcal{A}_2=2\E{Q(X_n+\alpha)(L+2\operatorname{Id})Q(X_n+\alpha) }= \E{Q^2(X_n+\alpha)-\frac{Q'(X_n+\alpha)(X_n+\alpha)}{3Q''(X_n+\alpha)} } \,, 
\end{align*}
where the right hand side is a combination of the first four positive moments of $X_n+\alpha$.  

The previous calculations of $\mathcal{A}_1$ and $\mathcal{A}_2$ imply
\begin{align*}
 \sup_{n\in \N}   \E{\brac{L\brac{\Gamma(X_n)-(X_n+\alpha)}}^2}\leq  \sup_{n\in \N}\E{Q^2(X_n+\alpha)-\frac{Q'(X_n+\alpha)(X_n+\alpha)}{3Q''(X_n+\alpha)} } . 
\end{align*}
Moreover, like in the first part of the proof of Lemma \ref{lemma_factwhenFnconvergetogamma}, $X_n+\alpha \xrightarrow{\operatorname{dist}}G_\alpha$ implies for any $p>0$
\begin{align}
\label{convergenceofmoments}
    \lim_{n\to\infty} \E{\abs{X_n+\alpha}^p}=\E{\abs{G_\alpha}^p}.
\end{align}
Thus, we can conclude that $\sup_{n\in \N}   \E{\brac{L\brac{\Gamma(X_n)-(X_n+\alpha)}}^2}<\infty$.      
\end{proof}

Finally, let us provide the proof of Theorem \ref{theorem_superconvergencetogamma}. 
\begin{proof}[Proof of Theorem \ref{theorem_superconvergencetogamma}]

\eqref{convergenceofmoments} includes the convergence of the first four positive moments of $X_n+\alpha$ to the corresponding moments of $G_\alpha$. Then according to Corollary \ref{cor_gammfourmoment}, the densities of $X_n+\alpha$ converges pointwise to the density of $G_\alpha$ as long as 
\begin{align*}
    &\sup_{n\in\N}\E{\frac{1}{F_n^6}}<\infty,\quad \sup_{n\in\N}\E{\Gamma(X_n)^{-4}}<\infty,\quad \sup_{n\in \N}\E{(X_n+\alpha)^{-4}}<\infty,\\
    &\sup_{n\in\N}\E{\brac{L\brac{\Gamma(X_n)-(X_n+\alpha)}}^2}<\infty. 
\end{align*}
The bounds in the first line are derived in Lemma \ref{lemma_gammaapplication_momentbound} under the condition $\alpha>6$, while the bound in the second line is done in Lemma \ref{lemma_momentbound_LQ1}. This completes the proof.

\end{proof}

\section{Application to exponential convergence toward stationary densities via Lyapunov-type conditions}
~\
\label{section_expconvergence}

There are scenarios where computing the bounds in our main theoretical results (namely Theorem \ref{theorem_generaldiffusion} and Theorem \ref{theorem_fourmoment}) in order to show density convergence can be challenging. In this  section, we demonstrate it is possible to directly apply our auxiliary results (namely the bounds on the Stein's equation in Section \ref{section_steinmethod_gendif}) and combine them with some classical result about convergence in Lyapunov-type norm to get density convergence.

We will consider the \emph{non-stationary} solution $X_t,t\geq 0$ to the stochastic differential equation from Section \ref{generaldiffusion} that is
\begin{align}
    dX_t=a(X_t)dt+\sqrt{b(X_t)\mathds{1}_{\{X_t\in(\ell,u) \} }}dB_t, \quad X_0=z_0. \label{e.7.1} 
\end{align} 
Recall that $X_t,t \geq 0$ is an ergodic process and converges in distribution to $\nu$ satisfying Assumption \ref{assum_generaldiffusion} as $t\to\infty$. $\mathcal{Y}$ is a random variable distributed as $\nu$. Moreover, the laws of both $X_t$ and $\mathcal{Y}$ are supported on the interval $(\ell,u)$.

Beside from Assumption \ref{assum_generaldiffusion}, we also suppose the following conditions hold.

\begin{customassumption}{H1}~
		\label{assumH1}
(Minorization condition) All compact sets in $\R$ are petite with respect to $X_{\eta n},n\geq 0$ for  some $\eta>0$.    More precisely,    
for any compact set $A$, 
  there is   a  sequence probability mass function 
$a_k$ ($a_k\ge 0$ and $\sum_{k=1}^\infty  a_k=1$) and a non-trivial 
measure $\phi$ on $\R$ so that  
\[
\sum_{k=1}^\infty  a_k P_{X_{\eta k}^x}(B)\ge \phi(B)\,, \qquad 
\forall \ x\in A\,,\  \forall \ B\in \mathcal{B}(\R)\,,
\]
where $X_{t}^x$ denotes the solution to 
\eqref{e.7.1} with initial condition $X_0=x$.  
\end{customassumption}

\begin{customassumption}{H2}~
		\label{assumH2}
  (Foster-Lyapunov condition)  There exists a   
  measurable function $V:\R\to\R_+$ such that $V(y)\to \infty $ when $y\to \infty$ 
  (e.g.  norm-like function) and constant $c,d>0$ such that 
\begin{align*}
    LV(y)\leq -cV(y)+d. 
\end{align*}
$L$ is the generator of $X_t,t\geq 0$ and acts on smooth function as $Lf(x)=a(x)f'(x)+\frac{1}{2}b(x)f''(x)$.  In this case, $V$ is a Lyapunov function associated with \eqref{generaldiffusion}. 
\end{customassumption}

\begin{customassumption}{H3}~
		\label{assumH3}
For every $t> 0$, the non-stationary solution $X_t$ admits a density on $\R$ which is denoted by $p_{X_t}$. Sometimes,  we also denote it by $p_{X_t^{z_0}}$ to specify the 
dependence on the initial condition $X_0=z_0$. 
\end{customassumption}

For the next assumption, let us fix $x\in (\ell,u)$ and let $\eta_x:\R\to \R$ be the function
\begin{align}
\label{def_etafunction}
    \eta_x(y):=2a(y)g_x(y)+1_{\{y\in(\ell,u)  \}}b(y)g_x'(y)
\end{align}
where $g_x(y)$ is the solution to the Stein's equation 
\eqref{equation_stein}   in Lemma \ref{lemma_xinlandu_gendif}. 

\begin{customassumption}{H4}~
		\label{assumH4}
There is a function $C_1:(\ell,u)\to (0,\infty)$ such that with respective to the Lyapunov function in Assumption   \ref{assumH3}, we have
\begin{align*}
    \eta_x(y)\leq C_1(x)V(y), \quad \forall y\in \R. 
\end{align*}
\end{customassumption}

\begin{theorem}
    \label{theorem_app_convergenceofgeneraldiffusion}
    Under Assumptions \ref{assum_generaldiffusion} in Section \ref{section_generaldiffusion} and \ref{assumH1}, \ref{assumH3}, \ref{assumH2}, \ref{assumH4} stated above, the densities of $X_t$ will converge to the density $p$ of the invariant probability measure $\nu$. In particular, we have exponential convergence in densities: there exists a constant $c>0$ and a function $C:(\ell,u)\to (0,\infty)$ such that for every $x\in (\ell,u)$,  
    \begin{align*}
       \abs{ p_{X_t^{z_0}}(x)- p(x)}\leq C(x)V(z_0)e^{-ct}. 
    \end{align*}
\end{theorem}

\begin{proof}
  According to \cite[Theorem 6.1]{meyn1993stability}, Assumption \ref{assumH1} and the Lyapunov function in Assumption \ref{assumH2} imply there exist constants $C_2,C_3>0$ such that for $\mathcal{Y}\sim \nu$, 
    \begin{align}
    \label{equation_Vnorm}
        \sup_{\abs{f}\leq V} \abs{\E{f\brac{X_t^{z_0} }-f(\mathcal{Y}) } }\leq C_2 V(z_0)e^{-C_3t}. 
    \end{align}

Based on the Stein's equation \eqref{equation_stein} and Lemma \ref{lemma_xinlandu_gendif}, we can write for a fixed $x\in (\ell,u)$, 
\begin{align*}
    \abs{ p_{X_t^{z_0}}(x)- p(x)}&=\abs{\E{\eta_x\brac{X_t^{z_0} } }}=\abs{\E{\eta_x\brac{X_t^{z_0}}-\eta_x(\mathcal{Y})} },
\end{align*}
noting that $\E{\eta_x(\mathcal{Y})}=0$. Combine this with Assumption \ref{assumH4} and \eqref{equation_Vnorm} to get the desired exponential convergence in density. 
\end{proof}

%

\begin{theorem}
\label{theorem_convergenceofpearsondiff}
     Consider  the Pearson diffusion  \eqref{sde_pearson} with $b(y)=b_2y^2+b_1y+b_0>0,\forall y\in \R$. Assume $Z_t,t\geq 0$ is its non-stationary solution and $\mathcal{Z}$ is a random variable distributed as its stationary measure. Then $Z_t$ and $\mathcal{Z}$ admit infinitely differentiable densities and in particular, there exists a constant $c>0$ such that for any $k\in \N$ and $x\in (\ell,u)$, 
  \begin{align*}
       \abs{ p^{(k)}_{Z_t^{z_0}}(x)- p^{(k)}_{\mathcal{Z}}(x)}\leq C(x)\sqrt{z_0^2+1}e^{-ct}. 
    \end{align*}
\end{theorem}

\begin{proof}
    We will verify similar assumptions to those appeared  in Theorem \ref{theorem_app_convergenceofgeneraldiffusion}.

       In the first step, we check Assumption \ref{assumH1} (Minorization condition).  $X_t,t\geq 0$ is an  It\^{o} process and therefore is also a Feller process per \cite[Theorem 3.1.9]{applebaum2009levy}. Moreover, since $X_t,t\geq 0$ converges to a Pearson distribution $\mu$ as $t\to\infty$, it is $\mu$-irreducible in the sense that for any Borel set $B$ for which $\pi(B)>0$, we have for sufficiently large $t$,
    \begin{align*}
        \int_0^\infty \mathbb{P}\brac{X^x_t\in B}dt>0,\forall x\in \R
    \end{align*}
 These facts and \cite[Theorem 3.4]{meyn1992stabilitypart1} imply the desired conclusion about compact sets and petite-ness. 

 In the second step, we check Assumption \ref{assumH2} (Foster-Lyapunov condition). Set $V(y):=\sqrt{y^2+1}$.  Then denoting $L$ the generator of \eqref{sde_pearson}, we have 
    \begin{align*}
        LV(y)=\frac{m-y}{2}\frac{y}{\sqrt{y^2+1}}+\frac{1}{2}(b_2y^2+b_1y+b_0)\frac{1}{\brac{\sqrt{y^2+1}}^3}. 
    \end{align*}
It holds for some large enough constant $R_1$ and $\abs{y}\geq R_1$ that  
\begin{align*}
    \frac{m-y}{2}\frac{y}{\sqrt{y^2+1}}\leq  -\frac{y}{2}, \qquad \frac{1}{2}(b_2y^2+b_1y+b_0)\frac{1}{\brac{\sqrt{y^2+1}}^3}\leq 1.
\end{align*}
This implies 
\begin{align*}
    LV(y)\leq -\frac{y}{2}+(d+1)\,, 
\end{align*}
where $d=\sup_{y\in [-R_1,R_1]}  \abs{\eta(y)}$, and thus Assumption \ref{assumH2} is satisfied. 

Per \cite[Theorem 6.1]{meyn1993stability}, the results in the first and second steps imply there exist constants $C_2,C_3>0$ such that
    \begin{align}
    \label{equation_Vnorm_pearson}
        \sup_{\abs{f}\leq V} \abs{\E{f\brac{Z_t^{z_0} }-f(\mathcal{Z}) } }\leq C_2 V(z_0)e^{-C_3t}. 
    \end{align}

In the third step, we check that $Z_t$ and $\mathcal{Z}$ admit infinitely differentiable densities. Infinite differentiability of the latter is clear from \eqref{densityanalyticformula_pearson}. Regarding the former, our assumption that $b(y)>0,\forall y\in \R$ implies \eqref{sde_pearson} satisfies H\"{o}rmander's condition, so that per \cite[Theorem 2.3.3]{nualart2006malliavin}, $Z_t$ admits a density that is infinitely differentiable. 

In the fourth and also last step, we define 
\begin{align*}
    \eta_{k,x}=(m-y)g_{k,x}(y)+b(y)g_{k,x}'(y)
\end{align*}
 as the left hand side of Equation \eqref{equation_stein_pearson}. The assumption $b(y)>0,\forall y\in \R$ imply $b_2>0$, so that $b(y)=O(b_2y^2)$ as $\abs{y}\to \infty$. This combined with the bounds on $g_{k,x}(y),g_{k,x}'(y)$ in Lemma \ref{lemma_xinlandu} suggest that 
 \begin{align}
 \label{lyapunovbound}
    \abs{\eta_{k,x}(y)}\leq C(x)\sqrt{\abs{y}^2+1}=C(x)V(y) 
\end{align}
for some function $C:(\ell,u)\to (0,\infty)$. Hence per \eqref{equation_Vnorm_pearson} and \eqref{lyapunovbound}, we arrive at the desired exponential convergence of the derivatives of the densities of $Z_t$.  
\end{proof}

\begin{remark} The condition $b(x)>0$ for all $x\in \R$ is quite strong and can certainly be weakened. We assume this condition to largely simplify  the argument.
It is interesting to continue research along this line for general  multidimensional ergodic  diffusions  under for instance  the
H\"ormander conditions. 
\end{remark}


\appendix

\section{Malliavin calculus on Gaussian space}
\label{appendix_malliavin}
~\

In this appendix, we will give an overview of Malliavin calculus on Gaussian space. In fact, it is another example of a Full Markov Triple. We will introduce Malliavin calculus in a standard way following \cite{nualart2006malliavin} and only at the end relates it back to the framework of Markov diffusion generators in Section \ref{section_prelim}.

Let $\mathfrak{H}$ be a real separable Hilbert space with orthonormal basis $\{e_i:i\in \N \}$ and $\left\{ Z(h)
\colon h \in \mathfrak{H} \right\}$ an {isonormal Gaussian process} defined on some probability space $(\Omega,\mathcal{F},P)$ and
indexed on $\mathfrak{H}$, that is, a centered Gaussian family of random variables 
such that 
\begin{align*}\E{Z(h)Z(g)} = \left\langle h,g \right\rangle_{\mathfrak{H}}.\end{align*} 
The associated probability space $(\Omega,\sigma(Z),P)$ is called a {Gaussian space}. The {Wiener chaos} of order $n$, denoted by $\mathcal{W}_n$, is the closed linear span
of the random variables $\left\{ H_n(Z(h)) \colon h \in \mathfrak{H},\
\norm{h}_{\mathfrak{H}}=1 \right\}$ where $H_{n}$ is the $n$-th Hermite
polynomial. Wiener chaos are the building blocks of the Gaussian space as we have the orthogonal decomposition
\begin{align}
	\label{decomposition_hermite}
	L^2(\Omega,\sigma(Z),P)= \bigoplus_{n=0}^\infty \mathcal{W}_n. 
\end{align}
Next, we will introduce {multiple Wiener integrals} with respect to the Gaussian process $Z$. In order to do that, let us first explain what it means by symmetrization of tensor products. For an integer $q\geq 2$, the expansion of $f\in\mathfrak{H}^q$ is given by \\$\sum_{i_1,\ldots,i_n=1}^\infty a(i_1,\ldots,i_n)e_{i_1}\ldots e_{i_n}$. Then the symmetrization of $f$, denoted by $\tilde{f}$, is the element in $\mathfrak{H}^{\odot n}$ given by
\begin{align*}
	\tilde{f}= \frac{1}{n!}\sum_{\sigma\in S_n} \sum_{i_1,\ldots,i_n=1}^\infty a(i_1,\ldots,i_n)e_{i_{\sigma(1)}}\ldots e_{i_{\sigma(n)}}
\end{align*}  
where $S_n$ the symmetric group of order $n$. Now, let $\Lambda$ be the set of sequence $a=\brac{a_1,a_2,\ldots}$ such that only a finite number of $a_i$ are non-zero. Set $a!=\prod_{i=1}^\infty a_i!$ and $\abs{a}=\sum_{i=1}^\infty a_i$. For $a\in \Lambda$ such that $\abs{a}=n$, the multiple integral $I_n$ is defined via
\begin{align*}
I_n\brac{\operatorname{sym}\brac{\otimes_{i=1}^\infty e_i^{\otimes a_i}}}=\sqrt{a!}\prod_{i=1}^\infty H_{a_i}(Z(e_i))
\end{align*}
and $\operatorname{sym}\brac{\otimes_{i=1}^\infty e_i^{\otimes a_i}}=\frac{1}{n!}\sum_{\sigma\in S_n } \otimes_{k=1}^n
	 e_{j_{\sigma(k)} } \,$.
If all $a_i=0$ except $a_{i_1}, \cdots, a_{i_k}$,
then we denote $e_{j_1}= e_{i_1}, \cdots, e_{j_{a_{i_1}}}=e_{i_1}$, $e_{j_{a_{i_1}+1}}=e_{i_2}, 
\cdots,  e_{j_{a_{i_1}+a_{i_2}  }}=  e_{i_2},\cdots,
e_{j_{a_{i_1}+\cdots+  a_{i_{k-1}} +1 }}= e_{i_k}, \cdots,  e_{j_{n
  }}=e_{i_k}$.
In fact, $I_{n}$ is an isometry map between the space of
symmetric tensor products $\mathfrak{H}^{\odot n}$  equipped with the scaled norm
$\frac{1}{\sqrt{n!}}\norm{\cdot}_{\mathfrak{H}^{\otimes n}}$ and the
Wiener chaos of order $n$. This, combined with the decomposition \eqref{decomposition_hermite}, implies that any $F\in L^2(\Omega)$ can be expanded into an orthogonal sum of multiple Wiener integrals, that is
\begin{equation}
	\label{wienerdecompose} 
 F=\sum_{n=0}^\infty I_{n}(f_{n}),
\end{equation}
where $f_{n}\in \mathfrak{H}^{\odot n}$ are (uniquely determined) symmetric
functions and $I_{0}(f_{0})=\mathbb{E}\left(  F\right)$. 

Next, to extend the definition of multiple Wiener integrals to non-symmetric $f\in \frak{H}^{\otimes n}$, we simply define the action of $I_n$ on the basis elements of $\frak{H}^{\otimes n}$ as
\begin{align*}
    I_n\brac{\otimes_{i=1}^\infty e_i^{\otimes a_i}}	=I_n\brac{\operatorname{sym}\brac{\otimes_{i=1}^\infty e_i^{\otimes a_i}}}. 
\end{align*}

We are now ready to introduce several important linear operators via the Gaussian chaos decomposition. The {Ornstein-Uhlenbeck semigroup} $\{P_t:t\geq 0\}$ is a semigroups of contraction operators that acts on a random variable $F$ with chaos decomposition \eqref{wienerdecompose} via $P_t F:=-\sum_{n=0}^{\infty} e^{-nt}I_{n}(f_{n})$. Denote by $L$ the {infinitesimal generator} of our semigroups. The action of $L$ on a random variable $F$ of the form \eqref{wienerdecompose} and such
that $\sum_{n=1}^{\infty} n^{2}n! \norm{f_n}^2_{\mathfrak{H}^{\otimes n}}<\infty$ is given by
\begin{equation*}
	LF=-\sum_{n=1}^{\infty} nI_{n}(f_{n}).
\end{equation*}

Let $\mathbb{D}^{k ,p }$ for $p>1$ and $k \in \mathbb{R}$ be the closure of
the set of polynomial random variables with respect to the norm
\begin{equation*}
	\Vert F\Vert _{k , p} =\Vert (I -L) ^{\frac{k }{2}} F \Vert_{L^{p} (\Omega )},
\end{equation*}
where $I$ denotes the identity operator. The spaces $\mathbb{D}^{k ,p }$ are known as {Sobolev-Watanabe spaces}.  Then, for any $p\geq 1$, the {Malliavin derivative} $D$ with respect to the Gaussian process $Z$ is a closable and continuous operator from $\mathbb{D}^{1,p} $ into $L^p(\Omega,\mathfrak{H})$, such that for $r\in \N$ and $f_n\in \mathfrak{H}^{\otimes n}$,
\begin{align*}
	DI_n(f_n)=n I_{n-1}(f_n).
\end{align*}
Alternatively, for a smooth random variable of the form $F=g(Z(h_1),\ldots , Z(h_n))$, where $g$ is a smooth function with compact support, and $h_i \in \mathfrak{H}$, $1 \leq i \leq n$, the Malliavin deriavative $D$ takes the form
\begin{equation*}
	DF:=\sum_{i=1}^{n}\frac{\partial g}{\partial x_{i}}(Z(h_1), \ldots , Z(h_n)) h_{i}.
\end{equation*}
An important feature of the Malliavin derivative $D$ is the chain rule (\cite[Proposition 1.2.3]{nualart2006malliavin}). Moreover, it also has a Hilbert adjoint $\delta$ which is commonly known as the {divergence operator} or the {Skorokhod integral}.

At this point, let us state an important result that is \textit{hypercontractivity} of the Ornstein-Uhlenbeck semigroups $\{P_t:t\geq 0\}$. 
\begin{lemma}(\cite[Theorem 2.8.12]{nourdinpeccatibook})
\label{lemma_hypercontract}
    Assume $p>1,t>0$ and $q(t)=e^{2t}(p-1)+1$. If $F\in L^p(\Omega)$ then
\begin{align*}
	\norm{P_t F}_{q(t)}\leq \norm{F}_p. 
\end{align*}
In particular, this implies for $p>1$,  any two $L^p$ norms of a random variable $F$ inside a single Wiener chaos are equivalent, and consequently $F$ has moments of all orders.
\end{lemma}

A consequence of hypercontractivity on Wiener chaos (more precisely the famous Meyer's inequalities) is the following functional inequality (\cite[Page 78]{nualart2006malliavin}): there exists a constant $C>0$ such that for any $F\in \mathbb{D}^{2,2}$, 
\begin{align}
\label{functionalinequality}
\E{(-LF)^2}\leq C\brac{\E{F^2}+\E{\norm{DF}^2_\frak{H}}+\E{\norm{D^2F}^2_{\frak{H}^{\otimes 2}}} }. 
\end{align}

Finally, let us informally relate the content of this section to the previous one. Malliavin calculus on Wiener space is in fact a Full Markov triple $\brac{E,\nu,\Gamma}$ where $E=\Omega$, $\nu$ is the laws of the Gaussian process $Z$ and $ \Gamma(F,G):=\inner{DF,DG}_\frak{H}$  for $F,G\in \mathbb{D}^{1,2}$. The chain rule of $D$ naturally induces a chain rule on $\Gamma$. The Markov diffusion chaos in this context are the Wiener chaos. One can refer to the book \cite{bouleauhirsch2010dirichlet} for more details on this connection.

\bibliography{refs}

\end{document}